\documentclass[notitlepage,11pt,reqno]{amsart}
\usepackage{amssymb,nicefrac,cite,bm,upgreek}
\usepackage[mathscr]{eucal}
\usepackage{color,verbatim}
\usepackage{url}

\definecolor{dmagenta}{rgb}{.4,.1,.5}
\definecolor{dblue}{rgb}{.0,.0,.6}
\definecolor{dred}{rgb}{.5,.0,.0}
\definecolor{dgreen}{rgb}{.0,.4,.0}
\definecolor{Eeom}{rgb}{.0,.0,.5}

\usepackage[margin=1in]{geometry}

\numberwithin{equation}{section}

\newtheorem{theorem}{Theorem}[section]
\newtheorem{lemma}{Lemma}[section]
\newtheorem{proposition}{Proposition}[section]
\newtheorem{corollary}{Corollary}[section]

\theoremstyle{definition}
\newtheorem{definition}{Definition}[section]
\newtheorem{assumption}{Assumption}[section]

\newtheorem{example}{Example}[section]

\newtheorem{property}{Property}[section]

\theoremstyle{remark}
\newtheorem{remark}{Remark}[section]

\newcommand{\df}{:=}
\DeclareMathOperator{\Exp}{\mathbb{E}}
\DeclareMathOperator{\Prob}{\mathbb{P}}
\newcommand{\D}{\mathrm{d}}
\newcommand{\E}{\mathrm{e}}
\newcommand{\R}{\mathbb{R}}
\newcommand{\RR}{\mathbb{R}}
\newcommand{\Rd}{\mathbb{R}^d}
\newcommand{\NN}{\mathbb{N}}
\newcommand{\ZZ}{\mathbb{Z}}
\newcommand{\Act}{\mathbb{U}}
\newcommand{\Uadm}{\mathfrak{U}}
\newcommand{\Usm}{\mathfrak{U}_{\mathrm{SM}}}
\newcommand{\Ussm}{\mathfrak{U}_{\mathrm{SSM}}}
\newcommand{\bUadm}{\bm{\mathfrak{U}}}
\newcommand{\bUsm}{\bm{\mathfrak{U}}_{\mathrm{SM}}}
\newcommand{\bUssm}{\bm{\mathfrak{U}}_{\mathrm{SSM}}}

\newcommand{\calA}{\mathcal{A}}   
\newcommand{\cA}{\mathcal{A}}     
\newcommand{\calB}{\mathcal{B}}   
\newcommand{\cB}{\mathcal{B}}     
\newcommand{\calC}{\mathcal{C}}   
\newcommand{\Cc}{\mathcal{C}}     

\newcommand{\dd}{d_{\mathrm P}}    
\newcommand{\eD}{\mathfrak{D}}    
\newcommand{\sF}{\mathfrak{F}}    
\newcommand{\eG}{{\mathscr{G}}}   
\newcommand{\cG}{{\mathcal{G}}}   
\newcommand{\eH}{{\mathscr{H}}}   
\newcommand{\calK}{\mathcal{K}}   
\newcommand{\sK}{\mathscr{K}}     
\newcommand{\cK}{\mathcal{K}}     
\newcommand{\Law}{{\mathscr{L}}}  
\newcommand{\calM}{\mathcal{M}}   
\newcommand{\sM}{\mathscr{M}}     
\newcommand{\calP}{\mathcal{P}}   
\newcommand{\Pm}{\mathcal{P}}     
\newcommand{\cT}{\mathcal{T}}     
\newcommand{\calV}{\mathcal{V}}   
\newcommand{\Lyap}{\mathcal{V}}   
\newcommand{\calX}{\mathcal{X}}   
\DeclareMathOperator{\Lip}{\mathrm{Lip}}

\newcommand{\calCl}{\mathcal{C}_{\mathrm{loc}}}
\newcommand{\Sob}{\mathscr{W}}
\newcommand{\Sobl}{\mathscr{W}_{\mathrm{loc}}}

\newcommand{\Lg}{L}
\newcommand{\grad}{\nabla}
\newcommand{\imeas}{\mu}

\newcommand{\Ind}{\mathbb{I}}        
\newcommand{\transp}{^{\mathsf{T}}}  

\newcommand{\eps}{\varepsilon}

\newcommand{\del}{\delta}

\newcommand{\bvnorm}[1]{[\kern-0.45ex[\kern0.1ex #1 \kern0.1ex]\kern-0.45ex]}

\newcommand{\tc}{{\Breve\uptau}}
\newcommand{\cast}{\circledast}
\newcommand{\order}{{\mathscr{O}}} 
\newcommand{\sorder}{{\mathfrak{o}}} 

\newcommand{\abs}[1]{\lvert#1\rvert}
\newcommand{\norm}[1]{\lVert#1\rVert}
\newcommand{\babs}[1]{\bigl\lvert#1\bigr\rvert}

\newcommand{\babss}[1]{\biggl\lvert#1\biggr\rvert}
\newcommand{\bnorm}[1]{\bigl\lVert#1\bigr\rVert}

\DeclareMathOperator*{\esssup}{ess\,sup}

\DeclareMathOperator{\trace}{trace}

\DeclareMathOperator{\conv}{conv}
\let\oldtocsection=\tocsection
\let\oldtocsubsection=\tocsubsection
\let\oldtocsubsubsection=\tocsubsubsection
\renewcommand{\tocsection}[2]{\hspace{0em}\oldtocsection{#1}{#2}}
\renewcommand{\tocsubsection}[2]{\hspace{1em}\oldtocsubsection{#1}{#2}}
\renewcommand{\tocsubsubsection}[2]{\hspace{2em}\oldtocsubsubsection{#1}{#2}}

\begin{document}

\title[MFG with ergodic costs]
{On Solutions of Mean Field Games with Ergodic cost}

\author{Ari Arapostathis}
\address{Department of Electrical and Computer Engineering,
The University of Texas at Austin,
1 University Station, Austin, TX 78712}
\email{ari@ece.utexas.edu}
\author{Anup Biswas}
\address{Department of Mathematics,
Indian Institute of Science Education and Research, 
Dr. Homi Bhabha Road, Pune 411008, India.}
\email{anup@iiserpune.ac.in}
\author{Johnson Carroll}
\address{Faculty of Engineering and the Built Environment,
University of Johannesburg, Johannesburg, South Africa}
\email{jcarroll@uj.ac.za}

\date{\today}

\begin{abstract}
A general class of mean field games are considered where the governing dynamics
are controlled diffusions in $\Rd$.
The optimization criterion is the long time average of a running cost function.
Under various sets of hypotheses, we establish the existence of
mean field game solutions.
We also study the long time
behavior of the mean field game solutions associated with the finite horizon
problem, and under the assumption of geometric ergodicity for the dynamics,
we show that these converge
to the ergodic mean field game solution as the horizon tends to infinity.
Lastly, we study the associated $N$-player games, show existence of Nash equilibria,
and establish the convergence of the solutions associated to Nash equilibria
of the game to a mean field game solution as $N\to\infty$.
\end{abstract}



\maketitle

\tableofcontents
\section{Introduction}
Mean field games (MFG) were introduced by J.\ M.\ Lasry and P.\ L.\ Lions
\cite{lasry-lions-06-1, lasry-lions-06-2, lasry-lions}, and independently,
by Huang, Malham\'{e} and Caines \cite{Huang-06}. Mean field games are the 
limiting models for symmetric, non-zero sum, non-cooperative $N$-player games
with the interaction between the players being of mean field type.
In view of the theory of McKean-Vlasov limits and propagation of chaos for 
uncontrolled weakly interacting particle systems \cite{sznitman-91}, one may
expect to obtain convergence result
for $N$-player game Nash equilibria, at least under some symmetry conditions.
With this heuristic in mind, Lasry and Lions
introduced the field of mean field games.
Recently, rigorous results have been established
for finite horizon control problems \cite{lacker, fischer},
for mean field games with ergodic
cost \cite{feleqi}, and for discrete time Markov processes with ergodic cost
\cite{biswas}.
On the other hand, it is also known that
one can construct $\varepsilon$-Nash equilibria for $N$-player games from 
mean field game solutions. See
for example
\cite{Huang-06, carmona-delarue, carmona-lacker, Kolokoltsov-11a, lacker-15}.
Mean field games have seen a wide variety of applications, and have been
studied extensively during the last decade
using both analytic and probabilistic techniques.
We refer to the surveys in
\cite{cardaliaguet, gomes-saude, gu-lasry-lions, ben-fre-yam} for recent
developments in the area of mean field games.

In this paper, we model the controlled dynamics
of the $i$th player, $i=1, \ldots, N$, by the It\^o equation
$$dX^i_t \;=\; b(X^i_t, U^i_t)\, \D{t} + \upsigma(X^i_t)\, \D{W^i_t},$$
where $\{W^i\}_{\{1\leq i\leq N\}}$
is a collection of independent Wiener processes in $\Rd$ and $U^i$
is an \emph{admissible} control, taking values in a compact metric
space $\Act$, adapted to the filtration generated by $W^i$.
Thus the players do not have access to the full state vector for purposes
of control.
Such strategies are referred to as \emph{narrow strategies} \cite{fischer}. 
The \emph{running cost} is given by a continuous function
$r:\Rd\times\Act\times\calP(\Rd)\to\R_+$.
The goal of the $i$-th player is to minimize the (ergodic) criterion
$$J^i(\bm{U})\; =\; \limsup_{T\to\infty}\,
\frac{1}{T}\, \Exp\Bigl[\int_0^T r(X^i_t, U^i_t, \mu^N_t)\, \D{t}\Bigr],
\quad \mu^N_t\df \frac{1}{N}\sum_{j=1}^N \delta_{X^j_t},$$
where $\bm{U}=(U^1, \ldots, U^N)$.
We note that the running cost function $r$ may depend upon the
empirical distribution $\mu^N$ of the private states of the players.
Since each player's objective depend on the action of others
we naturally look for Nash equilibria.

If the number of players $N$ is very large, the contribution of the
$i$-th player in the empirical distribution 
$\mu^N$ is negligible, and therefore $\mu^N$ may as well be treated
fixed for player $i$. This heuristic argument
leads to the mean field game formulation which can be described as follows:
\begin{itemize}
\item[(a)]
For a fixed element
$\eta\in\calC([0, \infty), \calP(\Rd))$ solve the optimal control problem,
\begin{equation}\label{Intro-1}
\begin{split}
\text{minimize} &\quad  \limsup_{T\to\infty}\,
\frac{1}{T}\, \Exp\Bigl[\int_0^T r(X_t, U_t, \eta_t)\, \D{t}\Bigr],
\\
\text{subject to} &\quad  \D X_t\;=\;
b(X_t, U_t)\, \D{t} + \upsigma(X_t)\, \D{W_t}\,,\quad \text{law of~} X_0=\eta(0)\,.
\end{split}
\end{equation}

\item[(b)] Find an optimal control $U^*$ for the above control problem, and let
$\eta^*$ denote the law of the
state dynamics under the optimal control $U^*$.

\item[(c)] Find a fixed point of the map $\eta\mapsto \eta^*$.
\end{itemize}
The above model can be interpreted as follows:
there is a single representative agent whose reward function is
effected by an environment distribution (coming from the large
number of agents),
and the state process of the representative can not influence the environment
while solving its own optimization problem.
Moreover, since all agents have identical dynamics
and the same objective function,
the distribution of the state process of the representative agent should
agree with the environment distribution.
We observe that the above problem is not a \emph{typical} optimal control problem.
The cost function here is not being optimized over all possible pairs $(X, \eta)$
where $X_t$ has distribution $\eta_t$ and $X$ satisfies the
dynamics in \eqref{Intro-1}.
This later class of problems
are known as mean field type control problems \cite{ben-fre-yam}. 

There are three major issues of interest in mean field games,
(1) existence and uniqueness of solutions of MFG,
(2) long time behavior of the finite horizon MFG, and
(3) establishing rigorous connection of $N$-player games
with MFG. The topic in (3) can also be divided in two parts:
(3a) convergence of a $N$-player game Nash equilibria
to a MFG solution, and (3b) construction of $\varepsilon$-Nash equilibria
for the $N$-player game from a MFG solution.
The chief goal of this article is to answer the questions in (1), (2) and (3a)
for the class of models we consider.

During the last decade many papers have been devoted to the study of the topics
above.
Existence of mean field game solutions with ergodic cost for a compact state space 
is studied in \cite{lasry-lions, feleqi}.
For existence of mean
field game solutions for finite horizon control problems we refer the reader to
\cite{carmona-lacker, lacker, carmona-delarue, bensou-13}.
These papers also allow the drift to depend on the environment distribution $\eta$.
The existence problem for finite state processes
is addressed in \cite{gomes-mohr-souza, gomes-2, gueant-15},
and  a more general class of discrete
time Markov processes to study the existence result when the cost is ergodic
is considered in \cite{biswas}.
Linear-Quadratic mean field games with ergodic costs are
considered in \cite{bardi-priuli},  and existence results are established.
However there is not much improvement as far as uniqueness in concerned.
A $L^2$ type monotonicity condition (or a variant of it) is generally used to
claim uniqueness of the mean field game solution
(see \cite{lasry-lions, cardaliaguet}).

In Section~\ref{S-existence} we study the existence of MFG solutions.
We show that existence of MFG solutions is 
related to the existence of
$(V, \Tilde\varrho, \mu)\in\calC^2(\Rd)\times\R_+\times\calP(\Rd)$ 
satisfying the following coupled equations
(see Theorems~\ref{T2.1} and \ref{T3.2})
\begin{align}
\min_{u\in\Act}\;\bigl[L^u V(x) + r(x,u, \mu)\bigr]
&\;=\;  L^v\, V + r\bigl(x,v(x), \mu\bigr) \;=\;
\Tilde\varrho\quad \text{a.\ e.}\; x\in\Rd\,,\label{Intro-2}\\
\int_{\Rd} L^v f(x)\, \mu(\D{x}) & \,=\, 0 \qquad \forall \; f\in\calC^2_c(\Rd).
\label{Intro-3}
\end{align}
Here $L^u$ (see \eqref{E2.3}) denotes the controlled extended generator
of the controlled diffusion in \eqref{Intro-1}.
As well known, \eqref{Intro-2} is the Hamilton--Jacobi--Bellman (HJB)
equation for an optimal ergodic control
problem with running cost function $(x, u)\mapsto r(x, u, \mu)$,
whereas \eqref{Intro-3} characterizes $\mu$ as the invariant
probability measure corresponding
to an optimal (stationary) Markov control $v$ of \eqref{Intro-2}.
We use convex analytic tools (see Section~\ref{S-convex}) and the
Kakutani--Fan--Glicksberg fixed point theorem to establish existence of a solution to
\eqref{Intro-2}--\eqref{Intro-3}.

One may also consider a finite horizon problem (say, with time horizon $T$)
for the mean field model.
In this situation the solution is again determined by two coupled equations,
where one equation depicts the evolution of transition density (or
transition probability)
and the other one is the HJB for the finite horizon optimal control problem.
For a model with  a compact state space,
it is shown in \cite{Cardaliaguet12, CLLP-13} that, as $T\to\infty$,
the solution of the finite horizon control problem
tends to the solution of \eqref{Intro-2}--\eqref{Intro-3} under
suitable normalization.
In Section~\ref{S-RVI} we study the analogous problem for our model.
We compensate for the non-compactness of the state space by imposing
a Lyapunov stability hypothesis to control behavior at infinity.
We show that as the horizon $T\to\infty$,
the law of the process for the finite horizon MFG tends to a stationary
law with marginals $\mu$
(see \eqref{Intro-3}), and the value function
of the finite horizon problem, suitably normalized, tends to $V$ in \eqref{Intro-2},
uniformly over compact sets
(see Theorems~\ref{Tn4.3} and~\ref{Tn4.4} for details).

Next we discuss topic (3). As stated earlier there are several papers
in which construction of approximate Nash equilibria is done using a MFG solution.
In fact, a similar construction is also possible in our set up as well.
However, the opposite direction is probably more natural and interesting
\cite[Remark~3.9]{cardaliaguet}.
Existence of Nash equilibria for $N$-player games with ergodic cost,
and convergence to them is studied in \cite{feleqi}, for a model with
compact state space and a running cost function that
has a special separable structure.
Recently, \cite{lacker-15, fischer} have addressed the same question
(assuming existence of approximate Nash equilibria) for a general class of
finite horizon control problems where the drift $b$ and the
diffusion matrix $\upsigma$ may also depend on $\mu^N$.
The approach in these papers uses the martingale formulation,
and the method of weak
convergence.
Under suitable conditions, and for finite horizon control problems,
it is established in \cite{lacker-15, fischer} that a
certain type of \emph{averages} of approximate Nash equilibria are tight and
their subsequential limits are a solution for the MFG problem.
The results in Section~\ref{S-Nplayer} are quite similar to that of \cite{feleqi}
(compare Theorem~\ref{T5.2} with \cite[Theorem~2]{feleqi}).
Since the state space is not compact, we work under
the assumption of geometric stability.
Also we impose fairly general hypotheses on the running cost function,
which are satisfied by a large class of
functions (Assumption~\ref{A5.3}).
For the analysis, we have borrowed several results from \cite{ari-bor-ghosh}.
The representation formula of the
ergodic value function
is shown to be quite useful in proving Theorem~\ref{T5.2}.
Let us also mention that the convergence results for 
Nash equilibria we present
are somewhat stronger than those obtained in \cite{lacker-15, fischer}.
In fact, we show that the maximum distance between the invariant measures in
the  Nash equilibrium tuple tends to $0$ as the
number of players increases to infinity (see Theorem~5.2\,(b)).

Summarizing our contributions in this article, we
\begin{itemize}
\item[--] establish existence of MFG solutions for a large class of mean field games;
\item[--] prove the convergence of the finite horizon MFG solution to
the stationary one, under the assumption of geometric stability;
\item[--] study the existence of Nash equilibria for $N$-player games and prove
that they converge to a MFG solution.
\end{itemize}

The organization of the paper is as follows:
In Section~\ref{S-existence} we introduce the model and the basic assumptions,
and state the main result (Theorem~\ref{T2.1}) on the
existence of MFG solutions.
Various convex analytic results
are in Section~\ref{S-convex},  where we also prove
the main results.
In Section~\ref{S-RVI}
we study the long time behavior of the finite horizon problem.
Finally, in Section~\ref{S-Nplayer} we show existence
of Nash equilibria for the $N$-player games and study
their convergence to MFG solutions.

\subsection{Notation}\label{S-notation}
The standard Euclidean norm in $\Rd$ is denoted by $\abs{\,\cdot\,}$.
The set of nonnegative real numbers is denoted by $\R_{+}$,
$\NN$ stands for the set of natural numbers, and $\Ind$ denotes
the indicator function.
The interior, closure, the boundary and the complement
of a set $A\subset\Rd$ are denoted
by $A^o$, $\overline{A}$, $\partial{A}$ and $A^{c}$, respectively. 
The open ball of radius $R$ around $0$ is denoted by $B_{R}$.
Given two real numbers $a$ and $b$, the minimum (maximum) is denoted by $a\wedge b$ 
($a\vee b$), respectively.
By $\delta_{x}$ we denote the Dirac mass at $x$.

For a continuous function $g\,\colon\,\Rd\to[1,\infty)$ 
we let $\order(g)$ denote the space of Borel measurable functions
$f\,\colon\,\Rd\to\RR$ satisfying
$\esssup_{x\in\Rd}\;\frac{\abs{f(x)}}{g(x)}<\infty$, and
by $\sorder(g)$ those functions satisfying
$\limsup_{R\to\infty}\;\esssup_{x\in B_{R}^{c}}\,\frac{\abs{f(x)}}{g(x)}\;=\;0$.
We also let 
$\Cc_{g}(\Rd)$ denote the Banach space of continuous functions under the norm
\begin{equation*}
\norm{f}_{g} \;\df\;
\sup_{x\in\Rd}\;\frac{\abs{f(x)}}{g(x)}\,.
\end{equation*}
For two nonnegative functions $f$ and $g$, we use the notation $f\sim g$
to indicate that $f\in\order(1+g)$ and $g\in\order(1+f)$.

We denote by $L^{p}_{\mathrm{loc}}(\Rd)$, $p\ge 1$, the set of
real-valued functions
that are locally $p$-integrable and by
$\Sobl^{k,p}(\Rd)$ the set of functions in $L^{p}_{\mathrm{loc}}(\Rd)$
whose $i$-th weak derivatives, $i=1,\dotsc,k$, are in
$L^{p}_{\mathrm{loc}}(\Rd)$.
The set of all bounded continuous functions is denoted by $\calC_{b}(\Rd)$.
By $\calCl^{k,\alpha}(\Rd)$ we denote the set of functions that are
$k$-times continuously differentiable and whose $k$-th derivatives are locally
H\"{o}lder continuous with exponent $\alpha$.
We define $\calC^{k}_{b}(\Rd)$, $k\ge 0$, as the set of functions
whose $i$-th derivatives, $i=1,\dotsc,k$, are continuous and bounded
in $\Rd$ and denote by
$\calC^{k}_{c}(\Rd)$ the subset of $\calC^{k}_{b}(\Rd)$ with compact support.
Given any Polish space $\calX$, we denote by $\cB(\calX)$ its Borel $\sigma$-field,
by $\calP(\calX)$ the set of
probability measures on $\cB(\calX)$ and $\calM(\calX)$ the set of all
bounded signed measures on $\cB(\calX)$.
For $\nu\in\calP(\calX)$ and a Borel measurable map $f\colon\calX\to\R$,
we often use the abbreviated notation
$$\nu(f)\;\df\; \int_{\calX} f\,\D{\nu}\,.$$
The  space of all continuous maps from
$[0, \infty)$ to $\calX$ is denoted by $\calC([0, \infty), \calX)$.
The law of a random variable $X$ is denoted by $\Law(X)$.
For presentation purposes the time variable appears
as a subscript for the diffusion process $X$. 
Also $\kappa_{1},\kappa_{2},\dotsc$ and $C_{1},C_{2},\dotsc$
are used as generic constants whose values might vary from place to place.

\section{Existence of solutions to MFG}\label{S-existence}
\subsection{Controlled diffusions}
The dynamics are modeled by a
controlled diffusion process $X = \{X_{t},\;t\ge0\}$
taking values in the $d$-dimensional Euclidean space $\Rd$, and
governed by the It\^o stochastic differential equation
\begin{equation}\label{E-sde}
\D{X}_{t} \;=\;b(X_{t},U_{t})\,\D{t} + \upsigma(X_{t})\,\D{W}_{t}\,.
\end{equation}
All random processes in \eqref{E-sde} live in a complete
probability space $(\Omega,\sF,\Prob)$.
The process $W$ is a $d$-dimensional standard Wiener process independent
of the initial condition $X_{0}$.
The control process $U$ takes values in a compact metric space
$(\Act, d_\Act)$, and
$U_{t}(\omega)$ is jointly measurable in
$(t,\omega)\in[0,\infty)\times\Omega$.
Moreover, it is \emph{non-anticipative}:
for $s < t$, $W_{t} - W_{s}$ is independent of
\begin{equation*}
\sF_{s} \;\df\;\text{the completion of~} \sigma\{X_{0},U_{r},W_{r},\;r\le s\}
\text{~relative to~}(\sF,\Prob)\,.
\end{equation*}
Such a process $U$ is called an \emph{admissible control},
and we let $\Uadm$ denote the set of all admissible controls.

We impose the following standard assumptions on the drift $b$
and the diffusion matrix $\upsigma$
to guarantee existence and uniqueness of solutions to \eqref{E-sde}.
\begin{itemize}
\item[{(A1)}]
\emph{Local Lipschitz continuity:\/}
The functions
\begin{equation*}
b\;=\;\bigl[b^{1},\dotsc,b^{d}\bigr]\transp\,\colon\,\Rd \times\Act\to\Rd
\quad\text{and}\quad
\upsigma\;=\;\bigl[\upsigma^{ij}\bigr]\,\colon\,\Rd\to\R^{d\times d}
\end{equation*}
are locally Lipschitz in $x$ with a Lipschitz constant $C_{R}>0$ depending on
$R>0$.
In other words,
for all $x,y\in B_{R}$ and $u\in\Act$,
\begin{equation*}
\abs{b(x,u) - b(y,u)} + \norm{\upsigma(x) - \upsigma(y)}
\;\le\;C_{R}\,\abs{x-y}\qquad \forall\,x,y\in B_{R}\,.
\end{equation*}
We also assume that $b$ is continuous in $(x,u)$.
\smallskip
\item[{(A2)}]
\emph{Affine growth condition:\/}
$b$ and $\upsigma$ satisfy a global growth condition of the form
\begin{equation*}
\abs{b(x,u)}^{2}+ \norm{\upsigma(x)}^{2}\;\le\;C_{1}
\bigl(1 + \abs{x}^{2}\bigr) \qquad \forall (x,u)\in\Rd\times\Act\,,
\end{equation*}
where $\norm{\upsigma}^{2}\;\df\;
\mathrm{trace}\left(\upsigma\upsigma\transp\right)$.
\smallskip
\item[{(A3)}]
\emph{Local nondegeneracy:\/}
For each $R>0$, it holds that
\begin{equation*}
\sum_{i,j=1}^{d} a^{ij}(x)\xi_{i}\xi_{j}
\;\ge\;C^{-1}_{R}\abs{\xi}^{2} \qquad\forall x\in B_{R}\,,
\end{equation*}
for all $\xi=(\xi_{1},\dotsc,\xi_{d})\transp\in\Rd$,
where $a\df \frac{1}{2}\upsigma \upsigma\transp$.
\end{itemize}

In integral form, \eqref{E-sde} is written as
\begin{equation}\label{E2.2}
X_{t} \;=\;X_{0} + \int_{0}^{t} b(X_{s},U_{s})\,\D{s}
+ \int_{0}^{t} \upsigma(X_{s})\,\D{W}_{s}\,.
\end{equation}
The third term on the right hand side of \eqref{E2.2} is an It\^o
stochastic integral.
We say that a process $X=\{X_{t}(\omega)\}$ is a solution of \eqref{E-sde},
if it is $\sF_{t}$-adapted, continuous in $t$, defined for all
$\omega\in\Omega$ and $t\in[0,\infty)$, and satisfies \eqref{E2.2} for
all $t\in[0,\infty)$ a.s.
It is well known that under (A1)--(A3), for any admissible control
there exists a unique solution of \eqref{E-sde}
\cite[Theorem~2.2.4]{ari-bor-ghosh}.

We define the family of operators
$L^{u}\colon\calC^{2}(\Rd)\to\calC(\Rd)$,
where $u\in\Act$ plays the role of a parameter, by
\begin{equation}\label{E2.3}
L^{u} f(x) \;\df\; a^{ij}(x)\,\partial_{ij} f(x)
+ b(x,u)\cdot \grad f(x)\,,\quad(x,u)\in\Rd\times\Act\,.
\end{equation}
We refer to $L^{u}$ as the \emph{controlled extended generator} of
the diffusion.
In \eqref{E2.3} and elsewhere in this paper we have adopted
the notation
$\partial_{t}\df\tfrac{\partial}{\partial{t}}$,
$\partial_{i}\df\tfrac{\partial~}{\partial{x}_{i}}$, and
$\partial_{ij}\df\tfrac{\partial^{2}~}{\partial{x}_{i}\partial{x}_{j}}$.
We also use the standard summation rule that
repeated subscripts and superscripts are summed from $1$ through $d$.
In other words, the right hand side of \eqref{E2.3} stands for
\begin{equation*}
\sum_{i,j=1}^{d}
a^{ij}(x)\,\frac{\partial^{2}f}{\partial{x}_{i}\partial{x}_{j}}(x)
+\sum_{i=1}^{d} b^{i}(x,u)\, \frac{\partial f }{\partial{x}_{i}}(x)\,.
\end{equation*}
Of fundamental importance in the study of functionals of $X$ is
It\^o's formula.
For $f\in\calC^{2}(\Rd)$ and with $L^{u}$ as defined in \eqref{E2.3},
it holds that
\begin{equation*}
f(X_{t}) \;=\;f(X_{0}) + \int_{0}^{t}L^{U_{s}} f(X_{s})\,\D{s}
+ M_{t}\,,\quad\text{a.s.},
\end{equation*}
where
\begin{equation*}
M_{t} \;\df\;\int_{0}^{t}\bigl\langle\nabla f(X_{s}),
\upsigma(X_{s})\,\D{W}_{s}\bigr\rangle
\end{equation*}
is a local martingale.

Recall that a control is called \emph{Markov} if
$U_{t} = v(t,X_{t})$ for a measurable map $v\colon\R_{+}\times\R^{d}\to \Act$,
and it is called \emph{stationary Markov} if $v$ does not depend on
$t$, i.e., $v\colon\R^{d}\to \Act$.
Correspondingly \eqref{E-sde}
is said to have a \emph{strong solution}
if given a Wiener process $(W_{t},\sF_{t})$
on a complete probability space $(\Omega,\sF,\Prob)$, there
exists a process $X$ on $(\Omega,\sF,\Prob)$, with $X_{0}$
as specified by the initial condition,
which is continuous,
$\sF_{t}$-adapted, and satisfies \eqref{E2.2} for all $t$ a.s.
A strong solution is called \emph{unique},
if any two such solutions $X$ and $X'$ agree
$\Prob$-a.s., when viewed as elements of $\calC\bigl([0,\infty),\R^{d}\bigr)$.
It is well known that under Assumptions (A1)--(A3),
for any Markov control $v$,
\eqref{E-sde} has a unique strong solution \cite{Gyongy-96}.

Let $\Usm$ denote the set of stationary Markov controls.
Under $v\in\Usm$, the process $X$ is strong Markov,
and we denote its transition function by $P_{t}^{v}(x,\cdot\,)$.
It also follows from the work of \cite{Bogachev-01,Stannat-99} that under
$v\in\Usm$, the transition probabilities of $X$
have densities which are locally H\"older continuous.
Thus $L^{v}$ defined by
\begin{equation*}
L^{v} f(x) \;\df\; a^{ij}(x)\,\partial_{ij} f(x)
+ b^{i} \bigl(x,v(x)\bigr)\,\partial_{i} f(x)\,,\quad v\in\Usm\,,
\end{equation*}
for $f\in\calC^{2}(\R^{d})$,
is the generator of a strongly-continuous
semigroup on $\calC_{b}(\R^{d})$, which is strong Feller.
We let $\Prob_{x}^{v}$ denote the probability measure and
$\Exp^{v}_{x}$ the expectation operator on the canonical space of the
process under the control $v\in\Usm$, conditioned on the
process $X$ starting from $x\in\R^{d}$ at $t=0$.
The expectation operator $\Exp^{U}_{x}$ is of course
also well defined for $U\in\Uadm$.

Recall that control $v\in\Usm$ is called \emph{stable}
if the associated diffusion is positive recurrent.
We denote the set of such controls by $\Ussm$,
and let $\mu_{v}$ denote the unique invariant probability
measure on $\R^{d}$ for the diffusion under the control $v\in\Ussm$.
Recall that $v\in\Ussm$ if and only if there exists an inf-compact function
$\calV\in\calC^{2}(\R^{d})$, a bounded domain $D\subset\R^{d}$, and
a constant $\varepsilon>0$ satisfying
\begin{equation*}
L^{v}\calV(x) \;\le\;-\varepsilon
\qquad\forall x\in D^{c}\,.
\end{equation*}
We denote by $\uptau(A)$ the \emph{first exit time} of a process
$\{X_{t}\,,\;t\in\R_{+}\}$ from a set $A\subset\R^{d}$, defined by
\begin{equation*}
\uptau(A) \;\df\;\inf\;\{t>0\,\colon\,X_{t}\not\in A\}\,.
\end{equation*}
The open ball of radius $R$ in $\R^{d}$, centered at the origin,
is denoted by $B_{R}$, and we let $\uptau_{R}\;\df\;\uptau(B_{R})$,
and $\tc_{R}\df \uptau(B^{c}_{R})$.

\subsection{Topologies on $\calP(\Rd)$}
We endow the space $\calP(\Rd)$ with the Prokhorov metric $\dd$ that renders
$\calP(\Rd)$ the topology of weak convergence.
As is well known this is defined by
\begin{equation}\label{Edd}
\dd(\mu_{1}, \mu_{2})\;\df\; \inf\bigl\{\eps :\, \eps\ge 0,
\text{~such that for all Borel~} F\subset\Rd, \,
\mu_{1}(F)\le \mu_{2}(F^\eps) + \eps\bigr\}\,.
\end{equation}
It is well known that $(\calP(\Rd),\dd)$ is a Polish space and $\dd(\mu_n, \mu)\to 0$
as $n\to\infty$ if and only if, for every $f\in\calC_b(\Rd)$, we have
$\mu_n(f)\to\mu(f)$ as $n\to\infty$.
By $\calP_p(\Rd)$, $p\ge 1$, we denote the subset of $\calP (\Rd)$
containing all probability measures $\mu$ with the property that
$\int_{\Rd}\abs{x}^p\mu(\D{x})<\infty$.
The Wasserstein metric on $\calP_p(\Rd)$ is defined as follows:
\begin{equation}\label{EDp}
\mathfrak{D}_p(\mu_{1}, \mu_{2})\df \inf\;\biggl\{\int_{\Rd\times\Rd}
\abs{x-y}^p \nu(\D{x}, \D{y})\, :\, \nu\in\calP(\Rd\times\Rd)
\text{~has marginals~}
\mu_{1}, \, \mu_{2}\biggr\}^{\nicefrac{1}{p}}.
\end{equation}
It is well known that
$(\calP_p(\Rd), \eD_p)$, $p\ge 1$, is a Polish space.
The topology generated by $\eD_p$ on $\calP_p(\Rd)$ is finer than the one induced
by $\dd$.
In fact, we have
the following assertion \cite[Theorem~7.12]{villani}.
\begin{proposition}\label{P-villani}
Let $\{\mu_n\}_{n\in\NN}$ be a sequence of probability measures in $\calP_p(\Rd)$,
and let $\mu\in\calP(\Rd)$. Then, the following statements are equivalent:
\begin{enumerate}
\item $\eD_{p}(\mu_n, \mu)\to 0$, as $n\to\infty$.
\smallskip

\item $\dd(\mu_n, \mu)\to 0$ as $n\to\infty$, and 
$$\int_{\Rd}\abs{x}^p \mu_n(\D{x})\;\xrightarrow[n\to\infty]{}\;
\int_{\Rd}\abs{x}^p \mu(\D{x})\,.$$

\item $\dd(\mu_n, \mu)\to 0$ as $n\to\infty$,
and  $\{\mu_n\}$ satisfies the following condition: 
$$\lim_{R\to\infty} \limsup_{n\to\infty}\int_{B^c_R(0)} \abs{x}^p \, \D{\mu_n}
\; =\; 0\,.$$
\end{enumerate}
\end{proposition}

Therefore, a set $\calK$ which is compact in $(\calP_p(\Rd), \eD_p)$
is also compact in $(\calP_p(\Rd),\dd)$.
In the rest of the paper $\calP_p(\Rd)$ and $\calP(\Rd)$ are always meant
to be metric spaces endowed with the metrics $\eD_p$ and $\dd$,
respectively, unless mentioned otherwise.

\subsection{The ergodic control problem}
In this paper we consider dynamics as in \eqref{E-sde} and associated
running cost functions belonging to one of the three classes described
in Assumption~\ref{A2.1} below.
We use the notation $r_{\mu}(x,u)\df r(x,u,\mu)$.
Also we write $r_{\mu}\in\sorder(h)$ for $h:\Rd\times\Act\to\RR_{+}$,
provided
\begin{equation*}
\limsup_{\abs{x}\to\infty}\;\sup_{u\in\Act}\;
\frac{\abs{r_{\mu}(x,u)}}{1+h(x,u)}\;=\;0\,.
\end{equation*}

\begin{assumption}\label{A2.1}
One of the following conditions holds:
\begin{itemize}
\item[{(C1)}]
The running cost
$r:\Rd\times\Act\times\calP(\Rd)\to \R_+$ is continuous,
 and for each $\mu\in\calP(\Rd),\, r_{\mu}(\cdot\,, \cdot)$ is locally 
Lipschitz in its first argument uniformly with respect to the second.
Moreover, for any compact subset $\calK$ of $\calP(\Rd)$
there exists
$\theta>0$ such that
\begin{equation}\label{A2.1a}
\liminf_{\abs{x}\to\infty}\;\inf_{u\in\Act}\;
\frac{r(x,u,\mu)}{r(x,u,\mu')}\;>\;\theta\qquad \forall\,\mu,\mu'\in\cK\,,
\end{equation}
and
\begin{equation}\label{A2.1b}
\inf_{(u,\mu)\,\in\,\Act\times\cK}\, r(x, u,\mu)\;\xrightarrow[\abs{x}\to\infty]{}\;
\infty\,.
\end{equation}

\item[{(C2)}]
The running cost takes the form
$r_{\mu}(x,u)= \mathring r(x,u) + F(x,\mu)$,
where $F: \Rd\times\calP_p(\Rd)\to \R_+$ is a continuous function,
satisfying
$$F(x,\mu) \le \kappa_{0}\,\biggl(1+\abs{x}^p+\int_{\Rd}\abs{x}^p\,\mu(\D{x})\biggr)
\qquad \forall\,(x,\mu)\in\Rd\times\calP_p(\Rd)\,,
$$
for some constant $\kappa_{0}$ and $p\ge1$.
Also, $\mathring r:\Rd\times\Act\to\R_+$ is continuous and locally Lipschitz
in $x$ uniformly in $u\in\Act$,
and satisfies
$$\min_{u\in\Act}\,\frac{\mathring r(x,u)}{1+\abs{x}^p}
\;\xrightarrow[\abs{x}\to\infty]{}\;\infty\,.$$

\item[{(C3)}]
The running cost $r:\Rd\times\Act\times\calP(\Rd)\to \R_+$ is continuous, and 
$x\mapsto r(x, u, \mu)$ is locally Lipschitz
uniformly in $u\in\Act$ and $\mu$ in  compact subsets of $\calP(\Rd)$.
Also
\begin{itemize}
\item[(C3a)]
There exist inf-compact functions $\calV\in\Cc^{2}(\Rd)$
and $h\in\Cc(\Rd\times\Act)$ such that 
for some positive constants $c_{1}$ and $c_{2}$ we have 
\begin{equation}\label{US}
L^u\calV(x) \; \le \; c_{1} - c_{2} h(x,u)
\qquad \forall (x,u)\in\Rd\times\Act\,,
\end{equation}
\item[(C3b)]
For any compact $\calK\subset\calP(\Rd)$ it holds that
$$\sup_{\mu\in\calK}\; r_{\mu}\in \sorder(h)\,.$$
\end{itemize}
\end{itemize}
\end{assumption}

A typical example of $F$ in (C2) is $F(x,\mu)=\int \abs{x-y}^p\, \mu(\D{y})$.
Also $r(x,u, \mu) = r_{1}(x) + r_{2}(x,u, \mu)$, with 
$r_{2}\in\calC_b\bigl(\Rd\times\Act\times\calP(\Rd)\bigr)$, $r_{1}\in\calCl^{0, 1}(\Rd)$,
and $\lim_{\abs{x}\to\infty} r_{1}(x)=+\infty$ is an example of 
running cost satisfying (C1).
  
The running costs in (C1) and (C2) satisfy the condition of \emph{near monotonicity}
\cite{ari-bor-ghosh}, while \eqref{US} implies
that the controlled diffusion is uniformly stable.
In \cite{lasry-lions, feleqi} cost functions satisfying (C2)
on a compact state space are considered. 
But in the current scenario the state space is $\Rd$ which is not compact.
The cost functions in (C3) are allowed to take more general forms.
Since $\calV$ and $h$ are bounded from below (being inf-compact),
without loss of generality we assume that $\calV\ge 1$, and $h\ge 0$.

In general, $\Act$ may not be a convex set.
It is therefore often useful to enlarge the control set to $\calP(\Act)$.
To do so, for  $v\in\calP(\Act)$ we replace the drift and the  running cost with
\begin{equation}\label{relax}
\Bar{b}(x,v)\;\df\;\int_\Act b(x,u)\,v(\D{u})\,,\quad \text{and}\quad
\Bar{r}(x,v, \mu)\;\df\;\int_\Act r(x,u, \mu)\,v(\D{u})\,.
\end{equation}
It is easy to see that $\Bar{b}$ satisfies (A1)--(A2),
while and running cost $\Bar{r}$ inherits the properties
in Assumption~\ref{A2.1} from $r$.
In what follows we assume that
all the controls take values in $\calP(\Act)$.
These controls are generally referred to as \emph{relaxed} controls.
We endow the set of relaxed stationary Markov controls with the following
topology: $v_{n}\to v$ in $\Usm$ if and only if
$$\int_{\R^{d}}f(x)\int_\Act g(x,u)v_{n}(\D{u}\,|\, x)\,\D{x}
\;\xrightarrow[n\to\infty]{}\;
\int_{\R^{d}}f(x)\int_\Act g(x,u)v(\D{u}\,|\, x)\,\D{x}\,,$$
for all $f\in L^{1}(\R^{d})\cap L^{2}(\R^{d})$
and $g\in\calC_{b}(\R^{d}\times \Act)$.
Then $\Usm$ is a compact metric space under this topology
\cite[Section~2.4]{ari-bor-ghosh}.
We refer to this topology as the \emph{topology of Markov controls}.
A control is said to be \emph{precise/strict} if it takes values in $\Act$.
It is easy to see that any precise control $U_{t}$ can also be understood
as a relaxed control by $U_{t}(\D{u})=\delta_{U_{t}}$.
Abusing the notation we denote the drift and running cost by $b$ and $r$,
respectively, and the action of a relaxed control
on them is understood as in \eqref{relax}.

Now we introduce the control problem.
Let $\eta\in\calC\bigl([0, \infty),\calP(\Rd)\bigr)$. We define the ergodic cost
as follows
\begin{equation}\label{EJx}
J_x(U,\eta)\;\df\; \limsup_{T\to\infty}\, \frac{1}{T}\,
\Exp^U_x\biggl[\int_0^T r(X_t, U_t, \eta_t) \,\D{t}\biggr]\,, \quad U\in\Uadm\,,
\quad x\in\Rd\,.
\end{equation}
Let
\begin{equation*}
\varrho_{\eta}(x) \;\df\; \inf_{U\in\Uadm}\; J_x(U,\eta)\,.
\end{equation*}

\begin{definition}\label{Defi-2.1}
$\eta\in\calC([0, \infty), \calP(\Rd))$ is said to be a Mean Field Game
(MFG) solution starting at $x\in\Rd$
if there exists an admissible control $v$ such that
\begin{align*}
\D{X_t} \; &=\; b(X_t, v_t)\, \D{t} + \upsigma(X_t)\, \D{W_t}\,,\\
\Law(X_t)\; &=\; \eta_t\,,\quad X_{0}=x\,,
\end{align*}
and $J_x(U,\eta)\ge J_x(v, \eta)$ for all admissible $U$.
We say the MFG solution is relaxed (strict) if the control $v$ is
a stationary Markov control
taking values in $\calP(\Act)$ ($\Act$, respectively).
\end{definition}

One of our main goals in this paper is to establish the existence of MFG solutions. 
First we review some basic facts about ergodic occupation measures
and invariant probability measures for a controlled diffusion
as in \eqref{E-sde}.
The set of all ergodic occupation measures is defined as
\begin{equation}\label{Eeom}
\eG\;\df\;\Bigl\{\uppi\in\calP(\Rd\times\Act)\; :\;
\int_{\Rd\times\Act} L^uf(x)\,\uppi(\D{x},\D{u})=0
\quad \forall\, f\in\calC^2_c(\Rd)\Bigr\}\,.
\end{equation}
By \cite[Lemma~3.2.3]{ari-bor-ghosh} $\eG$ is a closed and convex subset on
$\calP(\Rd\times\Act)$.
Disintegrating an ergodic occupation measure $\uppi$ we write
$\uppi(\D{x},\D{u})=\mu_v(\D{x}) v(\D{u} \,|\, x)$ for some
$\mu_{v}\in\Pm(\Rd)$ and some measurable kernel
$v: \Rd\to\calP(\Act)$.
We use the the notation
$\uppi=\mu_v\cast v$ to denote this disintegration.
It straightforward to verify that $\mu_v$ satisfies
$$\int_{\Rd} L^v f(x)\, \mu_v(\D{x})\; =\; 0 \quad \text{for all~}
f\in\calC^2_c(\Rd)\,,$$
and is therefore an invariant probability measure for the diffusion
controlled by $v$.
It follows that $v\in\Ussm$.
Conversely, if $v\in\Ussm$, then there exists a unique
invariant probability measure for the diffusion
under the control $v\in\Ussm$, and $\uppi_v\df \mu_v\cast v$
is an ergodic occupation measure.

Thus, the set of all invariant probability measures
may be defined as
\begin{equation}\label{Einvm}
\eH\;\df\; \bigl\{\nu\in\calP(\Rd)\; :\; \nu\cast v\in \eG \text{~for some~}
v\in\Usm\bigr\}\,.
\end{equation}
This is a convex subset of $\calP(\Rd)$.
We refer to $\uppi_v$ ($\mu_v$) as the ergodic occupation measure
(invariant probability measure) associated with $v\in\Ussm$.

The sets $\eG$ and $\eH$ play a key role in the analysis
of the ergodic control problem.
In fact, we are going to exhibit MFG solutions associated with
$v\in\Ussm$ and $\uppi\in\eG$ that satisfy the following
\begin{align}\label{Ehjb0}
\uppi \;=\; \mu_v\cast v, \qquad
\min_{u\in\Act}\;\bigl[L^u V_{\mu_v}(x) + r_{\mu_v}(x,u)\bigr]
\;=\; \varrho_{\mu_v}\,, 
\end{align}
for some function $V_{\mu_v}\in\calC^2(\Rd)$. 
Existence results of type \eqref{Ehjb0} is established in Section~\ref{S-convex}.
Such existence result is generally shown using fixed point arguments
\cite{lasry-lions}.
This is also related to the compactness property of $\eH$. 
When the state space is compact, then of course $\eH$ is also compact.
But this is not true in general for non-compact state spaces.
We adopt the following notation.
For any $G\subset\eG$ we let $\eH[G]$ denote the corresponding set of  invariant
measures, i.e.,
$$\eH[G]\;\df\;\{\mu\in\eH\,\colon\,\mu\cast v\in G\text{~for some~} v\in\Ussm\}\,.$$
Consider the following assumption.

\begin{assumption}\label{A2.2}
The following hold:
\begin{itemize}
\item[(i)]
There exist $\mu_{0}\in\eH$ and
$\uppi_0\in\eG$ such that $\uppi_0(r_{\mu_{0}})< \infty$.
\smallskip
\item[(ii)]
For models satisfying Assumption~\ref{A2.1}\,(C1), there exists
a nonempty compact set $\sK\subset\eG$ such that
\begin{equation*}
\uppi(r_{\mu})\;>\; \Tilde\varrho_{\mu}\qquad \forall\,\uppi\in
\eG\cap\sK^{c}\,,
\end{equation*}
and for all $\mu\in\eH$ where $\tilde{\varrho}_\mu= \inf_{\uppi\in\eG}\uppi(r_\mu)$.
\end{itemize}
\end{assumption}

\begin{remark}
Assumption~\ref{A2.2}\,(i) is rather standard in ergodic
control---if it is violated, the problem is vacuous.
Note that Assumption~\ref{A2.2}\,(i) always holds for the model
in (C3) of Assumption~\ref{A2.1}.
Also, Assumption~\ref{A2.2}\,(i) implies that
for  running costs satisfying (C1)--(C2) of Assumption~\ref{A2.1}
it holds that $\uppi_0(r_{\mu})< \infty$ for all $\mu\in\eH$.
\end{remark}

Define $h(p,x,u,\mu)\df p\cdot b(x,u) + r(x,u,\mu)$, $p\in\Rd$.
Our main result of this section is the following.

\begin{theorem}\label{T2.1}
Let Assumptions~\ref{A2.1} and \ref{A2.2} hold.
Then, for any $x\in\Rd$,
there exists a relaxed MFG solution starting at $x$
in the sense of Definition~\ref{Defi-2.1}.
Moreover, if $\Act$ is convex and $u\mapsto h(p, x,u, \mu)$ is strictly convex
for all $x, p\in\Rd$, and $\mu\in\calP(\Rd)$,
then there exists a strict MFG solution.
\end{theorem}

\section{MFG solutions for HJB}\label{S-convex}
In this section we investigate the existence of MFG solutions for the associated
Hamilton--Jacobi--Bellman (HJB) equation given by \eqref{Ehjb0}.

Recall the notation $r_\mu(x,u)\,=\, r(x,u, \mu)$.
Consider the ergodic control problem
\begin{equation*}
\varrho^{*}_\mu \,\df\, \inf_{U\in\Uadm} \,\limsup_{T\to\infty}\,
\frac{1}{T}\, \Exp^U_x\biggl[\int_0^T r_\mu(X_t, U_t) \,\D{t}\biggr]
\end{equation*}
for fixed $\mu\in\Pm(\Rd)$.
Also recall the abbreviated notation
$\uppi(r)=\int_{\Rd\times\Act} r\,\D{\uppi}$.
We need the following definition.

\begin{definition}
Let $f\,\colon\,\Rd\times\Act\to\RR_{+}$.
We say that $\Bar\uppi\in\eG$ is \emph{optimal relative to $f$}
(for the ergodic cost criterion) if
$\Bar\uppi(f) = \inf_{\uppi\in\eG}\, \uppi(f)$.
For $\mu\in\eH$, we let $\calA(\mu)\subset\eG$ denote the set of optimal ergodic
occupation measures relative to $r_{\mu}$,
and $\calA^{*}(\mu)\subset\eH$
denote the corresponding set of invariant probability measures.
We also let $\Tilde\varrho_\mu\df\inf_{\uppi\in\eG}\, \uppi(r_\mu)$.
\end{definition}

There are two general models for which there exists an
optimal ergodic occupation relative to $r_{\mu}$ for $\mu\in\Pm(\Rd)$,
and optimality can be characterized by the HJB equation:

\smallskip

\noindent\textbf{(H1)}
The running cost $r_{\mu}$ satisfies
$\liminf_{\abs{x}\to\infty}\,\inf_{u\in\Act}\,r_{\mu}(x,u) > \Tilde{\varrho}_{\mu}$,
and $\Tilde{\varrho}_{\mu}<\infty$.

\smallskip

\noindent\textbf{(H2)}
The set $\eH$ is compact, and $r_{\mu}$ is uniformly integrable
with respect to $\eH$.

\medskip
For models in (H1)--(H2) we assume that
$r:\Rd\times\Act\times\calP(\Rd)\to \R_+$ is continuous.
Hypothesis~(H2) is equivalent to (C3a) of Assumption~\ref{A2.1}, with $h$ satisfying
$r_{\mu}\in\sorder(h)$ by \cite[Theorem~3.7.2]{ari-bor-ghosh}.

We quote the following result which is contained in
Theorems~3.6.10 and Theorem~3.7.12 of \cite{ari-bor-ghosh}.

\begin{theorem}\label{T3.1}
If (H1) holds, then
there exists a unique $V_\mu\in\calC^2(\Rd)$
which is bounded below in $\Rd$
and satisfies
\begin{equation}\label{Ehjb1}
\min_{u\in\Act}\;\bigl[L^u V_\mu(x) + r_\mu(x,u)\bigr]\;=\;
\Tilde{\varrho}_\mu\,, \quad V_\mu(0)\;=\;0\,.
\end{equation}
Under (H2), there exists a unique
$V_\mu\in\calC^2(\Rd)\cap\sorder(\calV)$
satisfying \eqref{Ehjb1} (see \cite[Theorem~3.7.12]{ari-bor-ghosh}).
In either case, $\Tilde\varrho_{\mu}=\varrho^*_\mu$, and
$v\in\Usm$ is optimal for the ergodic
control problem if and only if it satisfies
\begin{equation}\label{Ehjb2}
\min_{u\in\Act}\;\bigl[L^u V_\mu(x) + r_\mu(x,u)\bigr]\;=\;
L^v V_\mu(x) + r_\mu\bigl(x,v(x)\bigr) \quad \text{almost everywhere in $\Rd$}\,.
\end{equation}
\end{theorem}

It follows by Theorem~\ref{T3.1} that if (H1) or (H2) hold, then
the set valued maps $\calA$ and $\calA^{*}$ can be
characterized by
\begin{align*}
\calA(\mu)&\,=\, \bigl\{\uppi\in \eG\;\colon\, \uppi=\uppi_v\,,~ \text{where}~
v\in\Ussm ~\text{satisfy}~ \eqref{Ehjb2}\}\,,
\\[5pt]
\calA^{*}(\mu) &\,=\, \{\nu\in\eH\;\colon\, \nu\cast v\in\calA(\mu)
~\text{for some}~ v\in\Ussm\}\,.
\end{align*}

This motivates the definition of the following notion of an MFG solution.

\begin{definition}\label{DMFG2}
An invariant probability measure $\mu\in\eH$ is said to be a
MFG solution if  $\mu\in\calA^*(\mu)$ and there exists
$V_\mu\in\calC^2(\Rd)$ and $v\in\Ussm$ such that
\begin{align}
\min_{u\in\Act}\;\bigl[L^u V_\mu(x) + r_\mu(x,u)\bigr]
&\;=\;  L^v\, V_\mu + r_\mu\bigl(x,v(x)\bigr) \;=\;
\Tilde\varrho_\mu\quad \text{a.e.}\; x\in\Rd\,,\label{Ehjb3}\\
\int_{\Rd} L^v f(x)\, \mu(\D{x}) & \,=\, 0 \quad \forall \; f\in\calC^2_c(\Rd)\,.
\label{Eipm}
\end{align}
We retain the notion of a relaxed, or strict solution form
Definition~\ref{Defi-2.1}.
\end{definition}

Equation \eqref{Ehjb3} is the HJB equation corresponding to the
ergodic control problem with running cost $r_\mu$,
while \eqref{Eipm} asserts that
$\mu=\mu_v$ is the invariant probability measure associated with the optimal
Markov control $v$.

\begin{remark}
The reader should have noticed the relation between Definitions~\ref{Defi-2.1}
and~\ref{DMFG2}.
It should observed that the initial distribution in Definition~\ref{Defi-2.1} is a
Dirac mass at $x$.
In fact, one may consider any \textit{nice} distribution as initial condition
in Definition~\ref{Defi-2.1}.
For example, if we fix
the initial condition to be $\mu$ satisfying \eqref{Eipm},
then a solution $\mu\in\Pm$ according to Definition~\ref{DMFG2}
gives rise to a solution according to 
Definition~\ref{Defi-2.1} due to stationarity.
\end{remark}

We start with the following lemma.

\begin{lemma}\label{L3.1}
Suppose that either (H1) or (H2) hold.
Then the set $\calA^*(\mu)$ is non-empty, convex and compact in
$\Pm(\Rd)$ under the total variation norm topology. 
\end{lemma}

\begin{proof}
It is well known that $\eG$ is convex
(see \cite[Lemma~3.2.3]{ari-bor-ghosh}).
The convexity of $\calA(\mu)$ follows by the linearity
of the map $\uppi\to \int_{\Rd\times\Act}r_\mu(x,u) \pi(\D{x},\D{u})$.
It then follows that
$\calA^*(\mu)$ is convex by the linearity of the projection.

To prove compactness, let $\{\nu_n\}$ be a sequence in $\calA^*(\mu)$, and
$\{\uppi_n\}$ be a corresponding sequence in 
$\calA(\mu)$ i.e., $\uppi_n=\nu_n\cast v_n$ for some $v_n\in\Ussm$ that
satisfies \eqref{Ehjb3}.
Let $(\Rd\times\Act)\cup\{\infty\}$ be the one point compactification
of $(\Rd\times\Act)$.
If $\{\uppi_n\}$ is not tight in $\Pm(\Rd\times\Act)$
then there exist  a constant $\varepsilon>0$,
and a subsequence, also denoted by $\{\uppi_n\}$, such that
$\uppi_{n}$ converges to a probability measure of the form
$\uppi'$
on $(\Rd\times\Act)\cup\{\infty\}$ such that $\uppi'(\infty)\ge\varepsilon$.
It is evident from the near monotone condition in (H1) that
$\uppi'(\Rd\times\Act)>0$.
It is also standard to show
that $\frac{1}{1-\uppi'(\infty)}\uppi'$ is an ergodic occupation
measure on $\Rd\times\Act$
which implies by optimality that
\begin{equation}\label{EL3.1a}
\frac{1}{1-\uppi'(\infty)}\uppi'(r_\mu) \;\ge\; \Tilde\varrho_\mu\,.
\end{equation}
However, the lower semicontinuity of the
map $\uppi\mapsto\uppi(r_\mu)$ and (H1) imply that
$\uppi'(r_\mu) < \bigl(1-\uppi'(\infty)\bigr)\Tilde\varrho_\mu$,
which contradicts \eqref{EL3.1a}.
Therefore $\{\uppi_n\}$ must be tight in $\Pm(\Rd\times\Act)$ which
implies that $\{\nu_n\}$ is tight.
On the other hand, under (H2), $\{\nu_n\}$ is trivially tight.
Consider any subsequence such that
$\nu_n\to \nu$ in $\calP(\Rd)$
and $v_n\to v$ in $\Usm$ under the topology of Markov controls,
and let $\uppi_n\df\nu_n\cast v_n$.
It follows by \cite[Lemma~3.2.6]{ari-bor-ghosh}
that $\uppi_n\to\uppi\df\nu\cast v\in\eG$
as $n\to\infty$.
By the lower semicontinuity of the map $\uppi\mapsto \uppi(r_{\mu})$ we have
$$\Tilde\varrho_\mu \;=\;
\liminf_{n\to\infty}\;\uppi_n(r_\mu)
\; \ge\; \uppi(r_\mu)\;\ge\;\Tilde\varrho_\mu\,,$$
which implies that $\calA(\mu)$ is closed and therefore compact.
It then follows by \cite[Lemma~3.2.5]{ari-bor-ghosh} that
$\calA^*(\mu)$ is compact in
$\Pm(\Rd)$ under the total variation norm topology.
Compactness of $\calA^*(\mu)$ is obvious under (H2).
\end{proof}

\begin{remark}
It follows by Proposition~\ref{P-villani},
that for a running cost satisfying (C2),
$\calA^*(\mu)$ is
compact in $\calP_p(\Rd)$ for all $\mu\in \calP_p(\Rd)$
such that $\Tilde\varrho_{\mu}<\infty$.
\end{remark}

The following theorem asserts the existence of MFG solutions in the sense of
Definition~\ref{DMFG2}.

\begin{theorem}\label{T3.2}
Suppose that Assumptions~\ref{A2.1}--\ref{A2.2} hold.
Then there exists a relaxed MFG solution in the sense of
Definition~\ref{DMFG2}.
Moreover, if $\Act$ is convex and $u\mapsto h(p, x,u, \mu)$ is strictly convex
for all $x, p\in\Rd$, and $\mu\in\calP(\Rd)$,
then there exists a strict MFG solution.
\end{theorem}

The rest of this section is devoted in proving the above result.
The proof is an application of the Kakutani--Fan--Glicksberg fixed point theorem.
A similar fixed point theorem has been
applied in \cite{lacker} to obtain
MFG solutions for finite horizon control problems.
Readers may consult \cite[Chapter~17]{ali-bor} for some basic
properties of set-valued maps used in the proofs below.

We recall the
definition of \emph{hemicontinuity} \cite[Section~17.3]{ali-bor}.

\begin{definition}
The map $\mu\mapsto\calA^*(\mu)$ is said to be \emph{upper hemicontinuous}
if whenever $\mu_n\to\mu$ as $n\to\infty$, and $\nu_n\in\calA^*(\mu_n)$ for all $n$,
then the sequence $\{\nu_n\}$ has a limit point in $\calA^*(\mu)$.
The map $\mu\mapsto\calA^*(\mu)$ is said to be \emph{lower hemicontinuous} 
if whenever $\mu_n\to\mu$ as $n\to\infty$ and $\nu\in\calA^*(\mu)$, then there
exists a subsequence $\{\nu_{n_k}\}$ such that
$\nu_{n_k}\in\calA^*(\mu_{n_k})$ and $\nu_{n_k}\to \nu$
as $n_k\to\infty$.
The map $\mu\mapsto\calA^*(\mu)$ is said to be \emph{continuous} if it is both
upper and
lower hemicontinuous.
\end{definition}

We have the following general lemma.

\begin{lemma}\label{L3.3}
Suppose that
\begin{itemize}
\item[(a)]
$r:\Rd\times\Act\times\calP(\Rd)\to \R_+$ is continuous;
\smallskip
\item[(b)] $\mu\mapsto\Tilde\varrho_{\mu}$ is upper semicontinuous;
\smallskip
\item[(c)] whenever $\mu_{n}\to\mu$, then 
$\calA^*(\mu_{n})$ is tight along some subsequence\,.
\end{itemize}
Then $\mu\mapsto\calA^*(\mu)$ is upper hemicontinuous,
and $\mu\mapsto\Tilde\varrho_{\mu}$ is continuous.
\end{lemma}

\begin{proof}
Since $\Usm$ is compact under the topology of Markov controls,
and (c) holds, it
is enough to show that $\mu\mapsto\calA(\mu)$ is upper hemicontinuous.
So suppose $\mu_n\to\mu$ as $n\to\infty$ and $\uppi_{n}\in\calA(\mu_{n})$.
Let $\Hat\uppi$ be the limit of $\uppi_{n}$ along some subsequence also
denoted as $\{\uppi_{n}\}$.
Then,
\begin{align}\label{Eusc}
\liminf_{n\to\infty}\;\Tilde\varrho_{\mu_{n}}&\;=\;
\liminf_{n\to\infty}\;\uppi_{n}(r_{\mu_{n}})\nonumber\\[5pt]
&\;\ge\;\Hat\uppi(r_{\mu})\nonumber\\[5pt]
&\;\ge\;\Tilde\varrho_{\mu}\,.
\end{align}
Since by hypothesis $\limsup_{n\to\infty}\;\Tilde\varrho_{\mu_{n}}\le
\Tilde\varrho_{\mu}$, equality follows in \eqref{Eusc}.
Since $\Hat\uppi\in\eG$ and
$\Hat\uppi(r_{\mu})=\Tilde\varrho_{\mu}$, we have that $\Hat\uppi\in\calA(\mu)$,
and upper hemicontinuity
of $\mu\mapsto\calA(\mu)$ follows.
Moreover, it follows by \eqref{Eusc} and (b)
that $\mu\mapsto\Tilde\varrho_{\mu}$ is necessarily continuous.
\end{proof}

Consider the model in (H1).
Note that  Assumption~\ref{A2.1}\,(C2) implies \eqref{A2.1a}
and \eqref{A2.1b}.

\begin{lemma}\label{L3.4}
Suppose that Assumptions~\ref{A2.1}\,(C1) and \ref{A2.2}\,(i) hold.
Then $\mu\mapsto\calA^*(\mu)$ is upper hemicontinuous on $\eH$.
\end{lemma}

\begin{proof}
It is evident that since $\Tilde\varrho_{\mu}$ is finite for some $\mu\in\eH$,
then \eqref{A2.1a} implies that it is finite for all $\mu\in\eH$.
It then follows by \eqref{A2.1a} that $\cup_{\mu\in\cK}\calA^{*}(\mu)$ is
tight, and it is routine to show that this together with \eqref{A2.1a} imply
that $\mu\mapsto\Tilde\varrho_{\mu}$ is continuous on $\eH$.
The result then follows by Lemma~\ref{L3.3}.
\end{proof}

Next we turn to the model in (H2).
By \cite[Theorem~3.7.2]{ari-bor-ghosh},  Assumption~\ref{A2.1}\,(C3)
is equivalent to
\begin{equation*}
\sup_{\uppi\in\eG}\; \int_{B^{c}_{R}\times\Act} \sup_{\mu\in\eH}\,r_{\mu}(x,u)
\,\uppi(\D{x},\D{u})\;\xrightarrow[R\to\infty]{}\;0\,.
\end{equation*}
We work under a weaker hypothesis.

\begin{lemma}\label{L3.5}
Let (H2) hold and suppose that
\begin{equation*}
\sup_{\uppi\in\eG}\; \sup_{\mu\in\eH}\;\int_{B^{c}_{R}\times\Act} \,r_{\mu}(x,u)
\,\uppi(\D{x},\D{u})\;\xrightarrow[R\to\infty]{}\;0\,.
\end{equation*}
Then $\mu\mapsto\calA^*(\mu)$ is upper hemicontinuous on $\eH$.
\end{lemma}

\begin{proof}
If $\mu_n\to\mu$ and $\Bar\uppi^\mu\in\eG$ is optimal relative
to $r_{\mu}$, then by uniform integrability we have
\begin{equation*}
\limsup_{n\to\infty}\;\Tilde\varrho_{\mu_{n}}\;\le\;
\limsup_{n\to\infty}\;\Bar\uppi^\mu(r_{\mu_{n}})
\;=\; \Tilde\varrho_{\mu}\,,
\end{equation*}
which implies that $\mu\mapsto\Tilde\varrho_{\mu}$ is continuous.
The result then follows by Lemma~\ref{L3.3}.
\end{proof}

Lemmas~\ref{L3.4} and \ref{L3.5} imply the following.

\begin{corollary}\label{C3.1}
Let Assumptions~\ref{A2.1}--\ref{A2.2} hold.
Then $\mu\mapsto\calA^*(\mu)$ is upper hemicontinuous on $\eH$.
\end{corollary}

In order to apply the fixed point theorem
it remains to show that there exists some nonempty, convex and compact set
$\cK\subset\eH$ such
that $\calA^{*}(\mu)\subset\cK$ for all $\mu\in\cK$.
For the models satisfying (H2) we can select $\cK\equiv\eH$.

We have the following lemma.

\begin{lemma}\label{L3.6}
Let Assumptions~\ref{A2.1} and \ref{A2.2} hold.
There exists a non-empty, convex and compact set $\calK\subset\calP(\Rd)$
such that for $\mu\in\calK$ we have $\calA^*(\mu)\subset\calK$. \end{lemma}

\begin{proof}
Under Assumption~\ref{A2.1}\,(C3) we choose $\calK=\eH$.

Suppose that Assumption~\ref{A2.1}\,(C1) holds and let $\sK$ be as in
Assumption~\ref{A2.2}. Then 
$\eH\bigl[\sK\bigr]$ is compact under the total variation norm topology
\cite[Lemma~3.2.5]{ari-bor-ghosh}.
Therefore it follows that $\overline{\conv\eH\bigl[\sK\bigr]}$ is also compact
in the total variation norm topology \cite[Theorem~5.35]{ali-bor}.
Defining $\calK=\overline{\conv\eH\bigl[\sK\bigr]}$ we see that $\calK$ is a convex,
compact subset of
$\calP(\Rd)$.
Again by Assumption~\ref{A2.2}\,(ii) it easy to see that $\calK$ has required property.

Next, consider Assumption~\ref{A2.1}\,(C2),
and let $\uppi_{0}=\nu_{0}\cast v_{0}$ be as
in Assumption~\ref{A2.2}\,(i).
For $R>0$, let $M_{R}, N_{R}\subset\eH$ be defined by
\begin{align*}
M_{R}&\;\df\;\biggl\{\mu\in\eH\,\colon\,
\kappa_{0}\,\biggl(1+\int_{\Rd}\abs{x}^p\,\nu_{0}(\D{x})
+\int_{\Rd}\abs{x}^p\,\mu(\D{x})\biggr)\,\le\, R\biggr\}\,,\\[5pt]
N_{R}&\;\df\;\biggl\{\mu\in\eH\,\colon\,
\int_{\Rd}\Bigl(\min_{u\in\Act}\;\mathring{r}(x,u)\Bigr)\,\nu(\D{x})\;\le\;
\uppi_{0}(\mathring r) + R\biggr\}\,.
\end{align*}
By Assumption~\ref{A2.1}\,(C2),
there exists $R_{0}>0$ such that $M_{R_{0}}\supset N_{R_{0}}$.
It is evident that $N_{R_{0}}$ is convex and compact in $\Pm_{p}(\Rd)$.
Let $\mu\in N_{R_{0}}\subset M_{R_{0}}$.
If $\uppi=\nu\cast v\in\eG$, and $\nu\in N_{R_{0}}^{c}$, then
\begin{align*}
\uppi(\mathring r) + \int_{\Rd} F(x,\mu)\,\nu(\D{x})
&\;>\; \uppi_{0}(\mathring r) + R_{0}\nonumber\\[5pt]
&\;\ge\; \uppi_{0}(\mathring r) + 
\kappa_{0}\,\biggl(1+\int_{\Rd}\abs{x}^p\,\nu_{0}(\D{x})
+\int_{\Rd}\abs{x}^p\,\mu(\D{x})\biggr)\nonumber\\[5pt]
&\;\ge\; \uppi_{0}(\mathring r) + F(\nu_{0},\mu)\,,
\end{align*}
where the first inequality follows since $\nu\in N_{R_{0}}^{c}$,
while the second follows from the hypothesis that $\mu\in N_{R_{0}}\subset M_{R_{0}}$.
This of course implies that $\uppi\notin\cA(\mu)$.
Therefore $\cA^*(\mu)\in N_{R_{0}}$ for all $\mu\in N_{R_{0}}$.
This completes the proof.
\end{proof}

Next we prove Theorem~\ref{T3.2}.

\begin{proof}[Proof of Theorem~\ref{T3.2}.]\,
Consider the map $\mu\in\calK\mapsto \calA^*(\mu)\in 2^{\calK}$ where $\calK$
is chosen from Lemma~\ref{L3.6}.
We note that $\calK$ is a non-empty, convex and compact
subset of $\calM(\Rd)$ which is a locally convex Hausdorff space under
the weak topology.
By Lemma~\ref{L3.1} $\calA^*(\mu)$ is non-empty, convex and compact.
From Lemma~\ref{L3.1},
Corollary~\ref{C3.1} and \cite[Theorem~17.10]{ali-bor} we conclude that the map
$\mu\mapsto A^*(\mu)$ has closed graph.
Therefore applying the Kakutani--Fan--Glicksberg fixed point theorem
(see \cite[Corollary~17.55]{ali-bor}) there exists $\mu\in\calK$
satisfying $\mu\in\calA^*(\mu)$.
This proves the existence of a relaxed MFG solution in
the sense of Definition~\ref{DMFG2}.
 
Suppose now that $\Act$ is convex and
$u\mapsto h(p, x,u, \mu)$ is strictly convex for all
$x, p\in\Rd$, $\mu\in\calP(\Rd)$.
Then we can find a unique continuous,
strict Markov control $v : \Rd\to \Act$ such that 
\begin{equation*}
\min_{u\in\Act}\;\bigl[L^u V_\mu(x) + r_\mu(x,u)\bigr]\;=\;
L^v V_\mu(x) + r_\mu\bigl(x,v(x)\bigr) \qquad \forall\,x\in\Rd\,.
\end{equation*}
Note that in this case $\cA^{*}(\mu)$ is a singleton,
and $\mu\mapsto\calA^*(\mu)$ is continuous in $\Pm(\Rd)$.
Hence, an application of the Schauder--Tychonoff fixed point theorem
suffices to assert existence of a strict MFG solution.
\end{proof}

\begin{remark}
It is possible to allow the drift $b$ to depend on the measure $\mu$.
In case (C3) we can even consider a continuous 
$b: \Rd\times\Act\times\calP(\Rd)\to \Rd$ such that $b(\cdot\,, u, \mu)$ is
locally Lipschitz uniformly with respect to $u\in\Act$ for all $\mu\in\calP(\Rd)$.
The argument in the proof of
 Theorem~\ref{T3.2} holds in this case if \eqref{US} is satisfied.
In particular, consider $b(x,u, \mu) \equiv b(x,u) + e(\mu)$
for some bounded continuous
vector valued map $e:\calP(\Rd)\to\Rd$ where $b$ satisfies the following:
there exists $\calV, h$ satisfying \eqref{US} when the operator $L$ is defined
using the drift $b(x,u)$ and $\abs{\grad V}\in\sorder(h)$. Then
it is easy to see that \eqref{US} holds for the original drift $b(x,u, \mu)$
with the same functions $\calV$ and $h$.
\end{remark}

\begin{proof}[Proof of Theorem~\ref{T2.1}.]
Consider a MFG solution $\mu$ in the sense of Definition~\ref{DMFG2}
and take a relaxed/strict control $v\in\Ussm$ associated to it
in $\calA(\mu)$. The existence of such a $v$ is
assured by Theorem~\ref{T3.2}.
We know that there exists a unique strong Markov process corresponding to
$v$ satisfying \eqref{E-sde} i.e.,
\begin{equation*}
\D{X}_{t} \;=\;b(X_{t}, v(X_t))\,\D{t} + \upsigma(X_{t})\,\D{W}_{t}\,,
\quad X_0=x\, .
\end{equation*}
By definition, $\mu$ is the unique invariant probability
measure of the process $X$ under the control $v$.
Let $\eta\in\calC([0, \infty), \calP(\Rd))$ be the path of transition
probabilities of this process.
It suffices to show that $J_x(v, \eta)=\Tilde{\varrho}_\mu$ for all
$x\in\Rd$, and that
for any admissible control $U\in\Uadm$ we have 
\begin{equation}\label{T2.1a}
J_x(U,\eta)\;\ge\; \Tilde{\varrho}_\mu\qquad \forall\, U\in\Uadm\,,
\qquad\forall\,x\in\Rd\,.
\end{equation}
where $J$ is defined by \eqref{EJx}.
We divide the proof in three cases.

\noindent\emph{Case~1.}
Consider models satisfying (C3). Applying \cite[Proposition~2.6]{ichihara-sheu}
it follows that there exists a compact set $G\in\calP(\Rd)$ such that
$\eta_{t}\in G$ for all $t\ge0$.
Therefore
\begin{equation}\label{T2.1b}
r_{\eta_{t}}\;\le\;
\sup_{\nu\in G}\;r_\nu \;\in\; \sorder(h)\qquad\forall t\ge0\,.
\end{equation}
Also by \eqref{US}, it follows that
\begin{equation*}
\limsup_{T\to\infty}\;\frac{1}{T}\;\Exp^U_x
\biggl[\int_0^T h(X_t,U_t)\,\D{t}\biggl]\;\le\;
\frac{1}{c_{2}}\bigr(c_{1}-\Lyap(x)\bigr)
\qquad\forall\,U\in\Uadm\,.
\end{equation*}
This shows that
\begin{equation}\label{T2.1c}
\limsup_{T\to\infty}\;\frac{1}{T}\;\Exp^U_x
\biggl[\int_0^T \Ind_{B_{R}^c}(X_t)\sup_{\nu\in G}\,
r_\nu(X_t,U_t)\,\D{t}\biggr]\;\xrightarrow[R\to\infty]{} 0
\qquad\forall\,U\in\Uadm\,.
\end{equation}
Therefore since $r$ is continuous and $\eta_t\to\mu$ in $\Pm(\Rd)$ as $t\to\infty$,
we obtain by \eqref{T2.1b} and \eqref{T2.1c} that
\begin{equation}\label{T2.1d}
\limsup_{T\to\infty}\;\frac{1}{T}\;\Exp^v_x\biggl[
\int_0^T r(X_t, v(X_t), \eta_t)\,\D{t}\biggr]\;=\;
\limsup_{T\to\infty}\;\frac{1}{T}\;\Exp^v_x
\biggl[\int_0^T r(X_t,  v(X_t), \mu)\,\D{t}\biggr]\;=\;\Tilde\varrho_\mu\,.
\end{equation}
It remains to show that
\begin{equation}\label{T2.1e}
\limsup_{T\to\infty}\;\frac{1}{T}\;\Exp^U_x\biggl[
\int_0^T r(X_t, U_t, \eta_t)\,\D{t}\biggr]\; \ge\; \Tilde\varrho_\mu
\qquad\forall\,U\in\Uadm\,.
\end{equation}
For this purpose, we consider a smooth cut-off function $\phi_R$ that equals
$1$ on $B_R$ and vanishes outside $B_{R+1}$.
By $\upomega$ we denote the local modulus of continuity of $r$,
defined by
\begin{multline*}
\upomega(R, G, \eps)\;\df\; \sup\;\bigl\{\abs{r(x,u,\mu
- r(\Bar{x},\Bar{u}, \Bar{\mu})} \;\colon\; \abs{x-\Bar{x}}
+ d_\Act(u, \Bar{u}) + \dd(\mu,\Bar{\mu})\le \eps, \\[3pt]
 x, \Bar{x}\in \Bar{B}_{R},~ \mu, \Bar{\mu}\in G\,,~
u, \Bar{u}\in\Act\bigr\}.
\end{multline*}
Since the mean empirical measures of the process $(X_t,U_t)$
are tight, applying Theorem~3.4.7 in \cite{ari-bor-ghosh}, we obtain
\begin{equation*}
\liminf_{T\to\infty}\;\frac{1}{T}\;\Exp^U_x\biggl[
\int_0^T \phi_R(X_t)\,r(X_t, U_t, \mu)\,\D{t}\biggr]
\;\ge\;\inf_{\uppi\in\eG}\;
\int_{\Rd\times\Act}\phi_{R}(x) r_\mu(x,u)\,\uppi(\D{x},\D{u})\,.
\end{equation*}
Thus, using the inequality
\begin{align*}
r(X_t, U_t, \eta_t)&\;\ge \;\phi_R(X_t)\,r(X_t, U_t, \eta_t)\nonumber\\[5pt]
&\;\ge\;\phi_R(X_t)\,r(X_t, U_t, \mu)
-\phi_R(X_t)\,\upomega\bigl(R+1,G,\dd(\eta_t,\mu)\bigr),
\end{align*}
 and
the fact that $\dd(\eta_t,\mu)\to0$ as $t\to\infty$,
we obtain
\begin{equation*}
\liminf_{T\to\infty}\;\frac{1}{T}\;\Exp^U_x\biggl[
\int_0^T r(X_t, U_t, \eta_t)\,\D{t}\biggr]\;\ge \;
\inf_{\uppi\in\eG}\;
\int_{\Rd\times\Act}\phi_{R}(x) r_\mu(x,u)\,\uppi(\D{x},\D{u})\,.
\end{equation*}
Letting $R\to\infty$,  and using the fact $\Tilde\varrho_\mu$ is the optimal value,
we obtain \eqref{T2.1e}.

\noindent\emph{Case~2.}
We consider running costs satisfying (C1).
From the HJB equation we have
\begin{equation*}
L^v V_\mu(x) + r_\mu(x,v(x))\;=\; \Tilde{\varrho}_\mu\,,
\end{equation*}
with $\mu\in\cA^*(\mu)$.
Hence, by \cite[Lemma~3.7.2]{ari-bor-ghosh}, there exist nonnegative,
inf-compact functions $\Lyap\in\Cc^{2}(\Rd)$ and $h\in\Cc(\Rd)$
such that $r_\mu(\cdot,v(\cdot))\in\sorder(h)$, and
satisfy
\begin{equation}\label{T2.1f}
L^v \Lyap(x) \;\le\; c_{0}-h(x)
\end{equation}
for some constant $c_{0}$.
It follows by \eqref{T2.1f} that $\eta_{t}\in G$ for all $t\ge0$,
where $G$ is a compact subset of $\Pm(\Rd)$.
By \eqref{A2.1a} we have
\begin{equation*}
\sup_{\nu\in G}\; r_{\nu}(\cdot,v(\cdot))\;\in\;\sorder(h)\,.
\end{equation*}
Repeating the argument used in Case~1, we obtain \eqref{T2.1d}.
Also \eqref{T2.1e} follows as in Case~1 by using the
near-monotone property of $r_{\mu}$.

\noindent\emph{Case~3.}
We consider (C2).
 To show \eqref{T2.1a} in this case, it is enough to show that
$F(x,\eta_t)\to F(x, \mu)$ as $t\to\infty$ uniformly in $x$ on
compact subsets of $\Rd$.
Since
$F$ is continuous in $\Rd\times\calP_p(\Rd)$, we need to show that
$\eD_{p}(\eta_t, \mu)\to 0$ as $t\to\infty$.
Since
$\int_{\Rd} \mathring{r}(x,v(x))\mu(\D{x})\le\Tilde\varrho_\mu$,
it follows that for any continuous $\phi$ with $\phi\in\order(\mathring{r})$ we have
$\int \phi \, \D{\mu}<\infty$.
Then by \cite[Proposition~2.6]{ichihara-sheu} we have
$\int \phi \, \D{\eta_t}\to \int\phi\, \D{\mu}$ as $t\to \infty$ for every 
initial condition $x$.
Combining this fact with Proposition~\ref{P-villani} we obtain that
$\eD_{p}(\eta_t, \mu)\to 0$ as $t\to\infty$.
This shows \eqref{T2.1a}. It is also easy to see that
$\varrho_\eta=\Tilde{\varrho}_\mu$.
\end{proof}

\begin{remark}
We note that for models satisfying
(C3) we can strengthen the assumption on $r$ depending on the growth rate of $h$.
For example, if $h\sim \abs{x}^p$ for $p\ge 1$, then
one may a consider continuous $r$ defined on 
$\Rd\times\Act\times\calP_p(\Rd)$
that is locally Lipschitz in first and third
arguments uniformly in $u\in\Act$, and with the property that
$\sup_{\mu\in\calK}\,r_{\mu}\in\sorder(h)$ for any compact
$\calK\subset\calP_p(\Rd)$.
The results of Theorem~\ref{T2.1} continue to hold in this case.
\end{remark}

\section{Long Time Behavior and the Relative Value Iteration}\label{S-RVI}

In this section we study the long time behavior
of the finite horizon mean field game equations.
The problem is as follows.
We are given a running cost function $r(x,u,\mu)$,
a horizon $T>0$, a `terminal cost function' $\varphi_{0}\in\Cc^{2}(\Rd)$,
and an initial distribution $\eta\in\Pm(\Rd)$.
For $U\in\Uadm$ and $\{\mu_{t}\in\Pm(\Rd)\,,\; t\in[0,T]\}$ we  define
\begin{equation*}
\mathcal{J}(U,\mu;\eta)\;\df\;
\Exp^{U}_\eta
\biggl[\int_{0}^{T} r(X_{t},U_{t},\mu_{t})\,\D{t} + V_{T}(X_{T})\biggr]\,,
\end{equation*}
where $X_{t}$ is governed by \eqref{E-sde} with $\Law(X_{0})=\eta$.
Let $\Law^{U}_{\eta}(X_{t})$ denote the law of the process $X_{t}$ governed
by \eqref{E-sde} under a control $U$ with $\Law(X_{0})=\eta$.
Then $\{\mu^{*}_{t}\in\Pm(\Rd)\,,\;t\in[0,T]\}$ is called an MFG solution
for the problem
on $[0,T]$ if there exists an admissible control $U^{*}$ such that
$\Law^{U^{*}}_{\eta}(X_{t}) = \mu^{*}_{t}$ for all $t\in[0,T]$
\begin{equation*}
\mathcal{J}(U,\mu^{*};\eta)\;\ge\; \mathcal{J}(U^{*},\mu^{*};\eta)
\qquad\forall\,U\in\Uadm\,.
\end{equation*}

We assume that
$r(x,u,\mu)$ has the separable form $\mathring{r}(x,u)+F(x,\mu)$,
so that the Hamiltonian $H(x,p)$ is given by
\begin{equation}\label{E-H}
H(x,p) = \min_{u}\; \left\{b(x,u)\cdot p + \mathring{r}(x,u)\right\}\,.
\end{equation}
Denoting by $\chi(t,\cdot\,)$ the density of $\mu_{T-t}=\Law(X_{T-t})$,
the dynamic programming formulation amounts to solving
\begin{subequations}
\begin{align}
\begin{split}
\partial_{t} V &\;=\;
 a^{ij}\partial_{ij} V + H(x,\nabla V) + F(x,\mu_{T-t})\,,\\[3pt]
&\mspace{50mu}\quad V(0,x) = \varphi_{0}(x)\quad \forall\,x\in\Rd\,,
\end{split}\label{HJB1a}\\[6pt]
\begin{split}
-\partial_{t} \chi &\;=\;
\partial_{i}\bigl(a^{ij}\partial_{j}\chi+ (\partial_{j}a^{ij})\chi\bigr)
-\text{div} \left(\tfrac{\partial H}{\partial p} (x,\nabla V) \chi\right)\,,\\
&\mspace{50mu}\chi(T,\cdot\,)~\text{is the density of}~\eta\,.
\end{split}
\label{HJB1b}
\end{align}
\end{subequations}
Equation \eqref{HJB1b} is the Kolmogorov equation for the density $\chi(t,\cdot\,)$,
running in backward time.
Therefore, if $(V,\chi)$ is a solution of \eqref{HJB1a}--\eqref{HJB1b},
then
$\mu^{*}_{t}(\D{x}) = \chi(T-t,x)\,\D{x}$
for $t\in[0,T]$ is a MFG solution in the sense of the above definition.
It also follows by the dynamic programming principle that the solution
$V_T(t,x)$, where the $T$ in the subscript denotes the dependence
of the solution on the horizon $[0,T]$,
has the stochastic representation
\begin{equation*}
V_{T}(t,x)\;=\;\inf_{U\in\Uadm}\;\Exp^{U}_{x} \left[\int_{0}^{T-t}
r\bigl(X_{s},U_{s},\mu^{*}_{T-t+s}\bigr)\,\D{t}+h(X_{T-t})\right]
\qquad\forall\,(t,x)\in[0,T]\times\Rd\,,
\end{equation*}
where the process $X$ is
governed by \eqref{E-sde}.

Lasry and Lions have examined thoroughly the case where $b(x,u)=-u$,
$\upsigma$ is the identity matrix,
$r(x,u) = \nicefrac{1}{2} \abs{u}^{2}$, and $F(x,\mu)$ takes
the form $F(x,\chi(x))$, where $\chi$ is the density of $\mu$.
Moreover, they assume that $F(x,t)\in\Rd\times\RR$ is $\Cc^{1}$,
is strictly increasing in $t$, and is
$\ZZ^{d}$-periodic in $x$.
As a result the state space is a $d$-dimensional torus $\mathbb{T}$.
Under the assumption
that the density $\eta$ is H\"older continuous and has
finite second moments, they have shown the existence and uniqueness
of a solution
to \eqref{HJB1a}--\eqref{HJB1b} for this problem \cite{lasry-lions}.
They have also proved the existence and uniqueness of a stationary solution,
i.e., the corresponding equation for the ergodic problem.
The behavior over a long horizon for this model has been studied,
with both local and non-local interactions in \cite{CLLP-13,Cardaliaguet12}.
In the case of non-local interactions,
they establish convergence in the average sense, i.e.,
$\lim_{T\to\infty}\,\frac{1}{T} V(\gamma T)= (1-\gamma) \Bar{\varrho}$,
for $\gamma\in(0,1)$, where $\Bar{\varrho}$ is the value of the associated
ergodic problem (see \eqref{T4.2a} below), and also convergence
in $L^{2}(\mathbb{T})$ uniformly over compact intervals of time.
Also, they show that the density $\chi_{T}$ converges to the density
of the stationary solution in $L^{2}(\mathbb{T})$
uniformly over compact intervals of time.
Under stronger assumptions on $F$ they show that convergence is exponential
in $T$.

For the problem on $\Rd$ we are dealing with, in order to avoid restrictive
assumptions on $F$, we have to compensate for the non-compactness
of the state space by imposing a uniform stability hypothesis on the
dynamics in \eqref{E-sde}.
We describe these assumptions in the next section.

\subsection{Assumptions and basic properties}
Existence and uniqueness of solutions to \eqref{HJB1a}--\eqref{HJB1b}
under general vector fields requires strong regularity of the data.
We refer the reader to \cite{Kolokoltsov-11a,Kolokoltsov-10B,Huang-06}.
We note here that the results in this paper can be extended to include
a drift $b$, a diffusion matrix $\sigma$ and running cost $r$ that
all depend on $\mu$, albeit necessitating various assumptions on the
smoothness of the data.

In this paper we are not interested in the regularity of the
Fokker--Planck equation \eqref{HJB1b}.
If a MFG solution $\mu_{t}$ is provided, then $F(x,\mu_{t})$ is a given
function of $(t,x)\in[0,T]\times\Rd$, and the Hamiltonian does
not depend on $\mu_{t}$.
If $F(x,\mu_{t})$ is
H\"older in $x$ and continuous in $t$, and $\varphi_{0}$ is smooth enough,
then \eqref{HJB1b} has almost classical solutions.
Therefore we concentrate on a set of assumptions that
guarantee the existence and uniqueness of a MFG solution, and at
the same time maintain sufficient regularity for the solutions
of \eqref{HJB1b}.

For $\eta\in\Pm(\Rd)$ we let
$\sM_{\eta}([0,T])$ denote the set of all trajectories
$\bigl\{\mu_{t}=\Law^{U}_{\eta}(X_{t})\,,\;U\in\Uadm\,,t\in[0,T]\bigr\}$,
and define $\mathscr{P}(\eta)
\df\{\Law^{U}_{\eta}(X_{t})\in\Pm(\Rd)\,\colon\,U\in\Uadm\,,\;t\ge0\}$.

\begin{assumption}\label{A4.1}
The following hold:
\begin{itemize}
\item[(i)]
Assumptions (A1)--(A3) on the data hold,
and $\mathring r:\Rd\times\Act\to\R_+$ is continuous and locally Lipschitz
in $x$ uniformly in $u\in\Act$.
\item[(ii)]
The function $F$ is defined on $\Rd\times\Tilde\Pm$, where $\Tilde\Pm$
is some subset of $\Pm(\Rd)$ which contains $\mathscr{P}(\eta)$,
and satisfies
\begin{equation}\label{E-mon}
\mathscr{F}(\mu,\mu')\;\df\;\int_{\Rd} \bigl(F(x,\mu) - F(x,\mu')\bigr)
\bigl(\mu(\D{x})-\mu'(\D{x})\bigr)\;\ge\;0\qquad\forall \mu,\mu'\in\Tilde\Pm\,.
\end{equation}
Moreover, $F(x,\mu)$ is locally Lipschitz
in $x$ uniformly on compact subsets of $\Tilde\Pm$,
and $\mu\mapsto F(\cdot\,,\mu)$ is a continuous map from $\Tilde\Pm$ to
$\Cc(\Rd)$ under the topology of uniform convergence on compact sets.
\item[(iii)]
The terminal cost $\varphi_{0}$ is in $\Cc^{2}(\Rd)$
and the density  of the initial distribution $\eta$ is H\"older continuous, and has
a finite second moment.
\item[(iv)]
There exists a unique $u^{*}$ that minimizes the Hamiltonian in \eqref{E-H}.
\end{itemize}
\end{assumption}

Under Assumption~\ref{A4.1}, existence of an MFG solution is
asserted in \cite[Theorem~2.1]{lacker}.
The (non-strict) monotonicity hypothesis \eqref{E-mon} together
with the fact that $\calA^{*}(\mu)$ is a singleton implied
by Assumption~\ref{A4.1}\,(iv), is enough
to guarantee uniqueness of the MFG solution for the ergodic problem.
The monotonicity hypothesis has become a standard assumption in the literature
\cite{lasry-lions,CLLP-13}.

Recall the definition of the weighted Banach space $\Cc_{\Lyap}(\RR^{d})$ from
Section~\ref{S-notation}.
The following assumption
is a strengthening of the stability hypothesis in (C3).

\begin{assumption}\label{A4.2}
A number $p\ge1$ is specified as a parameter.
There exists a nonnegative, inf-compact $\Lyap\in\Cc^{2}(\Rd)$,
and positive constants $c_{0}$ and $c_{1}$ satisfying
\begin{equation}\label{E-Lyap}
\Lg^{u} \Lyap(x) \,\le\,  c_{0} - c_{1} \Lyap(x) \qquad
\forall\, (x,u)\in\Rd\times\Act\,.
\end{equation}
Without loss of generality we assume $\Lyap\ge1$.
Also
\begin{itemize}
\item[(i)] It holds that
\begin{equation*}
\frac{\Lyap(x)}{1+\abs{x}^p}
\;\xrightarrow[\abs{x}\to\infty]{}\;\infty\,,\quad\text{and}\quad
\limsup_{\abs{x}\to\infty}\;
\frac{\sup_{u\in\Act}\;\mathring{r}(x\,,u)}{1+\abs{x}^p}\;<\;\infty\,.
\end{equation*}
\item[(ii)]
For any compact
$\cK\subset\calP_{p}(\RR^{d})$ and $R>0$,
there exists a constant $M_{p}(R)>0$ such that
such that
\begin{equation*}
\babs{F(x,\mu)-F(x',\mu)}\;\le\;M_{p}(R)
\abs{x-x'}\qquad\forall\, x,x'\in B_{R}\,,~\forall\,\mu\in\cK\,.
\end{equation*}
\item[(iii)]
The map $\mu\to F(\cdot\,,\mu)$ from $\calP_{p}(\RR^{d})$
to $\Cc_{\Lyap}(\Rd)$ is continuous.
\end{itemize}
\end{assumption}

\medskip
It is well known (see \cite{ari-bor-ghosh,GihSko72})
that \eqref{E-Lyap} implies that
\begin{equation}\label{EA4.2b}
\Exp^{U}_{x} \left[\Lyap(X_{t}) \right] \le
\frac{c_{0}}{c_{1}} + \Lyap(x) \E^{-c_{1} t}\qquad \forall x\in\RR^{d}\,,~
\forall\, U\in\Uadm\,.
\end{equation}
It follows by \eqref{EA4.2b} that
all stationary Markov controls $\Usm$ are stable and that
\begin{equation*}
\int_{\RR^{d}} \Lyap(x)\,\imeas_{v}(\D{x})\;\le\;
\frac{c_{0}}{c_{1}}\,,
\end{equation*}
where, as usual, $\imeas_{v}$ denotes the unique invariant probability measure
of the diffusion controlled under $v$.
Therefore, $\mu_{v}\in\Pm_{p}(\Rd)$ for all $v\in\Usm$.

It also follows that for any $v\in\Usm$ the controlled process
under $v$ is $\Lyap$-geometrically ergodic
(see \cite{Down-95,Fort-05}), or in other words, that
there exist constants $M_{0}$ and $\gamma>0$ such that,
if $h\,\colon\,\Rd\to\RR$ is Borel measurable and $h\in\order(\Lyap)$, then
\begin{equation}\label{E-geom}
\left|\,\Exp_{x}^{v}\bigl[h(X_{t})\bigr]
- \int_{\RR^{d}}h(x)\,\imeas_{v}(\D{x})\right|
\;\le\; M_{0} \E^{-\gamma t}
\bnorm{h}_{\Lyap} \bigl(1+\Lyap(x)\bigr)
\end{equation}
for all $t\ge0$ and $x\in\RR^{d}$.

We have the following simple assertion.

\begin{lemma}\label{L4.1}
Under Assumption~\ref{A4.2},
here exists a constant $\Tilde{M}_{p}$ which depends only on $p\ge1$
and $\Law(X_{0})\in\Pm_{p}(\Rd)$ such that
\begin{equation*}
\eD_{p}\bigl(\Law(X_{t}),\Law(X_{s})\bigr) \;\le\;
\Tilde{M}_{p}\,\sqrt{\abs{t-s}}\qquad
\forall\, s,t\in\RR_{+}\,,~\abs{s-t}<1\,,\qquad\text{under any}~ U\in\Uadm\,.
\end{equation*}
\end{lemma}

\begin{proof}
By the Burkholder--Davis--Gundy inequality, for some constant
$\kappa_{p}>0$, we obtain
\begin{equation*}
\Exp\,\biggl[\sup_{s\le r\le t}\; \abs{X_{r}-X_{s}}^{p}\biggr]
\;\le\; 2^{p-1} (t-s)^{p-1}
\Exp\,\biggl[\int_{s}^{t}\abs{b(X_{r},U_{r})}^{p}\,\D{r}\biggr]
+ 2^{p-1}
\kappa_{p} \Exp\,\biggl[\int_{s}^{t}\norm{\upsigma(X_{r})}^{p}\,\D{r}\biggr]
^{\nicefrac{p}{2}}\,.
\end{equation*}
Since $\sup_{u\in\Act}\,\abs{b(\cdot,u)}^{p}\in\order(\Lyap)$ and
$\norm{\upsigma(\cdot\,)}^{p}\in\order(\Lyap)$ by (A2) and Assumption~\ref{A4.2}\,(i),
the result follows from the inequality above and \eqref{EA4.2b}.
\end{proof}

Let
$$\mathscr{M}_{0} \df \left\{\mu\in\Pm(\RR^{d}) :
\int_{\RR^{d}} \Lyap(x)\,\mu(\D{x})\le
\frac{c_{0}}{c_{1}}\right\}\,.$$
The set $\mathscr{M}_{0}$ is compact in $\Pm_{p}(\Rd)$
and $\eH\subset\mathscr{M}_{0}$ by Assumption~\ref{A4.2}.
We have not assumed that $F$ is nonnegative. Nevertheless, Assumption~\ref{A4.2}
implies that $\inf_{\mu,\mu'\in\eH}\, \int F(x,\mu)\,\mu'(\D{x})<-\infty$, 
and therefore, the
ergodic cost problem is well posed.
Combining the preceding discussion with the results in Section~\ref{S-existence}, we
have the following.

\begin{theorem}
Let Assumptions~\ref{A4.1}--\ref{A4.2} hold.
Then there exists a unique
MFG solution $\Bar{\mu}\in\eH$ to the ergodic control problem. 
Associated with that, we obtain a unique $\Bar{V}\in\Cc^{2}(\Rd)\cap\Cc_{\Lyap}(\Rd)$,
satisfying $\Bar{V}\in\sorder(\Lyap)$ and $\Bar{V}(0)=0$, which solves
\begin{equation}\label{T4.2a}
a^{ij}(x)\partial_{ij} \Bar{V}(x) + 
H(x,\nabla\Bar{V}) + F(x,\Bar\mu)=\Bar{\varrho}
\end{equation}
with $\Bar{\varrho}=\varrho_{\Bar\mu}$.
\end{theorem}

For the rest of this section
we let $\Bar{v}$ denote some Markov control associated with
the stationary solution in \eqref{T4.2a}, i.e.,
a measurable selector from the minimizer
of the Hamiltonian $H(x,\nabla\Bar{V})$.
By uniqueness of the solutions we have $\mu_{\Bar{v}}=\Bar{\mu}$.

\subsection{The relative value iteration}
Note that the Markov control associated with
\eqref{HJB1a} is computed `backward' in time.
We need the following definition.

\begin{definition}
Let $\Hat{v}=\{\Hat{v}_{t}\,,\ t\in[0,T]\}$ denote a measurable selector
from the minimizer of the Hamiltonian in \eqref{HJB1a}.
For each $T>0$ we define the (nonstationary) Markov control
\begin{equation*}
\Hat{v}^{T}\df\bigl\{\Hat{v}^{T}_{s}\;=\;\Hat{v}_{T-s}\,,\; s\in[0,T]\bigr\}\,.
\end{equation*}
We also let 
$\Hat{\eta}^{T}_{s}$ denote the law of $X_{s}$, $s\in[0,T]$ under
the control $\Hat{v}^{T}$.
As remarked earlier $\Hat{\eta}^{T}_{s}(\cdot)=\Law^{\Hat{v}}_{\eta}(X_{s})$
for $s\in[0,T]$, and thus, $\Hat{\eta}^{T}_{0}$ agrees with the initial
law $\eta$, which we also denote by $\Hat{\eta}_{0}$.
\end{definition}

We modify \eqref{HJB1a} by normalizing it as follows:
\begin{equation}\label{VI}
\partial_{t}\varphi(t,x) \;=\; a^{ij}(x)\partial_{ij} \varphi(t,x) + 
H(x,\nabla\varphi(t,x)) + F(x,\Hat{\eta}^{T}_{T-t})-\Bar{\varrho}\,,
\quad \varphi(0,x) =\varphi_{0}(x)\,,
\end{equation}
where $\varphi_{0}\in\Cc^{2}(\Rd)\cap\sorder(\Lyap)$ denotes
the terminal cost.
It is evident the solution $\varphi$ depends also on the horizon $[0,T]$
and to distinguish among these solutions we adopt the notation
$\varphi^T(t,x)$, or
$\varphi^T_{t}(x)$.

For existence and uniqueness of solutions to \eqref{VI} in cylinders
we refer the reader
to \cite[Theorem~6.1, p.~452]{Lady} and to p.~492 of the same reference
for the Cauchy problem.
See also \cite{RVI,RVIM} for the Cauchy problem
in \eqref{VI} as well as \eqref{RVI} below.
We need to mention though that Theorem~6.1 in \cite{Lady} concerns
solutions in H\"older spaces, and in order to satisfy
the assumptions of this theorem $t\mapsto F(x,\Hat{\eta}^{T}_{T-t})$
has to be H\"older continuous.
However under our assumptions it is only continuous, which means that
the time derivative of the solution $\varphi(t,x)$ is not necessarily H\"older
continuous.
In general then, \eqref{VI} has to be solved in the parabolic Sobolev space
$\Sobl^{1,2,q}((0,\infty)\times\Rd)$ (see \cite[Section IV.9]{Lady}).
We don't require more regularity than that in this paper.

We are concerned here only with the solution $\varphi^T$ which
agrees with the stochastic representation
\begin{align}\label{E-VISR}
\varphi^T_{t}(x) &\;=\; \inf_{U}\;\Exp^{U}_{x}
\left[\int_{0}^{t}r_{\Hat{\eta}^{T}_{T-t+s}}(X_{s},U_{s})\,\D{s}
+\varphi_{0}(X_{t})\right] -\Bar{\varrho}\,t\nonumber\\[5pt]
&\;=\; \Exp^{\Hat{v}^{T}}_{x}
\left[\int_{0}^{t}r_{\Hat{\eta}^{T}_{T-t+s}}
\bigl(X_{s},\Hat{v}^{T}_{T-t+s}(X_{s})\bigr)\,\D{s}
+\varphi_{0}(X_{t})\right] -\Bar{\varrho}\,t
\qquad\forall\,t\in[0,T]\,,
\end{align}
and in general, for any $[t_{1},t_{2}]\in[0,T]$,
\begin{equation*}
\varphi^T_{t_{2}}(x) \;=\; \Exp^{\Hat{v}^{T}}_{x}
\left[\int_{0}^{t_{2}-t_{1}}r_{\Hat{\eta}^{T}_{T-t_{2}+s}}
\bigl(X_{s},\Hat{v}^{T}_{T-t_{2}+s}(X_{s})\bigr)\,\D{s}
+\varphi^T_{t_{1}}(X_{t_{2}-t_{1}})\right] -\Bar{\varrho}\,(t_{2}-t_{1})\,.
\end{equation*}
We also consider the following variation of \eqref{VI}:
\begin{equation}\label{RVI}
\partial_{t}\psi^T_{t}(x) \;=\; a^{ij}(x)\partial_{ij} \psi^T_{t}(x) + 
H(x,\nabla\psi^T_{t}(x))\\[5pt] + F(x,\Hat{\eta}^{T}_{T-t})-\psi^T_{t}(0)\,,
\quad \psi^{T}_{0}(x) = \varphi_{0}(x)\,.
\end{equation}

It is straightforward to show that $\Hat{v}_{t}$ is also
a measurable selector from the minimizer of the Hamiltonian in \eqref{RVI}, and that
$\varphi^T$ and $\psi^{T}$ are related by
\begin{equation*}
\varphi^T_{t}(x) \;=\; \psi^T_{t}(x) -\Bar{\varrho}\,t
+ \int_{0}^{t} \psi^T_s(0)\,\D{s}
\,,\qquad (t,x)\in[0,T]\times\RR^{d}\,.
\end{equation*}
We have in particular that
\begin{equation}\label{viloc}
\varphi^T_{t}(x) - \varphi^T_{t}(0) \;=\; \psi^T_{t}(x) - \psi^T_{t}(0)
\,,\qquad (t,x)\in[0,T]\times\RR^{d}\,.
\end{equation}
Conversely, if $\varphi^T$ is a solution of \eqref{VI}, then
one obtains a corresponding solution of \eqref{RVI}
that takes the form \cite[Lemma~4.4]{RVI}:
\begin{equation}\label{virvi}
\psi^T_{t}(x) \;=\; \varphi^T_{t}(x) - \int_{0}^{t}\E^{s-t}\,\varphi^T_s(0)\,\D{s} +
\Bar{\varrho}\,(1-\E^{-t})\,,\qquad (t,x)\in[0,T]\times\RR^{d}\,.
\end{equation}
We refer to \eqref{VI}, and \eqref{RVI} as the \emph{value iteration} (VI),
and \emph{relative value iteration} (RVI) equations, respectively.

The following technique is rather standard.
For $\eta\in\Pm(\Rd)$ and $v\in\Usm$ we define
$$\Bar F(\eta,\mu) \;\df\; \int F(x,\mu)\,\eta(\D{x})\,,\quad\text{and}\quad
\Bar{\mathring{r}}(\eta,v) \;\df\; \int \mathring{r}\bigl(x,v(x)\bigr)\,\eta(\D{x})$$
for $\eta\in\Pm(\Rd)$.
We consider $X$ in \eqref{E-sde}
under the following Markov controls:
$\Hat{v}_{t}$ which a measurable selector from the minimizer in \eqref{VI},
and the stationary control $\Bar{v}$ which corresponds to \eqref{T4.2a}.
Applying \eqref{E-VISR} and integrating with
respect to $\Hat{\eta}^{T}_{0}$ and $\Bar\mu$, respectively, we obtain
\begin{align}
\Hat{\eta}^{T}_{0}(\varphi^T_{T})
&\;=\; \int_{0}^{T} \bigl(\Bar{\mathring{r}}(\Hat{\eta}^{T}_{t},\Hat{v}_{t}) +
\Bar{F}(\Hat{\eta}^{T}_{t},\Hat{\eta}^{T}_{t})-\Bar{\varrho}\bigr)\,\D{t}
+\Hat{\eta}^{T}_{T}(\varphi_{0})\,,\label{EcompA}\\[5pt]
\Bar\mu(\varphi^T_{T})
&\;\le\; \int_{0}^{T}
\bigl(\Bar{\mathring{r}}(\Bar\mu,\Bar{v}) +
\Bar{F}(\Bar\mu,\Hat{\eta}^{T}_{t})-\Bar{\varrho}\bigr)\,\D{t}
+\Bar{\mu}(\varphi_{0})\,.\label{EcompB}
\end{align}
Repeating this with terminal cost $\Bar{V}$, and using
\eqref{T4.2a}, we obtain
\begin{align}
\Bar{\mu}(\Bar{V}) &\;=\; \int_{0}^{T} \bigl(\Bar{\mathring{r}}(\Bar{\mu},\Bar{v})
 +\Bar{F}(\Bar{\mu},\Bar{\mu})-\Bar{\varrho}\bigr)\,\D{t}
+ \Bar{\mu}(\Bar{V})\,,\label{EcompC}\\[5pt]
\Hat{\eta}^{T}_{0}(\Bar{V})
&\;\le\; \int_{0}^{T} \bigl(\Bar{\mathring{r}}(\Hat{\eta}^{T}_{t},\Hat{v}_{t}) +
\Bar{F}(\Hat{\eta}^{T}_{t},\Bar{\mu})-\Bar{\varrho}\bigr)\,\D{t}
 + \Hat{\eta}^{T}_{T}(\Bar{V})\,.\label{EcompD}
\end{align}
Adding together \eqref{EcompA}--\eqref{EcompC} and subtracting
\eqref{EcompB}--\eqref{EcompD} we obtain
\begin{equation}\label{Ecomp1}
\int_{0}^{T} \mathscr{F}(\Hat{\eta}^{T}_{t},\Bar{\mu})\,\D{t}\;\le\;
(\Hat{\eta}^{T}_{0}-\Bar\mu)(\varphi^T_T-\Bar V)
- (\Hat{\eta}^{T}_{T}-\Bar\mu)(\varphi_0-\Bar V)\,.
\end{equation}

Define
\begin{equation*}
\varGamma_{T}(t) \;\df\;
\bigl(\Hat\eta^{T}_{T-t}-\Bar\mu\bigr)
\bigl(\varphi^T_{t}-\Bar{V}\bigr)\,,\quad t\in[0,T]\,.
\end{equation*}
In complete analogy to \eqref{Ecomp1} we have
\begin{equation}\label{EcompX}
\int_{t_{1}}^{t_{2}}\mathscr{F}(\Hat{\eta}^{T}_{T-s},\Bar{\mu})\,\D{s}
\;\le\; \varGamma_{T}(t_2)-\Gamma_{T}(t_1)\,,\qquad t_1\le t_2\,. 
\end{equation}

\begin{remark}\label{R4.1}
We often use in the proofs the following fact:
if
$f_{t},\,h_{t}:\Rd\to\RR$ and $g:\Rd\to\RR$ are such
that $\sup_{t\ge0}\,\norm{f_{t}}_{\Lyap}<\infty$, $\norm{g}_{\Lyap}<\infty$,
and
$h_{t}(x)
= \Exp^{\Bar{v}}_{x}\bigl[\int_{0}^{t} f_{s}(X_{s})\,\D{s} + g(X_{t})\bigr]$,
then it holds that $\sup_{t>0}\,\norm{h_{t}(x)-h_{t}(0)}_{\Lyap}<\infty$.
Indeed, by \eqref{E-geom} we have
\begin{equation*}
\babss{h_{t}(x) - \int_{0}^{t} \Bar{\mu}(f_{s})\,\D{s} - \Bar{\mu}(g)}
\;\le\; \bigl(1+\Lyap(x))
\biggl(\int_{0}^{t}M_{0}\,\E^{-\gamma s}\norm{f_{s}}_{\Lyap}\,\D{s}
+ \norm{g}_{\Lyap}\biggr)\,,
\end{equation*}
so that
\begin{align*}
\abs{h_{t}(x)-h_{t}(0)}&\;\le\;
\bigl(2+\Lyap(x)+\Lyap(0))
\biggl(\int_{0}^{t}M_{0}\,\E^{-\gamma s}\norm{f_{s}}_{\Lyap}\,\D{s}
+ \E^{-\gamma t}\,\norm{g}_{\Lyap}\biggr)\\[5pt]
&\;\le\;
\bigl(2+\Lyap(x)+\Lyap(0))
\biggl(\frac{M_{0}}{\gamma}\,\sup_{s\ge0}\,\norm{f_{s}}_{\Lyap}
+ \E^{-\gamma t}\,\norm{g}_{\Lyap}\biggr)\,.
\end{align*}
\end{remark}

We start with the following result.

\begin{theorem}\label{T-basic}
Let Assumptions~\ref{A4.1}--\ref{A4.2} hold.
Then, for all $\lambda\in[0,1)$, we have
\begin{equation*}
\frac{1}{(1-\lambda) T}\;\int_{\lambda T}^{T}\Hat{\eta}^{T}_{t}\,\D{t}
\;\xrightarrow[T\to\infty]{}\;\Bar\mu\quad\text{in~} \Pm(\Rd)\,.
\end{equation*}
Moreover,
\begin{equation*}
\sup_{T>0}\;
\int_{0}^{T} \mathscr{F}(\Hat{\eta}^{T}_{t},\Bar{\mu})\,\D{t}\;<\; \infty\,.
\end{equation*}
\end{theorem}

\begin{proof}
Let
\begin{align*}
\Hat{g}^{T}_{t}(x)&\;\df\; \Exp^{\Hat{v}^{T}}_{x}
\biggl[\int_{0}^{t}r_{\Bar\mu}\bigl(X_{s},\Hat{v}^{T}_{T-t+s}(X_{s})\bigr)\,\D{s}
+\Bar{V}(X_{t})\biggr]-\Bar{\varrho}\,t\,,\\[5pt]
\Bar{g}^{T}_{t}(x)&\;\df\; \Exp^{\Bar{v}}_{x}
\biggl[\int_{0}^{t}r_{\Hat{\eta}^{T}_{T-t+s}}\bigl(X_{s},\Bar{v}(X_{s})\bigr)\,\D{s}
+\varphi_{0}(X_{t})\biggr]-\Bar{\varrho}\,t\,.
\end{align*}
Since
\begin{equation*}
\Bar{g}^{T}_{t}(x)-\Bar{V}(x)
\;=\;\Exp^{\Bar{v}}_{x}
\biggl[\int_{0}^{t}\bigl(F(X_{s},\Hat{\eta}^{T}_{T-t+s})
-F(X_{s},\Bar\mu)\bigr)\,\D{s} + \varphi_{0}(X_{t}) - \Bar{V}(X_{t})\biggr]\,,
\end{equation*}
it follows by \eqref{E-geom} and Remark~\ref{R4.1} that
$$\sup_{T>0}\;\sup_{t\,\in\,[0,T]}\;
\bnorm{\Bar{g}^{T}_{t}(x)-\Bar{V}(x)-\Bar{g}^{T}_{t}(0)-\Bar{V}(0)}_{\Lyap}
\;<\;\infty\,.$$
Therefore, for some constant $C$, we have
$\abs{(\Hat{\eta}^{T}_{T-t}-\Bar\mu)(\Bar{g}^{T}_{t}-\Bar{V})}< C$
for all $t\in[0,T]$ and $T>0$.
Hence, by \eqref{E-mon}, we have
\begin{align}\label{T-basicB}
\Hat{\eta}^{T}_{T-t}(\varphi^{T}_{t}) &\;\ge\;
\Hat{\eta}^{T}_{T-t}(\Hat{g}^{T}_{t}) +\Bar{\mu}(\Bar{g}^{T}_{t}-\Bar{V})
+(\Hat\eta^{T}_{T}-\Bar\mu)(\varphi_{0}-\Bar{V})
\nonumber\\[5pt]
&\;\ge\;\Hat{\eta}^{T}_{T-t}(\Hat{g}^{T}_{t})
+ \Hat{\eta}^{T}_{T-t}(\Bar{g}^{T}_{t}-\Bar{V})-C
+(\Hat\eta^{T}_{T}-\Bar\mu)(\varphi_{0}-\Bar{V})
\end{align}
for all $t\in[0,T]$ and $T>0$.
By suboptimality $\Hat{g}^{T}_{t}\ge \Bar{V}$ and $\Bar{g}^{T}_{t}\ge\varphi^{T}_{t}$.
Also
$$\abs{(\Hat\eta^{T}_{T}-\Bar\mu)(\varphi_{0}-\Bar{V})}
\;\le\; 2\frac{c_{0}}{c_{1}} +\norm{\varphi_{0}-\Bar{V}}_{\Lyap}
\bigl(\Hat\eta_{0}(\Lyap)+\Bar\mu(\Lyap)\bigr)\,.
$$
Hence, for some constant $C'$ we obtain
\begin{equation}\label{TbasicX}
\begin{split}
0&\;\le\;\Hat{\eta}^{T}_{T-t}(\Hat{g}^{T}_{t}-\Bar{V})\;\le\;C'\,,\\[5pt]
0&\;\le\;\Hat{\eta}^{T}_{T-t}(\Bar{g}^{T}_t-\varphi^{T}_{t})\;\le\; C'
\end{split}
\end{equation}
for all $t\in[0,T]$ and $T>0$.
From the first equation in \eqref{TbasicX}, and since $\Hat\eta^{T}_{T}(\Bar{V})$,
$\Hat\eta^{T}_{T}(\varphi_{0})$,
and $\int_{0}^{T} \bigl(\Bar{F}(\Law^{\Bar{v}}_{\Hat\eta_{0}}(X_{t}),\Bar\mu)
-\Bar{F}(\Bar\mu,\Bar\mu)\bigr)\,\D{t}$
are bounded uniformly in $T>0$, using a triangle inequality we obtain
\begin{equation*}
\sup_{T>0}\;
\babss{\int_{0}^{T} \bigl(\Bar{F}(\Hat\eta^T_t,\Bar\mu)
-\Bar{F}(\Bar\mu,\Bar\mu)\bigr)\,\D{t}\,}\;<\; C''
\end{equation*}
for some constant $C''$.
Similarly from the second equation, using the same constant $C''$, without
loss of generality, we have
\begin{equation*}
\sup_{T>0}\;
\babss{\int_{0}^{T} \bigl(\Bar{F}(\Bar\mu,\Hat\eta^T_t)
-\Bar{F}(\Hat\eta^T_t,\Hat\eta^T_t)\bigr)\,
\D{t}\,}\;<\; C''\,.
\end{equation*}
The second assertion of the theorem follows from these bounds.

From the first inequality in \eqref{TbasicX}
we obtain $0\le\Hat{\eta}^{T}_{(1-\lambda)T}(\Hat{g}^{T}_{\lambda T}-\Bar{V})\le C$
for all $\lambda\in[0,1]$.
Therefore, we have
\begin{equation}\label{T-basicD}
\frac{1}{\lambda T}\;\Hat{\eta}^{T}_{(1-\lambda)T}(\Hat{g}^{T}_{\lambda T})
\;\xrightarrow[T\to\infty]{}\;0\qquad\forall\,\lambda\in(0,1]\,.
\end{equation}
Let
\begin{equation*}
\Breve{\uppi}^{T}_{\lambda}(\D{x},\D{u})\;\df\;
\frac{1}{\lambda T}\int_{0}^{\lambda T}
\Hat{\eta}^{T}_{(1-\lambda)T+s}(\D{x})\cast\Hat{v}^{T}_{(1-\lambda)T+s}(\D{u}\mid x)
\,\D{s}\,,
\quad\lambda\in(0,1]\,,
\end{equation*}
and $\Bar\uppi \df \Bar\mu\cast\Bar{v}$.
Since $\Bar\uppi(r_{\Bar{\mu}})=\Bar{\varrho}$,
and write \eqref{T-basicD} as
\begin{equation*}
0\;\ge\;\int_{\Rd\times\Act}
\Bigl(\Breve{\uppi}^{T}_{\lambda}(\D{x},\D{u})
- \Bar\uppi(\D{x},\D{u})\Bigr)\, r_{\Bar\mu}(x,u)
\;\xrightarrow[T\to\infty]{}\;0\,.
\end{equation*}
Since $\{\Breve{\uppi}^{T}_{\lambda}\,,\;T>0\}$ is tight,
any limit point of $\Breve{\uppi}^{T}_{\lambda}$ as $T\to\infty$
is an element of $\eG$ \cite[Lemma~3.4.6]{ari-bor-ghosh}.
Let $\{T_{n}\}$ be any sequence, and select a subsequence also denoted
as $\{T_n\}$ along which
$\Breve{\uppi}^{T_{n}}_{\lambda}\to\Breve{\uppi}^{*}\in\eG$.
Then $\Breve{\uppi}^{*}(r_{\Bar\mu})=\Bar\uppi(r_{\Bar\mu})$,
and since by Assumption~\ref{A4.1}\,(iv) the set $\cA(\Bar\mu)$ is a singleton
it follows that $\Breve{\uppi}^{*}=\Bar\uppi$ which, in turn, implies
the first assertion in the theorem.
\end{proof}

We also have the following simple lemma concerning the growth
of $\varphi^T_{t}(0)$ in $t$.

\begin{lemma}\label{L-aux}
Let Assumptions~\ref{A4.1}--\ref{A4.2} hold.
Then there exists a constant $\Bar{C}_{0}>0$ which depends
only on $\varphi_{0}$ and $\Law(X_{0})$, such that
\begin{equation*}
\babs{\varphi^T_{t}(0)-\varphi^T_{t-\tau}(0)}\;\le\;
\Bar{C}_{0} (1+\tau)\qquad \text{for all}~
t\in[0,T]\,,~\tau\in[0,t]\,,~\text{and}~T>0\,.
\end{equation*}
\end{lemma}

\begin{proof}
By Assumptions~\ref{A4.1} and \ref{A4.2},
there is a constant $\Bar{C}$ depending only on $\varphi_{0}$ and $\Law(X_{0})$,
such that
\begin{equation*}
\Exp^{U}_{0}\,\babs{r_{\Hat{\eta}^{T}_{T-t}}(X_{s},U_{s})}
\;\le\; \Bar{C}
\end{equation*}
for all $t\in[0,T]$, $s\ge0$, $U\in\Uadm$, and $T>0$.
Therefore, we have
\begin{align*}
\babs{\varphi^T_{t-\tau}(0)-\varphi^T_{t}(0)}
&\;=\; \babss{\inf_{U}\;\Exp^{U}_{0}
\biggl[\int_{0}^{t-\tau}r_{\Hat{\eta}^{T}_{T-t+\tau+s}}(X_{s},U_{s})\,\D{s}
+\varphi_{0}(X_{t-\tau})\biggr]\\[5pt]
&\mspace{100mu}
- \inf_{U}\;\Exp^{U}_{0}
\biggl[\int_{0}^{t}r_{\Hat{\eta}^{T}_{T-t+s}}(X_{s},U_{s})\,\D{s}
+\varphi_{0}(X_{t})\biggr]}
\nonumber\\[5pt]
&\;\le\;
\sup_{U}\;\Exp^{U}_{0}
\biggl[\int_{t-\tau}^{t}\abs{r_{\Hat{\eta}^{T}_{T-t+s}}(X_{s},U_{s})}\,\D{s}
+\babs{\varphi_{0}(X_{t-\tau})-\varphi_{0}(X_{t})}\biggr]
\nonumber\\[5pt]
&\;\le\; \Bar{C} \tau + 2\norm{\varphi_{0}}_{\Lyap}
\bigl(\tfrac{c_{0}}{c_{1}} + \E^{-c_{1}t}\Lyap(0)\bigr)\,.\qedhere
\end{align*}
\end{proof}

\subsection{Convergence of the RVI}

Theorem~\ref{T-basic} shows that $\Hat{\eta}^{T}$ converges to $\Bar{\mu}$
in a time average sense.
We wish to show that $\psi^T_{\lambda T}$ converges to $\Bar{V}-\Bar{\varrho}$
as $T\to\infty$.
It is evident by \eqref{viloc} that this cannot happen unless
$\varphi^T_{\lambda T} -\varphi^T_{\lambda T}(0)$ is at least locally bounded,
uniformly in $T>0$.
We state this necessary condition for convergence as a property.

\begin{property}\label{P4.1}
Define $\overline\varphi^{T}_{t}\df\varphi^{T}_{t}-\varphi^{T}_{t}(0)$.
Suppose that $\Hat{\eta}_{0}(\Lyap)\le \kappa_{1}$ and
$\norm{\varphi_{0}}_{\Lyap}\le\kappa_2$,
where $\Lyap$ is as in Assumption~\ref{A4.2}.
Then there exists a constant $\Tilde{C}_{1}=\Tilde{C}_{1}(\kappa_{1},\kappa_{2})$
such that
\begin{equation*}
\sup_{T>0}\;\sup_{t\,\in\,[0,T]}\;
\bnorm{\overline\varphi^T_{t}}_{\Lyap}
\;<\;\Tilde{C}_{1}\,.
\end{equation*}
\end{property}

It is unclear if Assumption~\ref{A4.2} suffices to
establish Property~\ref{P4.1}.
Instead of imposing additional assumptions on $F$, we choose instead
to show that this property is satisfied
for a large class of controlled diffusions, and then assume only
Property~\ref{P4.1} in the statement of the main results.

We introduce the following notation: 
for $x$, $z$ in $\RR^{d}$ define
\begin{align*}
\Delta_{z}b(x,u)&\;\df\; b(x+z,u)-b(x,u)\,,\\[2pt]
\Delta_{z}\upsigma(x)&\;\df\;\upsigma(x+z)-\upsigma(x)\,,\\[2pt]
\Tilde{a}(x;z) &\;\df\; \Delta_{z}\upsigma(x) 
\Delta_{z}\upsigma\transp(x)\,.
\end{align*}

\begin{definition}
We say that the controlled diffusion in \eqref{E-sde} is
\emph{asymptotically flat} if the following hold:
\begin{itemize}
\item[(a)]
The diffusion matrix $\upsigma$ is Lipschitz continuous.
\item[(b)]
There exist a symmetric positive definite matrix $Q$ and a
constant $r > 0$ such that for $x,z \in \RR^{d}$, with $z\ne0$,
and $u \in \Act$, it holds that
\begin{equation*}
2 \Delta_{z}b\transp(x,u) Q z -
\frac{\abs{\Delta_{z}\upsigma\transp(x) Qz}^{2}}{z\transp Q z}
+ \trace \bigl(\Tilde{a}(x;z) Q\bigr)
\;\le\; -r \abs{z}^{2}\,.
\end{equation*}
\end{itemize}
\end{definition}

A standard model of asymptotically flat diffusions is
given by $\Act = [0, 1]^{d}$, $b(x,u) = Bx + Du$, where
$B$, $D$ are
constant $d \times d$ matrices and $B$ is Hurwitz (i.e., its eigenvalues
have negative real parts).
Note also that if $\upsigma$ is constant, then asymptotic flatness amounts
to the requirement that
$\langle b(x+z,u)-b(x,u), Qz\rangle \le -r \abs{z}^{2}$.
Nevertheless, the class of asymptotically flat diffusions is significantly richer
than models with stable linear drifts.
Asymptotically flat diffusions satisfy an ``incremental stability'' property.
For recent work along similar directions see \cite{Caraballo-04,Pham-09}.

We quote the following result \cite[Lemmas~7.3.4 and~7.3.6]{ari-bor-ghosh}.

\begin{lemma}\label{Ln4.3}
Suppose that the diffusion in \eqref{E-sde} is asymptotically flat,
and let $X^{x}_{t}$ be the solution with initial condition
$X_{0} = x$, corresponding to an admissible relaxed control $U$.
Then there exist constants
$\Hat{c}_{0} > 0$ and $\Hat{c}_{1} > 0$, which do not depend on $U$, such that
\begin{equation}\label{Ln4.3a}
\Exp^{U}\,\babs{X^{x}_{t} - X^{y}_{t}}\;\le\; \Hat{c}_{0}
\E^{-\Hat{c}_{1}t} \abs{x-y} \qquad
\forall x,\,y\in\RR^{d}\,.
\end{equation}
Moreover there exists a nonnegative, inf-compact $\Lyap\in\Cc^{2}(\Rd)$
satisfying \eqref{E-Lyap}, and such that
\begin{equation}\label{Ln4.3b}
\frac{\Lyap(x)}{1+\abs{x}^{p_{0}}}
\;\xrightarrow[\abs{x}\to\infty]{}\;\infty
\end{equation}
for some $p_{0}>1$.
\end{lemma}

Let $\Lip(f)$ denote a Lipschitz constant for a function $f$,
which is assumed Lipschitz.
We often use the fact that for an asymptotically flat diffusion,
if $g_{t}(x) = \Exp_{x}^{U}[f(X_{t}]$, then
$\Lip(g_{t})\le \Hat{c}_{0}
\E^{-\Hat{c}_{1}t}\,\Lip(f)$ for all $U\in\Uadm$.

We have the following lemma.

\begin{lemma}\label{Ln4.4}
Let Assumptions~\ref{A4.1} and~\ref{A4.2} hold,
and suppose that the diffusion is asymptotically flat,
$\Hat{\eta}_{0}(\Lyap)<\infty$, and $\varphi_{0}\in\Cc^{2}(\Rd)$ is Lipschitz.
We also assume that $\mathring{r}(\cdot\,,u)$ is
Lipschitz uniformly in $u\in\Act$, and that
Assumption~\ref{A4.2}\,(ii) holds for a constant
$M_{1}$ which is independent of $R$.
Then there exists a constant $\Tilde{C}_{1}'$
which depends only on $\Hat{\eta}_{0}(\Lyap)$ and $\Lip(\varphi_{0})$
such that
\begin{align*}
\Lip\bigl(\varphi^T_t) &\;\le\; \Tilde{C}_{1}'\qquad\forall t\in[0,T]\,,~
\forall\,T>0\,.
\intertext{In particular,}
\sup_{T>0}\;\sup_{t\in[0,T]}\;
\bnorm{\overline\varphi^T_t}_{\Lyap}&\;\le\; \Tilde{C}_{1}'\,,
\end{align*}
and thus Property~\ref{P4.1} holds.
\end{lemma}

\begin{proof}
We fix some
compact set $\calK_{0}\subset\Pm_{1}(\Rd)$ of initial distributions that
contains $\Bar\mu$, and satisfies $\sup_{\mu\in\calK_{0}}\,\mu(\Lyap)<\infty$.
The initial distribution
$\Hat\eta_{0}$ is assumed to lie in the set $\calK_{0}$.
The corresponding collection
$\mathscr{P}(\calK_{0})
\df\{\Law^{U}_{\mu}(X_{t})\,\colon\,\mu\in\calK_{0}\,,\;U\in\Uadm\,,
\;t>0\}$
is compact in $\Pm_{1}(\Rd)$ by \eqref{E-Lyap} and \eqref{Ln4.3b}.
Therefore, for some constant $\Tilde{C}_{0}$, it holds that
$\Lip(F(\cdot,\mu))\le \Tilde{C}_{0}$ and
$\abs{F(x,\mu)}\le \Tilde{C}_{0}(1+\abs{x})$
for all $\mu\in\calK_{0}$.
It is straightforward to show that, under asymptotic flatness,
$\Bar{V}$ is Lipschitz.
Without loss of generality,
we let $\Tilde{C}_{0}$ be also a
Lipschitz constant for $\varphi_{0}$, $\Bar{V}$, and $\mathring r(\cdot,u)$
as well.

By Lemma~\ref{Ln4.3}, we have
\begin{align*}
\babs{\varphi^T_{t}(x)-\varphi^T_{t}(y)} &\;=\;
\babss{\inf_{U}\;\Exp^{U}_{x}
\biggl[\int_{0}^{t}r_{\Hat{\eta}^{T}_{T-t+s}}(X_{s},U_{s})\,\D{s}
+\varphi_{0}(X_{t})\biggr]\nonumber\\[5pt]
&\mspace{130mu}-\inf_{U}\;\Exp^{U}_{y}
\biggl[\int_{0}^{t}r_{\Hat{\eta}^{T}_{T-t+s}}(X_{s},U_{s})\,\D{s}
+\varphi_{0}(X_{t})\biggr]}\nonumber\\[5pt]
&\;\le\;
\sup_{U}\;\Exp^{U}
\biggl[\int_{0}^{t}\babs{r_{\Hat{\eta}^{T}_{T-t+s}}(X^{x}_{s},U_{s})
-r_{\Hat{\eta}^{T}_{T-t+s}}(X^{y}_{s},U_{s})}\,\D{s}\biggr]\nonumber\\[5pt]
&\mspace{200mu}+\sup_{U}\;\Exp^{U}
\Bigl[\babs{\varphi_{0}(X^{x}_{t})-\varphi_{0}(X^{y}_{t})}\Bigr]\nonumber\\
&\;\le\;
2\Tilde{C}_{0}\frac{\Hat{c}_{0}}{\Hat{c}_{1}}\,\abs{x-y}
+ \Hat{c}_{0}\abs{x-y}\qquad\forall\,x,y\in\Rd\,,~t\in[0,T]\,,~\text{and}~
T>0\,.\qedhere
\end{align*}
\end{proof}

\begin{remark}
Even though $\mathring r$ and $F$ have been assumed Lipschitz in $x$,
running costs with higher growth in $x$ can be treated, depending on the
 diffusion matrix.
In particular, if the diffusion matrix is constant, then \eqref{Ln4.3a}
can be replaced by
\begin{equation*}
\Exp^{U}\,\babs{X^{x}_{t} - X^{y}_{t}}^{2}\;\le\; \Hat{c}_{0}
\E^{-\Hat{c}_{1}t} \abs{x-y}^{2} \qquad
\forall x,\,y\in\RR^{d}\,.
\end{equation*}
Thus, in this case, the results of Lemma~\ref{Ln4.4}
can be extended to include running costs
with up to quadratic growth.
\end{remark}

As mentioned earlier, Property~\ref{P4.1}, which is
implied by asymptotic flatness, together with Assumption~\ref{A4.2}
are sufficient to prove convergence.
So in the statement of the main results we use Property~\ref{P4.1} in lieu
of asymptotic flatness.

Since $(\Hat{\eta}^{T}_{0}-\Bar\mu)(\varphi^T_T-\Bar V)
=(\Hat{\eta}^{T}_{0}-\Bar\mu)(\overline\varphi^T_T-\Bar V)$,
it is evident that if Property~\ref{P4.1} holds, then
the right hand side of \eqref{Ecomp1} is bounded uniformly in $T>0$.
Therefore, by \eqref{EcompX}, we have the following.

\begin{corollary}\label{Cn4.1}
Let Assumptions~\ref{A4.1}--\ref{A4.2} and Property~\ref{A4.1} hold,
and suppose $\Hat{\eta}_{0}(\Lyap)<\infty$ and
$\varphi_{0}\in\Cc^{2}(\Rd)\cap\Cc_{\Lyap}(\Rd)$.
Then there exists a constant $C_{0}$
such that
\begin{equation*}
\int_{t_{1}}^{t_{2}}\mathscr{F}(\Hat{\eta}^{T}_{T-s},\Bar{\mu})\,\D{s}
\;\le\; \varGamma_{T}(t_2)-\Gamma_{T}(t_1)
\;\le\; C_{0}\,,\qquad t_1\le t_2\,.
\end{equation*}
In particular, $t\mapsto \varGamma_{T}(t)$ is nondecreasing and bounded
on $t\in[0,T]$ uniformly in $T>0$.
\end{corollary}

We are now ready to state the main results.

\begin{theorem}\label{Tn4.3}
Let Assumptions~\ref{A4.1}--\ref{A4.2} and Property~\ref{A4.1} hold,
and suppose $\Hat{\eta}_{0}(\Lyap)<\infty$ and
$\varphi_{0}\in\Cc^{2}(\Rd)\cap\Cc_{\Lyap}(\Rd)$.
Then for any $\lambda\in(0,1)$, and $t_{0}>0$ we have
\begin{equation}\label{Tn4.3aA}
\sup_{t\in[0,t_{0}]}\;
\eD_{1}(\Hat{\eta}^{T}_{\lambda T - t},\Bar\mu)\;\xrightarrow[T\to\infty]{}\;0\,.
\end{equation}
Moreover,
\begin{equation}\label{Tn4.3a}
\sup_{t\in[0,t_{0}]}\;
\bnorm{\varphi^T_{\lambda T-t} - \varphi^T_{\lambda T-t}(0)-\Bar{V}}_{\Lyap}
\;\xrightarrow[T\to\infty]{}\;0\,,
\end{equation}
and
\begin{equation}\label{Tn4.3b}
\sup_{t\in[0,t_{0}]}\,\babs{\varphi^T_{\lambda T}(0)-\varphi^T_{\lambda T-t}(0)}
\;\xrightarrow[T\to\infty]{}\;0\,.
\end{equation}
\end{theorem}

\begin{proof}
Let $\varepsilon\in\bigl(0,\frac{1}{2}\min(\lambda,1-\lambda)\bigr)$, and
$\tau>0$.
Consider the interval $[(1-\varepsilon) T-\tau,T]$ and let
$\mathcal{I}_{T}$ be the collection of consecutive closed intervals
$[(1-\varepsilon) T,(1-\varepsilon) T+2\tau]$,
$[(1-\varepsilon) T+2\tau,(1-\varepsilon) T+4\tau],\dotsc$ contained in it.
Let $T_n\to\infty$ be any sequence.
By Corollary~\ref{Cn4.1} there exists a sequence
$[t_n-\tau,t_n+\tau]\in \mathcal{I}_{T_n}$
such that $\varGamma_{T_n}(t_n+\tau)-\varGamma_{T_n}(t_n-\tau)\to 0$ as $n\to\infty$.

Property~\ref{P4.1} together with Lemma~\ref{L-aux} imply that
$(s,x)\mapsto\varphi^{T}_{t+s}(x)-\varphi^{T}_{t}(0)$ is bounded on
compact sets of $\RR_{+}\times \Rd$ uniformly in $T>0$ and $t\in[0,T]$.
By well known interior estimates of parabolic equations, this means
that the maps $(s,x)\mapsto\varphi^{T}_{t+s}(x)-\varphi^{T}_{t}(0)$ are
locally H\"older equicontinuous on
compact sets of $(1,\infty)\times\Rd$.
Therefore,
$(t,x)\mapsto\varphi^{T_{n}}_{t}(x)-\varphi^{T_{n}}_{t_{n}}(0)$ is equicontinuous
on $[t_n - \tau, t_{n}+\tau]$.
At the same time, by Lemma~\ref{L4.1}, the laws
$\{\Hat{\eta}^{T_n}_{T_n-t_n+s}\,,\; s\in[-\tau,\tau]\}$
are precompact in $[-\tau,\tau]\times\Pm_{1}(\Rd)$.
Passing to the limit along a subsequence, also denoted as $\{T_n\}$,
we define
\begin{equation}\label{Tn4.3BB}
\varphi^{*}_{t}\;\df\;\lim_{n\to\infty}\;
\bigl(\varphi^{T_{n}}_{t_{n}-\tau+t}-\varphi^{T_{n}}_{t_{n}}(0)\bigr)\,,
\quad\text{and}\quad
\Hat\eta^{*}_{t}\;\df\;\lim_{n\to\infty}\;\Hat\eta^{T_{n}}_{T_{n} -t_{n}-\tau+t}\,,
\qquad t\in[0,2\tau]\,.
\end{equation}
Let $\overline\varphi^{*}_{t}\df\varphi^{*}_{t}-\varphi^{*}_{t}(0)$.
By Property~\ref{P4.1} and Lemma~\ref{L-aux} we have
\begin{equation}\label{Tn4.3BC}
\sup_{t\,\in\,[0,2\tau]}\,\norm{\varphi^{*}_{t}}_\Lyap\;<\;
\Tilde{C}_{1}+\Bar{C}_{0}(1+\tau)\,,\quad\text{and}\quad
\sup_{t\,\in\,[0,2\tau]}\,\norm{\overline\varphi^{*}_{t}}_\Lyap\;<\;
\Tilde{C}_{1}\,.
\end{equation}
It is evident that $\{\Hat\eta^{*}_{t}\,,\;t\in[0,2\tau]\}$ is a MFG solution
for the finite horizon problem on $[0,2\tau]$
with initial law $\Hat\eta^{*}_{0}$ and terminal cost $\varphi^{*}_{0}$.
Therefore,
\begin{equation*}
\varphi^*_{t}(x)
\;=\;\inf_{U\in\Uadm}\;\Exp^{U}_{x}
\biggl[\int_{0}^{t}r_{\Hat{\eta}^{*}_{s}}\bigl(X_{s}, U_s\bigr)\,\D{s}
+\varphi^*_{0}(X_{t})\biggr] - \bar{\varrho}\,t\,.
\end{equation*}
Let $\Hat{v}^{*}$ be a Markov control that realizes this infimum, i.e.,
$\Hat{v}^{*}$ is the a.e.\ unique minimizer from the Hamiltonian of
the associated HJB.
By suboptimality we have
\begin{subequations}
\begin{align}
\varphi^*_{t}(x)
&\;\le\;\Exp^{\Bar{v}}_{x}
\biggl[\int_{0}^{t}r_{\Hat{\eta}^{*}_{s}}\bigl(X_{s},\Bar{v}(X_s)\bigr)\,\D{s}
+\varphi^*_{0}(X_{t})\biggr] - \bar{\varrho}\, t\label{Tn4.3c}\,,\\[5pt]
\Bar{V}(x)
&\;\le\;\Exp^{\Hat{v}^{*}}_{x}
\biggl[\int_{0}^{t}r_{\Bar{\mu}}\bigl(X_{s},\Hat{v}^{*}_{s}(X_s)\bigr)\,\D{s}
+\Bar{V}(X_{t})\biggr]- \bar{\varrho}\, t\label{Tn4.3d}
\end{align}
\end{subequations}
for $t\in[0,\tau]$.
Since $\varGamma_{T_n}(t_n+\tau)-\varGamma_{T_n}(t_n-\tau)\to 0$
as $n\to\infty$ along the subsequence,
taking limits we have
\begin{equation}\label{Tn4.3e}
(\Hat{\eta}^{*}_{0}-\Bar\mu)(\varphi^*_{2\tau}-\Bar V)
- (\Hat{\eta}^{*}_{2\tau}-\Bar\mu)(\varphi^*_0-\Bar V) \;=\; 0\,.
\end{equation}
Thus
\begin{equation}\label{Tn4.3f}
\int_{0}^{2\tau}\mathscr{F}(\Hat{\eta}^{*}_{s},\Bar{\mu})\,\D{s}\;=\;0
\end{equation}
by \eqref{Ecomp1}.
However, since $\Bar\mu$ and $\Hat{\eta}^{*}_{0}$ have strictly positive density,
then \eqref{Tn4.3e}--\eqref{Tn4.3f}  imply
that \eqref{Tn4.3c} and \eqref{Tn4.3d} must hold with equality.
By a.e.\ uniqueness of the minimizer in the Hamiltonian,
we must have $\Hat{v}^{*}=\Bar{v}$ a.e.\ in $[0,2\tau]\times\Rd$.
Recall that $P_{t}^{\Bar{v}}(x,\cdot\,)$ denotes the transition probability
of the process $X$ in \eqref{E-sde} under the control $\Bar{v}$.
Thus, by \eqref{E-geom} we have
\begin{equation*}
\Hat\eta^*_{t}(\cdot)\;=\;
\int_{\Rd}\Hat\eta^*_0(\D{y})\,P_{t}^{\Bar{v}}(y,\cdot\,)\,,
\quad t\in[0,2\tau]\,,
\end{equation*}
and using \eqref{E-geom} and \eqref{Tn4.3BC} we obtain
\begin{align}\label{Tn4.3g}
\babs{(\Hat{\eta}^{*}_{t}-\Bar\mu)(\varphi^*_0-\Bar V)}
&\;\le\; M_{0}\,\E^{-\gamma t}
\bnorm{\overline\varphi^*_0-\Bar V}_{\Lyap}
\bigl(1+\Hat{\eta}^{*}_{0}(\Lyap)\bigr)\nonumber\\[5pt]
&\;\le\; M_{0}\bigl(\Tilde{C}_{1}+
\norm{\Bar{V}}_{\Lyap}\bigr)\,\bigl(
1+\tfrac{c_{0}}{c_{1}}+\Hat{\eta}_{0}(\Lyap)\bigr)\,
\E^{-\gamma t}
\,,\quad t\in[0,2\tau]\,.
\end{align}
Note also that by the Kantorovich duality theorem,
we have the estimate
\begin{equation}\label{Tn4.3h}
\eD_{1}(\Hat\eta^*_{t},\Bar\mu)
\;\le\;
M_{0}\,\E^{-\gamma t}\bigl(1+\Hat{\eta}^{*}_{0}(\Lyap)\bigr)\,
\sup_{x\in\Rd}\;\tfrac{\abs{x}}{\Lyap(x)}\,,\quad t\in[0,2\tau]\,.
\end{equation}

We claim that
$\Hat{v}^{*}=\Bar{v}$ a.e.\ in $[0,2\tau]\times\Rd$ also implies that
\begin{align}
\sup_{t\,\in\,\bigl[\frac{\tau}{4},\frac{3\tau}{2}\bigr]}\;
\bnorm{\overline\varphi^*_{t}-\Bar{V}}_{\Lyap}
&\;\xrightarrow[\tau\to\infty]{}\;0\,,\label{Tn4.3ia}\\
\intertext{and}
\sup_{t\,\in\,\bigl[\frac{\tau}{2},\frac{3\tau}{2}\bigr]}\;
\babs{\varphi^*_{t}(0)-\varphi^*_{\tau}(0)}
&\;\xrightarrow[\tau\to\infty]{}\;0\label{Tn4.3ib}\,.
\end{align}
To prove the claim,  we estimate $\varphi^*_{t}$ by
\begin{align}\label{Tn4.3j}
\varphi^*_{t}(x) - \varphi^*_{0}(0)
&\;=\;\Exp^{\Bar{v}}_{x} \biggl[\int_{0}^{t}
r_{\Hat{\eta}^{*}_{2\tau-t+s}}\bigl(X_{s},\Bar{v}(X_s)\bigr)\,\D{s}
+\overline\varphi^*_{0}(X_{t})\biggr]
- \bar{\varrho}\,t\nonumber\\[5pt]
&\;=\;\Exp^{\Bar{v}}_{x}
\biggl[\int_{0}^{t}r_{\Bar\mu}\bigl(X_{s},\Bar{v}(X_s)\bigr)\,\D{s}
+\Bar{V}(X_{t}) - \Bar{\varrho}\,t\biggr]
\nonumber\\[5pt]
&\mspace{50mu}+\Exp^{\Bar{v}}_{x}\biggl[\int_{0}^{t}
\Bigl(F\bigl(X_{s},\Hat{\eta}^{*}_{2\tau-t+s}\bigr)
-F\bigl(X_{s},\Bar{\mu}\bigr)\Bigr)\,\D{s}
+\overline\varphi^*_{0}(X_{t}) - \Bar{V}(X_{t})\biggr]\,.
\end{align}
The first term in \eqref{Tn4.3j} equals $\Bar{V}(x)$.
We use the estimate
\begin{equation}\label{Tn4.3jA}
\babs{\Exp^{\Bar{v}}_{x} \bigl[\overline\varphi^*_{0}(X_{t}) -\Bar{V}(X_{t})\bigr]
- \Bar{\mu}\bigl(\overline\varphi^*_{0} -\Bar{V}\bigr)}
\;\le\;
M_{0} \E^{-\gamma t}
\bnorm{\overline\varphi^*_{0}-\Bar{V}}_{\Lyap} \bigl(1+\Lyap(x)\bigr)\,,
\end{equation}
which holds by \eqref{E-geom}.
Similarly, with $\Tilde{F}_{\mu}(x)\df
F(x,\mu)-F(x,\Bar\mu)$, we have
\begin{equation}\label{Tn4.3jB}
\babss{\,\Exp^{\Bar{v}}_{x}\biggl[\int_{0}^{t}
\Tilde{F}_{\Hat{\eta}^{*}_{2\tau-t+s}}(X_{s})\,\D{s}\biggr]
- \int_{0}^{t}
\Bar\mu\bigl(\Tilde{F}_{\Hat{\eta}^{*}_{2\tau-t+s}}\bigr)\,\D{s}}
\;\le\;
 M_{0} \bigl(1+\Lyap(x)\bigr)\int_{0}^{t}\E^{-\gamma s}
\bnorm{\Tilde{F}_{\Hat{\eta}^{*}_{2\tau-t+s}}}_{\Lyap} \,\D{s}\,.
\end{equation}
Let
\begin{equation*}
\zeta(t)\;\df\;
\Bar{\mu}\bigl(\overline\varphi^*_{0} -\Bar{V}\bigr)
+ \int_{0}^{t}
\Bar\mu\bigl(\Tilde{F}_{\Hat{\eta}^{*}_{2\tau-t+s}}\bigr)\,\D{s}\,.
\end{equation*}
We evaluate $\varphi^*_{t}(x) - \varphi^*_{0}(0)$ in
\eqref{Tn4.3j} first at $x$ and then at $x=0$,
using also \eqref{Tn4.3jA}--\eqref{Tn4.3jB}
to form a triangle inequality,
as well as the fact that $\Bar{V}(0)=0$,
to obtain
\begin{align}\label{Tn4.3jC}
\babs{\overline\varphi^*_{t}(x)-\Bar{V}(x)}
&\;\le\; \babs{\varphi^*_{t}(x)-\varphi^{*}_{0}(0)-\Bar{V}(x)
-\zeta(t)}
+ \babs{\varphi^*_{t}(0)-\varphi^{*}_{0}(0)-\Bar{V}(0)-\zeta(t)}\nonumber\\[5pt]
&\;\le\;
M_{0} \bigl(2+\Lyap(x)+\Lyap(0)\bigr)\,
\biggl(\E^{-\gamma t }\bnorm{\overline\varphi^*_{0}-\Bar{V}}_{\Lyap}
+ \int_{0}^{t}\E^{-\gamma s}
\bnorm{\Tilde{F}_{\Hat{\eta}^{*}_{2\tau-t+s}}}_{\Lyap} \,\D{s}\biggr)\,.
\end{align}
By Assumption~\ref{A4.2}\,(iii), which holds with $p=1$, and
\eqref{Tn4.3h} we have
\begin{equation}\label{Tn4.3jD}
\sup_{t\,\in\,[\nicefrac{\tau}{2}, 2\tau]}\;
\bnorm{\Tilde{F}_{\Hat{\eta}^{*}_{t}}}_{\Lyap}\;\xrightarrow[\tau\to\infty]{}\;0\,.
\end{equation}
Therefore \eqref{Tn4.3ia} follows by \eqref{Tn4.3jC}--\eqref{Tn4.3jD}.
Using \eqref{Tn4.3j} once more, we obtain
\begin{multline}\label{Tn4.3jE}
\babs{\varphi^*_{t}(0)-\varphi^*_{\nicefrac{\tau}{4}}(0)}
\;\le\;
\Bar{\mu}\bigl(\overline\varphi^*_{\nicefrac{\tau}{4}} -\Bar{V}\bigr)
+\int_{0}^{t-\nicefrac{\tau}{4}}
\Bar\mu\bigl(\Tilde{F}_{\Hat{\eta}^{*}_{2\tau-t+s}}\bigr)\,\D{s}\\
+ M_{0} \bigl(1+\Lyap(0)\bigr)\,
\biggl(\E^{-\gamma (t-\nicefrac{\tau}{4}) }
\bnorm{\overline\varphi^*_{\nicefrac{\tau}{4}}-\Bar{V}}_{\Lyap}
+ \int_{0}^{t-\nicefrac{\tau}{4}}\E^{-\gamma s}
\bnorm{\Tilde{F}_{\Hat{\eta}^{*}_{2\tau-t+s}}}_{\Lyap} \,\D{s}\biggr)
\end{multline}
for $t\ge\nicefrac{\tau}{4}$.
The first term in \eqref{Tn4.3jE} vanishes as $\tau\to\infty$ by \eqref{Tn4.3ia},
and the same holds for the integrals by \eqref{Tn4.3jD}.
This proves \eqref{Tn4.3ib}.

Repeating the same argument on the interval $[0,\varepsilon T]$,
we obtain the analogous to \eqref{Tn4.3g}.
Combining the two, and
using the fact that
$\varGamma_{T}((1-\varepsilon) T)-\varGamma_{T}(\varepsilon T)\to 0$
as $T\to\infty$, and
$\varGamma_{T}(t)-\varGamma_{T}(t')\ge0$
for $t\ge t'$,
we deduce that
\begin{equation}\label{Tn4.3k}
\sup_{\lambda\,\in\,[\varepsilon,1-\varepsilon]}\;\varGamma_{T}(\lambda T)
\;\xrightarrow[T\to\infty]{}\; 0\,.
\end{equation}

Let $\Tilde\varepsilon>0$ be given.
By \eqref{Tn4.3g} and \eqref{Tn4.3jC}--\eqref{Tn4.3jD}
we can select $\tau_{0}$ such that, if $\tau>\tau_{0}$ then any
limits $\varphi^{*}$ and $\Hat{\eta}^{*}$ as defined in \eqref{Tn4.3BB}
satisfy
\begin{equation}\label{Tn4.3l}
\sup_{t\,\in\,\bigl[\frac{\tau}{2},\frac{3\tau}{2}\bigr]}\;
\max\;\Bigl(\eD_{1}(\Hat\eta^*_{t},\Bar\mu),
\babs{\varphi^*_{t}(0)-\varphi^*_{\tau}(0)},
\bnorm{\overline\varphi^*_{t}-\Bar{V}}_{\Lyap}\Bigr)
\;\le\; \frac{\Tilde\varepsilon}{4}\,.
\end{equation}
Next, we select any interval of the form
$[\lambda T-\tau,\lambda T+\tau]\subset[\varepsilon T,(1-\varepsilon)T]$,
$\tau>\tau_{0}$.
Given any sequence $T_{n}\to0$,
we can take limits along some subsequence $T_{n}\to\infty$
by \eqref{Tn4.3k}, as done earlier,
and define
\begin{equation*}
\varphi^{*}_{t}\;\df\;\lim_{n\to\infty}\;
\bigl(\varphi^{T_{n}}_{\lambda T_{n}-\tau+t}-\varphi^{T_{n}}_{\lambda T_{n}}(0)
\bigr)\,,\qquad
\Hat\eta^{*}_{t}\;=\;\lim_{n\to\infty}\;\Hat\eta^{T_{n}}_{(1-\lambda) T_{n}-\tau+t}\,,
\qquad t\in[0,2\tau]\,.
\end{equation*}
Therefore,
$\overline\varphi^{*}_{t}\,=\,\lim_{n\to\infty}\,
\overline\varphi^{T_{n}}_{\lambda T_{n}-\tau+t}$.
Since convergence is uniform on $[0,2\tau]$, there exists $n_{0}\in\NN$, such
that
\begin{equation}\label{Tn4.3m}
\sup_{t\,\in\,[0,2\tau]}\;
\max\;\Bigl(\eD_{1}\bigl(\Hat\eta^*_{t},
\Hat\eta^{T_{n}}_{(1-\lambda) T_{n}-\tau+t}\bigr),
\bnorm{\overline\varphi^*_{t}
-\overline\varphi^{T_{n}}_{\lambda T_{n}-\tau+t}}_{\Lyap}\Bigr)
\;\le\; \frac{\Tilde\varepsilon}{2}\qquad \forall\,n\ge n_{0}\,.
\end{equation}
Since
$$\varphi^*_{t}(0)-\varphi^*_{\tau}(0)\;=\;
\lim_{n\to\infty}\;\bigl(\varphi^{T_{n}}_{\lambda T_{n}-\tau+t}(0)
-\varphi^{T_{n}}_{\lambda T_{n}}(0)\bigr)\,,$$
the result clearly follows by \eqref{Tn4.3l}--\eqref{Tn4.3m},
and a standard triangle inequality.
\end{proof}

Convergence of the (RVI) is asserted in the following theorem.

\begin{theorem}\label{Tn4.4}
Under the assumptions of Theorem~\ref{Tn4.3}, it holds that
\begin{equation}\label{Tn4.4AA}
\bnorm{\psi^{T}_{\lambda T}(x) -\Bar{V}(x)+\Bar{\varrho}}_{\Lyap}
\;\xrightarrow[T\to\infty]{}\; 0
\end{equation}
for all $\lambda\in(0,1)$.
\end{theorem}

\begin{proof}
We write \eqref{virvi} as
\begin{equation}\label{Tn4.4a}
\psi^T_{\lambda T}(x) \;=\;
\overline\varphi^T_{\lambda T}(x)
- \int_{0}^{\lambda T}\E^{s-\lambda T}\,\bigl(\varphi^T_s(0)
- \varphi^T_{\lambda T}(0)\bigr)\,\D{s}
+\E^{-\lambda T} \varphi^T_{\lambda T}(0) +
\Bar{\varrho}\,(1-\E^{-\lambda T})
\end{equation}
for $(t,x)\in[0,T]\times\RR^{d}$.
We expand the integral as
\begin{equation*}
\int_{0}^{\lambda T-t_{0}}\E^{s-\lambda T}\,\bigl(\varphi^T_s(0)
- \varphi^T_{\lambda T}(0)\bigr)\,\D{s} +
\int_{\lambda T-t_{0}}^{\lambda T}\E^{s-\lambda T}\,\bigl(\varphi^T_s(0)
- \varphi^T_{\lambda T}(0)\bigr)\,\D{s}
\end{equation*}
for $t_{0}>0$.
The first integral
has a bound $\kappa\E^{-t_{0}}$ for some constant $\kappa$
by Lemma~\ref{L-aux},
while second integral vanishes as $T\to\infty$ by \eqref{Tn4.3b}.
Since $t_{0}>0$ is arbitrary, \eqref{Tn4.4AA} follows by \eqref{Tn4.3a} and
\eqref{Tn4.4a}.
\end{proof}

\begin{remark}
The result of Theorem~\ref{Tn4.4} can be improved to assert convergence
which is uniform on $[\varepsilon T,(1-\varepsilon)T]$, or in other words
that for any $\varepsilon\in(0,\nicefrac{1}{2})$ we have
\begin{equation*}
\sup_{\lambda\,\in\,[\varepsilon,(1-\varepsilon)]}\;
\bnorm{\psi^{T}_{\lambda T}(x) -\Bar{V}(x)+\Bar{\varrho}}_{\Lyap}
\;\xrightarrow[T\to\infty]{}\; 0\,.
\end{equation*}
The same applies to the convergence in \eqref{Tn4.3aA}--\eqref{Tn4.3b}.
To establish this one may follow the argument in the proof of Theorem~\ref{Tn4.3}.
First, under the hypotheses, the map
$$\overline{\cT}\,\colon\,\Pm_{p}(\Rd)
\times\bigl(\Cc^{2}(\Rd)\cap\Cc_{\Lyap}(\Rd)\bigr)
\;\to\;\Cc\bigl([0,2\tau],\Pm_{p}(\Rd)\bigr)$$
which determines the MFG solution
for the finite horizon problem on an interval $[0,2\tau]$ from
an initial distribution $\eta\in\Pm_{p}(\Rd)$ and
a terminal cost $g\in\Cc^{2}(\Rd)\cap\Cc_{\Lyap}(\Rd)$ is continuous
in $\eta$ and $g-g(0)$.
Since $\Hat\eta^{T}_{t}$ lives in some compact set $\sK$ of $\Pm_{p}(\Rd)$
and $\overline\varphi^{T}_{t}$ lives in some compact set $\cG$ of
$\Cc^{2}(\Rd)\cap\Cc_{\Lyap}(\Rd)$
for all $t\in[0,T]$ and $T>0$, then by the uniform continuity of the
map $\overline{\cT}$ on $\sK\times\cG$, it is evident from the proof
of Theorem~\ref{Tn4.3} that the convergence in \eqref{Tn4.3aA}--\eqref{Tn4.3b}
is uniform in $\lambda\in[\varepsilon T,(1-\varepsilon)T]$.
\end{remark}

\begin{remark}
As shown in \cite{CLLP-13}, under the hypothesis
$\norm{F(x,\mu)-F(x,\mu')}_{\Cc^{1+\alpha}}\le C \norm{\mu-\mu'}_{H^{-1}_{0}}$,
convergence is exponential in $T$
for the problem on the $d$-dimensional torus $\mathbb{T}$.
For the model addressed in this paper, it would be interesting
to investigate whether strengthening Assumption~\ref{A4.2}\,(ii) and (iii)
to
$\abs{F(x,\mu)-F(x',\mu')}\le C \abs{x-x'}\,\eD_{p}(\mu,\mu')$
is sufficient to guarantee exponential convergence.
\end{remark}

%
%

\section{Limits of N-Player Games}\label{S-Nplayer}
In this section we consider certain classes of $N$ player games,
 and show that as $N\to\infty$, 
the limiting value function and invariant probability measure solve mean field games. 
As earlier, we consider a 
probability space $(\Omega,\sF, \Prob)$ on which we are given $N$ independent 
$d$-dimensional standard Brownian motions $\{W^1,\dotsc, W^N\}$ with 
respect to a
complete filtration $\{\sF^N_t\}$.
The initial conditions $\{X^i_0\}$ are assumed to be independent of these 
Brownian motions.
The control for the $i^{\text{th}}$ player
lives in a compact, metrizable control set $\Act^i$. 
The set of all admissible controls
is denoted by $\bUadm^N$ and contains paths $(U^1,\dotsc, U^N)$,
satisfying the following: $\{U^i_t(\omega)$, $1\le i\le N\}$, is jointly measurable in 
$(t, \omega)\in [0, \infty)\times\Omega$,
takes values in $\Act^1\times\dotsb\times\Act^N$,  and $U^i$ is
adapted to the Brownian motion $W^i$ for $1\le i\le N$. 
Therefore the game under consideration is non-cooperative.
We note that the controls in $\bUadm^N$ satisfy the
non-anticipativity condition. 
We consider the collection of controlled diffusions
\begin{equation}\label{E5.1}
\D{X}^i_t\; =\; b^i(X^i_t, U^i_t)\, \D{t} + \upsigma^i(X^i_t) \,\D{W}^i_t\, , 
\quad 1\le i\le N.
\end{equation}
We assume that $b^i$, $\upsigma^i$, $1\le i\le N$, satisfy conditions (A1)--(A3) 
possibly for different constants $C_{1}$, $C_R$. 
Therefore, for any admissible control $\bm{U}^N=(U^1,\dotsc, U^N)
\in\bUadm^N$, \eqref{E5.1} has a unique strong solution for every 
deterministic initial condition. 
It might be convenient to think of this system of diffusions as a single 
controlled diffusion with
state space $\R^{dN}$. The cost functions 
$$r^i:\Rd\times\Act\times\calP(\Rd)\to \R_+\, ,$$
are assumed to be continuous and for all $\mu$, $r$ is locally Lipschitz
in the variable $x$ uniformly in $u\in\Act$.
We extend
the action space to the relaxed control framework, and assume that
the admissible control takes values in 
$\calP(\Act^1)\times\dotsb\times\calP(\Act^N)$.
Let $\bUsm=\Usm^1\times\dotsb\times\Usm^N$, where $\Usm^i$ denotes the set of 
measurable maps $v^i:\Rd\to\calP(\Act^i)$. We endow $\bUsm$ with
the product topology; therefore $\bUsm$ forms a compact space. 
By $\bUssm$ we denote the set of all stable stationary Markov controls
in $\bUsm$.
The cost function for the $i$-th player is given by
\begin{equation}\label{E5.2}
J^i(\bm{U}^N)\; \df\; \limsup_{T\to\infty}\, \frac{1}{T}\,
\Exp_x\biggl[\int_0^T r^i\biggl(X^i_t, U^i_t, \frac{1}{N-1}
\sum_{j\neq i}\del_{X^j_t}\biggr)\, \D{t}\biggr]. 
\end{equation}
From \eqref{Edd} it is easy to see that for all $x^j, y^j\in\Rd$, $1\le j\le N-1$, 
we have
$$\dd\biggl(\frac{1}{N-1}\sum_{j=1}^{N-1}\del_{x^j},
\frac{1}{N-1}\sum_{j=1}^{N-1}\del_{y^j}\biggr)
\;\le\; \sum_{j=1}^{N-1} \abs{x^j-y^j}\,.$$
Therefore, defining $\Breve{r}^i : \R^{dN}\times\Act^i\to \R_+$ by 
$$\Breve{r}^i(x^1,\dotsc, x^N, u^i)\;\df\; r^i\biggl(x^i, u^i, 
\frac{1}{N-1}\sum_{j\neq i}\del_{x^j}\biggr)\,,$$ it follows that $\Breve{r}^i$
is continuous in $\R^{dN}$ uniformly in $u^i\in\Act^i$. 
Hence we can redefine the ergodic criterion in \eqref{E5.2} as 
\begin{equation*}
J^i(\bm{U}^N)\; \df\; \limsup_{T\to\infty}\, \frac{1}{T}\, 
\Exp_x\biggl[\int_0^T \Breve{r}^i(X^1_t,\dotsc, X^N_t, U^i_t)\, \D{t}\biggr].
\end{equation*}
By $\Uadm^i$, $1\le i\le N$, we denote the set of all jointly measurable 
functions $U^i:[0, \infty)\times\Omega\to\Act^i$ that are adapted to $W^i$.
\begin{definition}\label{D5.1}
A strategy $\bm{U}=(U^1,\dotsc, U^N)\in\bUadm^N$ is called a 
\emph{Nash equilibrium} for the $N$-player game if for every 
$i\in\{1,\dotsc, N\}$ and $\Tilde{U}^i\in\Uadm^i$, we
have
$$J^i(\bm{U}) \; \le \; J^i(U^1,\dotsc, U^{i-1}, 
\Tilde{U}^i, U^{i+1},\dotsc, U^N)
\quad \text{for almost for all initial points}\; x\,.$$
\end{definition}
\begin{remark}
The above definition of Nash equilibrium is the one used in 
\cite{bardi-priuli, lasry-lions, feleqi}. 
In \cite{fischer} such equilibria are referred to as \emph{local} Nash equilibria. 
In the terminology of \cite{fischer}, $\bUadm^N$ is the set of all 
\emph{narrow} strategies.
\end{remark}

Let $a^i(x) \df \frac{1}{2}\upsigma^i(x)\upsigma^i(x)\transp$.
We define the family of operators
$L^{u}_i\colon\calC^{2}(\Rd)\to\calC(\Rd)$,
where $u\in\Act^i$ plays the role of a parameter, by
\begin{equation*}
L^{u}_i f(x) \;\df\; \trace\bigl(a^{i}(x)\,\grad^2 f(x)\bigr)
+ b^i(x,u)\cdot \grad f(x)\,,\quad u\in\Act^i\,.
\end{equation*}
Therefore, $L^u_i$ is the controlled
extended generator of the $i$-th process in \eqref{E5.1}. 
We define $\eG^i$ and $\eH^i$ similar to \eqref{Eeom} and \eqref{Einvm} 
relative to the operator $L^u_i$.
We assume the following.

\begin{assumption}\label{A5.1}
\begin{enumerate}
\item[(i)]
For $1\le i\le N$, there exist an inf-compact $\calV^i\in\calC^2(\Rd)$, 
an inf-compact,
locally Lipschitz $h^i$, such that 
for some positive constants $\gamma^i_3$, and $\gamma^i_4$ we have
\begin{equation}\label{USN}
 L^u_i\calV^i(x) \; \le \; \gamma^i_4 - \gamma^i_{3} h^i (x)
 \qquad \text{for all}\; u\in\Act^i.
\end{equation}
Moreover, for any compact $\calK\subset\calP(\Rd)$ with respect to the metric $\dd$,
we have 
$$\sup_{u\in\Act^i, \,\nu\in\calK} r^i(\cdot, u, \nu)\;\in\; \sorder(h^i)\,.$$

\item[(ii)]
There exist non-negative locally Lipschitz functions $g^i\in\sorder(h^i)$,
$1\le i\le N$, and
$g^0\in\sorder(\min_{i}h^i)$, satisfying
$$\Breve{r}^i(x^1,\dotsc, x^N, u^i)\; \le\;  g^0(x^i)
+ \frac{1}{N-1}\sum_{j\neq i} g^j(x^j)\qquad \text{for all~} u^i\in\Act^i\,.$$
\end{enumerate}
\end{assumption}
There are quite a few cost functions considered in the literature that satisfy 
Assumption~\ref{A5.1}\,(ii).

\begin{example}
Consider $g^0:\Rd\times\Act^i\to\R_+$, $g^1:\Rd\to\R_+$, such that 
$\sup_{u\in\Act^i}g^0(\cdot, u)$ and $g^1$ are in $\sorder(\min_{i}h^i)$, and define
$$r^i(x,u, \mu) \; \df\; g^0(x,u) + \int g^1\, \D{\mu}\,.$$ 
Note that these running cost functions satisfy Assumption~\ref{A5.1}\,(ii).
\end{example}

We first show that, under Assumption~\ref{A5.1}, there exists a Nash equilibrium
in the sense of Definition~\eqref{D5.1}. 
First, we need to introduce some additional notation.
Let
$$\bm\eG^N \;\df\; \eG^1\times\dotsb\times\eG^N, 
\quad \bm\eH^N\;\df\; \eH^1\times\dotsb\times\eH^N\,.$$
By \eqref{USN} the sets $\eG^i$, $\eH^i$, $1\le i\le N$, are compact, and 
as a result, $\bm\eG^N$ and $\bm\eH^N$ are convex and compact.
For $\bm{\mu}=(\mu^1,\dotsc, \mu^N)\in\bm\eH^N$ we define
\begin{equation}\label{E5.5}
\Breve{r}^i_{\bm{\mu}}(x,u)\; \df\; 
\int_{\Rd\times\dotsb\times\Rd} 
\Breve{r}^i(x^1,\dotsc, x^{i-1}, x, x^{i+1},\dotsc, x^N, u) 
\prod_{j\neq i} \mu^j(\D{x^j}),
\quad 1\le i\le N. 
\end{equation}
Using Assumption~\ref{A5.1} and the dominated convergence theorem,
we deduce that $\Breve{r}^i_{\bm{\mu}}:\Rd\times\Act^i\to\R_+$, is a continuous
function.
 
\begin{assumption}\label{A5.2}
For all $i\in\{1,\dotsc, N\}$ and $\bm\mu\in\bm\eH^N$, the function 
$\Breve{r}^i_{\bm{\mu}}$ is locally Lipschitz in the variable 
$x$ uniformly with respect to 
$u\in\Act^i$ and $\bm\mu\in\bm\eH^N$.
\end{assumption}

\begin{example}
Let $F:\calP(\Rd):\to\R_+$ be a bounded, locally Lipschitz function 
(with respect to the metric $\dd$).
Consider maps $r_i:\Rd\times\Act\to\R$, $i=1,2$, 
having the property
that $r_i$ is locally Lipschitz in first variable uniformly with respect 
to the second. Define
$$r(x,u, \mu)\; \df\; r_{1}(x,u) + r_{2}(x,u)\, F(\mu)\,.$$
It is easy to see that for this running
cost, Assumption~\ref{A5.2} is met.
\end{example}

\begin{example}
Let $\varphi, \varphi_{1}:\Rd\to\R$ be symmetric, locally Lipschitz functions 
with the property that 
$$\abs{\varphi(x)-\varphi(y)}\; \le\; 
\abs{\varphi_{1}(x) + \varphi_{1}(y)}\, \abs{x-y}
\quad \text{for all~} x, y\in\Rd\,,$$
and $\varphi, \varphi_{1}\in\sorder(h)$.
Define
$$r(x,u,\mu)\; =\; \int_{\Rd} \varphi(x-y)\, \mu(\D{y})\,.$$
Assumption~\ref{A5.2} is met for this running cost.
\end{example}

By Assumption~\ref{A5.1}\,(ii), we have 
$\sup_{u^i\in\Act^i} \Breve{r}^i_{\bm{\mu}}(\cdot, \mu)
\in\sorder (h^i)$ for all $i\in\{1,\dotsc, N\}$,
and all $\bm{\mu}\in\bm\eH^N$. 
Since $\sup_{\mu\in\eH^i} \int h^i\, \D{\mu^i} <\infty$ for all $i$ by \eqref{USN}, 
we obtain that 
\begin{equation}\label{E5.06}
\sup_{\bm{\mu}\in\bm\eH^N}\sup_{u^i\in\Act^i} 
\Breve{r}^i_{\bm{\mu}}(\cdot, \mu) 
\;\in\; \sorder (h^i) \quad \text{for all~} i\in\{1,\dotsc, N\}\,.
\end{equation}
Next we treat $\Breve{r}^i_{\bm{\mu}}$ as a running cost, and define 
the ergodic control problem for $\bm{\mu}\in\bm\eH^N$ as
\begin{equation}\label{E5.07}
\Tilde{\varrho}^i_{\bm\mu}\; \df\; \inf_{U^i\in\Uadm^i}\limsup_{T\to\infty} 
\frac{1}{T} \Exp\biggl[\int_0^T\Breve{r}^i_{\bm{\mu}}(X^i_t, U^i_t)\, 
\D{t}\biggr],
\quad 1\le i\le N.
\end{equation}
For every $\bm{\mu}\in\bm\eH^N$, there exists 
a unique $V^i_{\bm\mu}\in\calC^2(\Rd)$, $1\le i\le N$, satisfying 
(see \cite[Theorem~3.7.12]{ari-bor-ghosh})
\begin{equation}\label{E5.08}
\min_{u\in\Act^i}\;\bigl[L^u_i V^i_{\bm\mu}(x) 
+ \Breve{r}^i_{\bm\mu}(x,u)\bigr]
\;=\; \Tilde{\varrho}^i_{\bm\mu}, \quad V^i_{\bm\mu}(0)\;=\;0, 
\quad V^i_{\bm\mu}\;\in\;\sorder(\calV^i)\,.
\end{equation}
As we have discussed earlier in \eqref{Ehjb2}, any measurable selector of 
\eqref{E5.08} is an optimal Markov control for \eqref{E5.07} and vice-versa.
We define
\begin{align*}
\bm{\calA}(\bm{\mu}) &\;\df\; \bigl\{\bm{\uppi}=(\uppi^1,\dotsc,\uppi^N)\in
\bm\eG^N
\;\colon\; \uppi^i=\mu^i_{v^i}\cast v^i\,, \text{~and~} v^i\in\Ussm^i\\
&\mspace{250mu} \text{is a measurable selector satisfying \eqref{E5.08}
for all~} i\bigr\}\,,
\\[2mm]
\bm{\calA}^{*}(\bm{\mu}) 
&\;\df\; \bigl\{\bm{\nu}=(\nu^1,\dotsc, \nu^N)\in\bm\eH^N\;\colon\;
 (\nu^1\cast v^1,\dotsc, \nu^N\cast v^N)\in\bm{\calA}(\bm{\mu})\\
&\mspace{250mu}
\text{for some}~ \bm{v}=(v^1,\dotsc, v^N)\in\bUssm\bigr\}\,.
\end{align*}
It is easy to find the analogy of the above maps with $\calA$ and $\calA^*$ 
defined in Section~\ref{S-convex}.
The following theorem establishes the existence of a
Nash equilibrium for the N-person game.

\begin{theorem}\label{T5.1}
Assume that Assumptions~\ref{A5.1} and \ref{A5.2} hold.
Then there exists $\bm{\mu}\in\bm\eH^N$ satisfying 
\begin{equation}\label{E5.09}
\bm{\mu}\;\in\;\bm{\calA}^*(\bm{\mu})\,.
\end{equation}
In particular, there exists a stable Markov control
$\textbf{v}=(v^1,\dotsc, v^N)\in\bUssm$ such that, for all
$i\in\{1,\dotsc, N\}$, we have
\begin{equation}\label{E5.10}
\Tilde{\varrho}^i_{\bm\mu} \;=\;
J^i(\textbf{v}) \;\le\; J^i(v^1,\dotsc, v^{i-1}, \Tilde{U}^i, v^{i+1},\dotsc, v^N)
\quad \text{for all}\;\, \Tilde{U}^i\in\Uadm^i\,.
\end{equation}
\end{theorem}

\begin{proof}
The main idea of the proof is to use the Kakutani--Fan--Glicksberg
fixed point theorem as we have done in the proof of Theorem~\ref{T3.2}.
Consider the convex, compact set $\bm\eH^N$. Since the product of Hausdorff 
locally convex spaces is again a Hausdorff locally convex space, it follows that
$\bm\eH^N$ is a non-empty, convex, compact  subset of 
$\calM(\Rd)\times\dotsb\times\calM(\Rd)$. 
Following an argument similar to Lemma~\ref{L3.1} we deduce that
$\bm{\calA}^*(\bm{\mu})$ is non-empty, convex and compact for all 
$\bm{\mu}\in\bm\eH^N$. Let $\bm{\mu}_n\to \bm{\mu}$ as $n\to\infty$. 
Using the non-degeneracy
of $a^i$, $1\le i\le N$, we can improve this to convergence 
under the total variation norm (\cite[Lemma~3.2.5]{ari-bor-ghosh}).
Therefore using \eqref{E5.06} together with an argument similar to
the proofs of Lemmas~\ref{L3.3}--\ref{L3.5}
we obtain that $\bm{\mu}\mapsto\bm{\calA}^*(\bm{\mu})$ is upper-hemicontinuous. 
Hence we can apply the Kakutani--Fan--Glicksberg fixed point theorem 
\cite[Corollary~17.55]{ari-bor-ghosh} to obtain a $\bm{\mu}\in\bm\eH^N$ 
satisfying $\bm{\mu}\in\bm{\calA}^*(\bm{\mu})$. This proves \eqref{E5.09}.

By the definition of an ergodic occupation measure,
we can find $\textit{\textbf{v}}=(v^1,\dotsc, v^N)\in\bUssm$ 
such that for $\bm{\mu} =(\mu^1,\dotsc, \mu^N)$ we have 
\begin{equation}\label{E5.11}
(\mu^1\cast v^1,\dotsc, \mu^N\cast v^N) \;\in\; \bm{\calA}(\mu)\,.
\end{equation}
In particular, $\mu^i$ the unique invariant probability measure of \eqref{E5.1} 
associated to the stationary Markov control $v^i$.
Without loss of generality we fix $i=1$. 
To show \eqref{E5.10} we consider $\Tilde{U}^1\in\Uadm^i$. 
Define the occupation measure $\upxi_T$ on 
$\Rd\times\Act^1\times\R^{d(N-1)}$ as follows:
\begin{equation}\label{E5.12}
\upxi_T(A\times B\times C)\, \df\, \frac{1}{T}
\Exp\biggl[\int_0^T \Ind_{A\times B\times C}
(X^1_t, U^1_t, \Hat{X}^2_t,\dotsc, \Hat{X}^{N}_t)\, \D{t}\biggr]\,, \quad T>0\,, 
\end{equation}
for $A\in\calB(\Rd)$, $B\in\calB(\Act^1)$, $C\in\calB(\R^{d(N-1)})$,
where $(X^1, U^1)$ solves \eqref{E5.1} for $i=1$, and $X^j$, $j>1$, are the solutions 
to \eqref{E5.1} under the Markov control $v^j$. 
Using Assumption~\ref{A5.1}, we deduce that $\{\upxi_T, \, T>0\}$ is a 
tight family of probability measures.
By weak convergence,  we have
$\Exp[f(\Hat{X}^i(T))]\to \mu^i(f)$, $i\geq 2$, as $T\to\infty$,
for any bounded continuous function
$f:\Rd\to\R$.
Hence using the independence property of $((X^1, U^1), \Hat{X}^2,\dotsc, \Hat{X}^N)$
and the definition in \eqref{E5.12}, we can easily show that as $T\to\infty$, 
the limit points of $\upxi_T$ as $T\to\infty$ belong to the set
$$\{\uppi^1\times\mu^2\times\dotsb\times\mu^N \, :\; \uppi\in\eG^1\}\,.$$
For above to hold we also use the fact that the collection
$\{g(x^1, u)\cdot f^2(x^2)\cdot f^N(x^N)\; \colon\; g\in\calC_b(\Rd\times\Act),
f^i\in\calC_b(\Rd)\}$
determines the probability measures on $\Rd\times\Act\times(\Rd)^{N-1}$.
By lower-semicontinuity we have
$$J^1(U^1, v^2,\dotsc, v^N)\;\ge\; \int_{\Rd\times\Act^1} 
\Breve{r}^1_{\bm{\mu}} (x,u)\, \D{\uppi}^1\,,$$
for some $\uppi^1\in\eG^1$.
Hence using \eqref{E5.11} and 
\cite[Theorem~3.7.12]{ari-bor-ghosh} we obtain
$$J^1(U^1, v^2,\dotsc, v^N) \;\ge\; \int_{\Rd\times\Act^1} 
\Breve{r}^1_{\bm{\mu}} (x,u)\, 
\D{\uppi}^1\;\ge\; \Tilde{\varrho}^1_{\bm{\mu}}\,.$$
To complete the proof we observe that $(\Hat{X}^1,\dotsc, \Hat{X}^N)$ 
is a strong Markov process with invariant probability measure
$\mu^1\times\dotsb\times\mu^N$.
Therefore, by Birkhoff's ergodic theorem, we have
$$J^1(v^1,\dotsc, v^N) \;=\; \Tilde{\varrho}^1_{\bm{\mu}} \;=\; 
\limsup_{T\to\infty}\;\frac{1}{T}\,\Exp\biggl[\int_0^T
\Breve{r}^1\bigl(\Hat{X}^1_t, 
\dotsc, \Hat{X}^N_t, v^1(\Hat{X}^1)\bigr)\, \D{t}\biggr]\,.\qedhere$$
\end{proof}

\subsection{Symmetric Nash equilibria}
We now let the number of players $N$ tend to infinity, 
assuming that all the players are identical. 
Hence, for the rest of this section we assume that
\begin{align*}
\Act^i=\Act, \quad  b^i=b, \quad \upsigma^i=\upsigma, \quad r^i= r, 
\quad \calV^i=\calV, \quad h^i=h
\qquad\text{for all~}i\in\{1,\dotsc, N\}\,.
\end{align*}
A similar argument as in the proof of Theorem~\ref{T5.1} provides us
with a Nash  equilibrium
of the form $\textit{\textbf{v}}=(v,\dotsc, v)$ and 
$\bm{\mu}=(\mu,\dotsc, \mu)$, where
$\mu$ is the unique invariant probability measure corresponding to the Markov control 
$v$, and $v$ is a measurable selector of \eqref{E5.08} 
(compare this with \cite[Theorem~2.2]{lasry-lions}). 
These equilibria are known as the symmetric Nash equilibria.

\begin{remark}
Any (Markovian) Nash equilibrium for the $N$-player game is related to a
fixed point of the map  $\bm{\calA}^*(\cdot)$.
To elaborate consider any tuple $(v^N, \ldots, v^N)\in\bm{\Ussm}^N$ that
corresponds to a
Nash equilibrium.
Let $\bm{\mu}=(\mu^{1}, \ldots, \mu^{N})$ be the corresponding invariant measures.
Solve equations \eqref{E5.07}--\eqref{E5.08} with above choice of ${\bm{\mu}}$.
Since $(v^1, \ldots, v^N)$ is a Nash equilibrium, it follows that $v^i$ is an
optimal Markov control in \eqref{E5.07}.
Thus $\bm{\mu}\in\bm{\calA}^*(\bm{\mu})$.
\end{remark}

\begin{remark}
Assumption~\ref{A5.2} is not very crucial for Theorem~\ref{T5.1}.
If $\breve{r}^i_{\bm{\mu}}$ is only continuous, then
the value function $V^i_{\bm\mu}$ is in
$\Sobl^{2, p}(\Rd), \, p\geq 1$ instead of
$\calC^2(\Rd)$, but the conclusion of Theorem~\ref{T5.1} still holds.
\end{remark}

In the rest of this section we discuss the convergence of the $N$-person
game as $N$ tends to infinity.
In what follows we work with Wasserstein metric instead of the metric
of weak convergence.
We also need some additional regularity assumptions on $r$,
which are as follows.

\begin{assumption}\label{A5.3}
the following hold:
\begin{itemize}
\item[(i)] There exists an inf-compact $\calV\in\calC^2(\Rd)$, such that 
for some positive constants $\gamma_3, \gamma_4$, 
\begin{equation}\label{again-US}
 L^u\calV(x) \; \le \; \gamma_4 - \gamma_{3}\, \abs{x}^q, 
 \quad \text{for all}\; u\in\Act,\quad \text{and}\; q>1,
\end{equation}
and for any compact $\calK\subset\calP(\Rd)$ with respect to the metric 
$\eD_{\bar{q}}$, ${\bar{q}}\in[1, q)$, we have 
$$\sup_{u\in\Act, \nu\in\calK} r(\cdot, u, \nu)\;\in\; \sorder(\abs{x}^q)\,.$$
Moreover, there exists  non-negative locally Lipschitz functions
$g^0\in\sorder(\abs{x}^q)$ and
$g^1\in\sorder(\abs{x}^q)$ satisfying
\begin{equation}\label{E5.14}
r\biggl(x,u, \frac{1}{N}\sum_{j=1}^N\delta_{y^j}\biggr)\; \le \; 
g^0(x) + \frac{1}{N}\sum_{j=1}^N g^1(y^j)
\quad \text{for all}\; u\in\Act, \quad N\ge 1\,,
\end{equation}
and 
\begin{equation}\label{hat-r}
(x,u)\;\mapsto\; \Hat{r}^N_{\bm\mu}(x,u)\; \df\; \int_{\R^{Nd}}
r\biggl(x,u, \frac{1}{N}\sum_{j=1}^N\del_{y^j}\biggr)\,
\prod_{j=1}^N \mu^j(\D{y^j})
\end{equation}
is continuous, and locally Lipschitz in $x$
(the Lipschitz constant might depend on $N$) uniformly in $u$, for all
$\bm{\mu}=(\mu^1,\dotsc, \mu^N)\in\bm\eH^N$.
\smallskip
\item[(ii)]
For some ${\bar{q}}\in[1, q)$ we have 
\begin{equation}\label{E5.16}
\abs{r(x,u, \nu)-r(x,u, \Tilde{\nu})}\; \le\; \kappa_R\,
\biggl(1+\int_{\Rd} \abs{y}^{\bar{q}} {\nu(\D{y})} + 
\int_{\Rd} \abs{y}^{\bar{q}} \Tilde{\nu}(\D{y})\biggr)^{1-\nicefrac{1}{\bar{q}}}
\, \eD_{\bar{q}}(\nu, \Tilde\nu)
\end{equation}
for all $\abs{x}\le R$, $u\in\Act$, and $R>0$.
For every $(x, u)\in\Rd\times\Act$ and $R>0$ there exists
a constant $\kappa^\prime$, depending on $x$, $u$, and $R$,
such that
\begin{equation}\label{E5.16.5}
\babss{r\biggl(x, u, \frac{1}{N}\sum_{j=1}^N \delta_{y^j}\biggr)
- r\biggl(x, u, \frac{1}{N-1}\sum_{j=1}^{N-1} \delta_{y^j}\biggr)}\;
\leq \kappa^\prime\, \frac{1}{N}, \quad \forall\; y^j\,\in B_{R}\,.
\end{equation}
\smallskip
\item[(iii)]
$\Act$ is a convex set and for all $R>0$ the following holds:
for any $\theta\in(0, 1)$  there exists $\kappa_{\theta, R}>0$,
such that
\begin{multline*}
\, b(x, \theta u + (1-\theta) u')\cdot p + r(x, \theta u + (1-\theta) u', \mu) 
\\[3pt]
\le \; \theta [ b(x,u )\cdot p + r(x,  u, \mu)] 
+ (1-\theta)[b(x,u')\cdot p + r(x,  u', \mu)]-\kappa_{\theta, R}
\end{multline*}
for all $u, u'\in\Act$, $\mu\in\calP(\Rd)$, and $\abs{x}, \abs{p}\le R$.
\end{itemize}
\end{assumption}

We note that \eqref{again-US} is the uniform stability condition
we have used before, and
\eqref{E5.16} is a Lipschitz property of the function
$r$ in the variable $\mu$.
Note also that for ${\bar{q}}\in[1, q)$, $\mu\mapsto r(x,u, \mu)$ is 
locally Lipschitz uniformly
with respect to $(x,u)$ in compact subsets of $\Rd\times\Act$. 
Assumption~\ref{A5.3}\,(iii) is a strict convexity condition that we need
in order to 
resolve the issue of non-uniqueness of the optimal control. 
Running costs considered in \cite{bardi-priuli, feleqi, lasry-lions} 
do satisfy this condition.

\begin{example}
Let $r(x, u, \mu) = R(x, u, \zeta(x, \mu))$ where
$R\colon\Rd\times\Act\times\R$ is a continuous function and for every compact
$K\subset\Rd$ there exists  constant $\gamma_K$ satisfying
\begin{equation}\label{Eg1}
\abs{R(x, u, z)-R(y, u, z_1)}\; \le \;
\gamma_K\,  (\abs{x-y} + \abs{z-z_1}) \quad \text{for all}\;  \; z,\, z_1\in\R,\;
\text{and}~(x, y)\in K\times K.
\end{equation}
Also suppose that for some $g_0\in\sorder(\abs{x}^q)$ we have
$$R(x, u, z) \leq g_0(x) + \kappa \abs{z} \qquad \forall \; x\in\Rd, \;  z\in\R\,,$$
for some constant $\kappa>0$, and that for some $\bar{q}\in[1, q)$ we have
$$ \zeta(x, \mu) \; =\; \int_{\Rd}\abs{x-y}^{\bar{q}}\, \D{\mu}\,.$$
Since $a^{\bar{q}}- b^{\bar{q}} \le {\bar{q}}\, (a^{\bar{q}-1}
+ b^{\bar{q}-1})\abs{a-b}$ for all $a, b\ge 0$ and 
${\bar{q}}\ge 1$, 
then, for any $\gamma\in\calP(\Rd\times\Rd)$ with marginals $\nu, \Tilde\nu$,
it holds that
\begin{align*}
\zeta(x,\nu)-\zeta(x,\Tilde\nu) &\le {\bar{q}}\int_{\Rd\times\Rd} 
(\abs{x}^{\bar{q}-1}+\abs{y}^{\bar{q}-1})\abs{x-y}\gamma(\D{x}, \D{y})
\\
&\le\kappa\biggl[\int_{\Rd\times\Rd} (1+\abs{x}^{\bar{q}}+\abs{y}^{\bar{q}})
\gamma(\D{x}, \D{y})\biggr]^{\frac{\bar{q}-1}{\bar{q}}}
\biggl[\int_{\Rd\times\Rd}
\abs{x-y}^{\bar{q}}\gamma(\D{x},\D{y})\biggr]^{\nicefrac{1}{\bar{q}}}.
\end{align*}
Since $\gamma$ is arbitrary, using \eqref{Eg1},
we deduce that $r$ satisfies \eqref{E5.16}. One can also show that \eqref{hat-r}
and \eqref{E5.16.5} are also satisfied.

\end{example}

Similar to \eqref{E5.5} we define
for $\bm{\mu}=(\mu^1_N,\dotsc, \mu^N_N)\in\bm\eH^N$,
\begin{equation*}
\Breve{r}^{i,N}_{\bm{\mu}}(x,u)\; \df\; 
\int_{\Rd\times\dotsb\times\Rd} 
\Breve{r}^i(x^1,\dotsc, x^{i-1}, x, x^{i+1},\dotsc, x^N)\, 
\prod_{j\neq i} \mu^j_N(\D{x^j}),
\quad 1\le i\le N. 
\end{equation*}

\begin{theorem}\label{T5.2}
Let Assumption~\ref{A5.3} hold. 
Let $\bm{\mu}=(\mu^1_N,\dotsc, \mu^N_N)$ be such that for 
 $\Tilde\varrho^i_N\in\R$, $V^i_N \in\calC^2(\Rd)$, $1\le i\le N$, we have
\begin{align}
\min_{u\in\Act}\;\bigl[L^u V^i_N(x) + \Breve{r}^{i,N}_{\bm\mu}(x,u)\bigr] 
&\;=\; L^{v^i_N} V^i_N(x) + \Breve{r}^{i,N}_{\bm{\mu}}(x,v^i_N)
\,=\, \Tilde{\varrho}^i_N,
\quad V^i_N(0)=0, 
\quad V^i_N\in\sorder(\calV),\label{E5.17}
\\[2mm]
\int_{\Rd} L^{v^i_N} f(x)\, \mu^i_N(\D{x})&\;=\;0 
\quad \text{for all~} f\in\calC^2_c(\Rd),\quad 1\le i\le N.\label{E5.18}
\end{align}
Then the following hold:
\begin{enumerate}
\item[{(a)}] $\{(\Tilde\varrho^i_N, V^i_N, \mu^i_N)\}_{i, N}$ 
is relatively compact in
$\R\times\Sobl^{2, p}(\Rd)\times(\calP_{\bar{q}}(\Rd), \eD_{\bar{q}})$ for
any $1\le p< \infty$;
\smallskip
\item[{(b)}] $\sup_{i, j} (\abs{\Tilde\varrho^i_N-\Tilde\varrho^j_N} 
+ \norm{V^i_N-V^j_N}_{\Sob^{2, p}(K)} + \eD_{\bar{q}}(\mu^i_N, \mu^j_N))\to 0$
as $N\to\infty$, for all compact subsets $K\subset\Rd$;
\smallskip
\item[{(c)}] any limit point $(\Tilde\varrho, V, \mu)$ 
of $\{(\Tilde\varrho^i_N, V^i_N, \mu^i_N)\}_{i, N}$ solves 
\begin{align}
\min_{u\in\Act}\;\bigl[L^u V(x) + r(x,u, \mu)\bigr] 
&\;=\; L^{v} V(x) + r(x,v, \mu)\;=\; \Tilde{\varrho}\,,
 \quad V(0)=0, 
\quad V\in\sorder(\calV),\label{E5.19}
\\[2mm]
\int_{\Rd} L^{v} f(x)\, \mu(\D{x})&\;=\;0
\quad \text{for all}\; f\in\calC^2_c(\Rd)\,.\label{E5.20}
\end{align}
\end{enumerate}
\end{theorem}

We see that \eqref{E5.19}--\eqref{E5.20} defines a MFG solution in the sense of 
Definition~\ref{DMFG2}.
Theorem~\ref{T5.2} asserts that the limits
of N-player games are solutions to mean field games.
Similar results are also 
obtained in \cite{lasry-lions, feleqi} in the case of a compact state space. 
One of the key ideas to prove Theorem~\ref{T5.2} is to use \eqref{again-US} 
to show that one can consider compact subsets of $\Rd$ to
approximate integrals.
This is done following the method in 
\cite[Corollary~5.13]{cardaliaguet}.
To accomplish this we introduce the projection map
\begin{equation*}
\mathfrak{P}(x)\;=\;\mathfrak{P}_R(x)\;\df\; \begin{cases}
x & \text{if}\; x\in\Bar{B}_R(0)\,,
\\[3pt]
0 & \text{otherwise}.
\end{cases}
\end{equation*}
Also define $\Tilde\mu^i_N (B) \;=\; \mu^i_N(\mathfrak{P}^{-1}(B))$ for 
$B\in\calB(\Rd)$. Then we have the following result.
\begin{lemma}\label{L5.1}
Let $\eps>0$ be given. 
Then for any compact set $C\subset\Rd$, there exists $R>0$, 
such that
\begin{equation*}
\sup_{(x,u)\in C\times\Act}\;
\bigl|\Hat{r}^{N}_{\bm\mu}(x,u)- \Hat{r}^{N}_{\Tilde{\bm\mu}}(x,u)\bigr|
\; \le\;\varepsilon\qquad\forall\,N\ge 1\,,
\end{equation*}
where $\Hat{r}^{N}_{\bm\mu}$ is given by \eqref{hat-r}.
\end{lemma}

\begin{proof}
We claim that for any $x^j, y^j\in\Rd$, $1\le j\le N$, we have
\begin{equation}\label{E5.21}
\textstyle\eD_{\bar{q}}\Bigl(\frac{1}{N}\sum_{j=1}^N\delta_{x^j},
\frac{1}{N}\sum_{j=1}^N\delta_{y^j}\Bigr)\;\le\;
\Bigl(\frac{1}{N}\sum_{i=1}^N \abs{x^j-y^j}^{\bar{q}}\Bigr)^{\nicefrac{1}{\bar{q}}}\,.
\end{equation}
Indeed, this can be obtained by choosing
$\nu(\D{x}, \D{y})\df \frac{1}{N}\sum_{j=1}^N \del_{(x^j, y^j)}$ in \eqref{EDp}.
Using \eqref{again-US} we can find a constant $\kappa_{1}$ such that
\begin{equation}\label{E5.22}
\sup_{\nu\in\eH}\;\int_{\Rd} \abs{x}^q \, \nu(\D{x})\; \le\; \kappa_{1}\,.
\end{equation} 
Then
\begin{align*}
& \bigl|\Hat{r}^{N}_{\bm\mu}(x,u)- \Hat{r}^{N}_{\Tilde{\bm\mu}}(x,u)\bigr|
\\
&\quad=\; \biggl|\int_{\R^{Nd}} 
\biggl[r\biggl(x,u, \frac{1}{N}\sum_{j=1}^N\delta_{y^j}\biggr)
-  r\biggl(x,u, \frac{1}{N}\sum_{j=1}^N\delta_{\mathfrak{P}(y^j)}\biggr)\biggr]
\prod_{j=1}^N \mu^j_N(\D{y^j})\biggr|
\\
&\quad\le\; \kappa_C \biggl|\int_{\R^{Nd}}
\biggl[\biggl(1+ \frac{1}{N}\sum_{j=1}^N 
\bigl(\abs{y^j}^{\bar{q}}
+ \abs{\mathfrak{P}(y^j)}^{\bar{q}} \bigr)\biggr)^{\frac{\bar{q}-1}{\bar{q}}}
\, \eD_{\bar{q}}\biggl(\frac{1}{N}\sum_{j=1}^N\delta_{y^j}, \frac{1}{N}\sum_{j=1}^N
\delta_{\mathfrak{P}(y^j)}\biggr) \biggr]\,\prod_{j=1}^N \mu^j_N(\D{y^j})\biggr|
\\
&\quad\le\; \kappa_C \biggl|\int_{\R^{Nd}}
\biggl(1+ \frac{1}{N}\sum_{j=1}^N
\bigl(\abs{y^j}^{\bar{q}} + \abs{\mathfrak{P}(y^j)}^{\bar{q}} \bigr) \biggr)\,
\prod_{j=1}^N \mu^j_N(\D{y^j})\biggr|^{\frac{\bar{q}-1}{\bar{q}}}
\\
&\mspace{250mu} \times\; \biggl|\int_{\R^{Nd}} \biggl[\eD_{\bar{q}}\biggl(\frac{1}{N}
\sum_{j=1}^N\delta_{y^j}, \frac{1}{N}\sum_{j=1}^N\delta_{\mathfrak{P}(y^j)}\biggr)
\biggr]^{\bar{q}} \prod_{j=1}^N \mu^j_N(\D{y^j})
\biggr|^{\nicefrac{1}{\bar{q}}}
\\
&\quad\le\; \kappa_{2} \biggl|\int_{\R^{Nd}} \frac{1}{N}\sum_{j=1}^N 
\abs{y^j-\mathfrak{P}(y^j)}^{\bar{q}}
\prod_{j=1}^N \mu^j_N(\D{y^j})\biggr|^{\nicefrac{1}{\bar{q}}}
\\
&\quad\le\; \kappa_{2} \biggl| \frac{1}{N}\sum_{i=1}^N \int_{B_R^c(0)}
\abs{y^j}^{\bar{q}}  \mu^j_N(\D{y^j})\biggr|^{\nicefrac{1}{\bar{q}}}
\\
&\quad\le\; \kappa_3 \frac{1}{R^{q-{\bar{q}}}}\
\qquad\forall\, (x,u)\in C\times\Act\,,
\end{align*}
for some constants $\kappa_{2}, \kappa_3$, where in the third line we use 
\eqref{E5.16}, in the fourth line we use the H\"{o}lder inequality, 
\eqref{E5.21} is used in the fifth line, and in the last line we use \eqref{E5.22}. 
Choosing $R$ large enough completes the proof.
\end{proof}

\begin{proof}[Proof of Theorem~\ref{T5.2}.]\,
Since $(g^0\vee g^1)(x)\le \kappa_{2}(1+ \abs{x}^q)$ we obtain from 
\eqref{E5.14} that 
\begin{equation}\label{E5.23}
\Breve{r}^{i,N}_{\bm\mu}(x,u)\; \le\; g^0(x) 
+ \frac{\kappa_{2}}{N-1}\sum_{j=1}^{N-1}\int_{\Rd} (1+ \abs{x}^q)\, \D{\mu^j_N}
\;\le \; g^0(x) + \kappa_{2}(1+\kappa_{1})\,,
\end{equation}
where we also use \eqref{E5.22}.
Recall that $\tc_r$ denotes the hitting time to the ball $B_r(0)$ and 
$\uptau_R$ is the exit time from the ball $B_R(0)$. 
From \cite[Lemma~3.3.4]{ari-bor-ghosh} and \eqref{again-US} we know that 
$\sup_{v\in\Ussm}\Exp_x[\tc_r]<\infty$ for $r>0$. Therefore using It\^{o}'s formula
in \eqref{again-US} we obtain that for $r>0$,
\begin{equation}\label{E5.24}
\limsup_{R\to\infty}\sup_{v\in\Ussm}\Exp^v_x\biggl[\calV(X_{\tc_r\wedge\uptau_R}) 
+ \gamma_3\int_0^{\tc_r\wedge\uptau_R}\abs{X_t}^q
\D{t}\biggr]\;\le\; \calV(x) + \kappa_3
\end{equation}
for some constant $\kappa_3$.
Since $V^i_N\in\sorder(\calV)$, for every $\eps>0$, there exists $\kappa_\eps$ 
satisfying $V^i_N(x) \le \kappa_\eps + \eps \calV(x)$.
Therefore, using \eqref{E5.24} we obtain
\begin{align*}
\limsup_{R\to\infty}\;\sup_{v\in\Ussm}\Exp_x\bigl[\Ind_{\{\uptau_R<\tc_r\}}
V^i_N(X_{\uptau_R})\bigr]&\;\le\;
\eps\, \limsup_{R\to\infty}\;\sup_{v\in\Ussm}\Exp_x\bigl[\Ind_{\{\uptau_R<\tc_r\}}
\calV(X_{\uptau_R})\bigr]\\[5pt]
&\;\le\; \eps(\calV(x) + \gamma_4)\,.
\end{align*}
Since $\eps$ is arbitrary, we have $\limsup_{R\to\infty}\sup_{v\in\Ussm}\Exp_x
\bigl[\Ind_{\{\uptau_R<\tc_r\}}V^i_N(X_{\uptau_R})\bigr]=0$.
Thus applying It\^{o}'s lemma to \eqref{E5.17} we have
\begin{equation}\label{E5.25}
V^i_N(x)
\;=\;\inf_{v\in\Ussm}\Exp_x^v\biggl[\int_0^{\tc_r}
\big(\Breve{r}^{i,N}_{\bm\mu}(X_t, v(X_t))
-\Tilde{\varrho}^i_N\big)\, \D{t} + V^i_N(X_{\tc_r})\biggr]\,,
\quad x\in B^c_r(0)\,,~ r>0\,.
\end{equation}
A similar argument as in \eqref{E5.24} gives us
\begin{equation}\label{E5.26}
\limsup_{T\to\infty}\sup_{v\in\Ussm}\Exp^v_x\biggl[\frac{1}{T}\calV(X_{T}) 
+ \frac{\gamma_3}{T}\int_0^{T}\abs{X_t}^q
\D{t}\biggr]\;\le\; \gamma_4\,.
\end{equation}
Therefore using \eqref{E5.26} and the fact that $V^i_N\in\sorder(\calV)$ we have 
$\lim_{T\to\infty}\frac{1}{T}\sup_{v\in\Ussm}\Exp_x[V^i_N(X_T)]=0$.
Applying Dynkin's theorem to \eqref{E5.17}, and using \eqref{E5.23} and 
\eqref{E5.26}, we obtain
\begin{align*}
\Tilde{\varrho}^i_N &=\; \lim_{T\to\infty}\frac{1}{T}\Exp_x^{v^i_N}
\biggl[\int_0^T \Breve{r}^{i,N}_{\bm\mu}(X_t, v^i_N(X_t))\, \D{t}\biggr]\\[3pt]
& \le\; \lim_{T\to\infty}\frac{1}{T}\Exp_x^{v^i_N}\biggl[\int_0^T g^0(X_t)\, \D{t}
\biggr] + \kappa_{2}(1+\kappa_{1})\\[3pt]
&\le\; \kappa_4
\end{align*}
for some constant $\kappa_4$, independent of $i$ and $N$. This shows that 
$\{\Tilde{\varrho}^i_N\}_{i, N}$ is relatively compact. By \eqref{E5.26} 
and Proposition~\ref{P-villani}
we see that $\{\mu^{i}_N\}_{i, N}$ is relatively compact in
$(\calP_{\bar{q}}(\Rd), \eD_{\bar{q}})$. 
Next we prove the compactness of $\{V^i_N\}_{i, N}$. Define
$$V^{\alpha}_{i, N}(x)\; \df\; \inf_{U\in\Uadm}\Exp^U_x
\biggl[\int_0^\infty e^{-\alpha t}\;  
\Breve{r}^{i,N}_{\bm\mu}(X_t, U_t)\, \D{t}\biggr]\,.$$
It is shown in \cite[Theorem~3.7.12]{ari-bor-ghosh} that 
$V^{\alpha}_{i,N}(x)-V^{\alpha}_{i,N}(0)$ are locally bounded in 
$\Sobl^{2, p}(\Rd)$ and converge to $V^i_N$, as $\alpha\to 0$,
in $\Sobl^{2, p}(\Rd)$, $p\geq 1$.
By \eqref{E5.23} we have that $\Breve{r}^{i,N}_{\bm\mu}$ are locally bounded 
uniformly in $N$ and $i\in\{1,\dotsc, N\}$. 
Thus applying \cite[Lemma~3.6.3]{ari-bor-ghosh}
we obtain that for any $R>0$ there exists a constant $\varpi_R$, 
independent of $i, N$, such that
\begin{equation}\label{E5.27}
\norm{V^{\alpha}_{i,N}(\cdot)-V^{\alpha}_{i,N}(0)}_{\Sob^{2, p}(B_R(0))}
\; \le \; \varpi_R \Bigl(1 + \alpha \sup_{B_R}V_\alpha^{i,N}\Bigr)\,.
\end{equation}
Since $g^0\in\sorder(h)$, using \eqref{E5.26} and \eqref{again-US} 
it is easy to see that
$$\sup_{x\in B_R(0)} \alpha V^{i, N}_\alpha(x) \;\le\; \widehat\varpi_R$$
for some constant $\widehat\varpi_R$.
Thus using \eqref{E5.27} we have 
$\norm{V^{\alpha}_{i,N}(\cdot)-V^{\alpha}_{i,N}(0)}_{\Sob^{2, p}(B_R(0))}
\;\le\; \varpi_R\bigl(1+\widehat\varpi_R\bigr)$, which  gives 
\begin{equation}\label{E5.28}
\norm{V^i_N}_{\Sob^{2, p}(B_R(0))}\; \le\; \varpi_R\bigl(1+\widehat\varpi_R\bigr)\,,
\end{equation}
with $\varpi_R$, $\widehat\varpi_R$ do not depending on $i, N$.
Since $V^i_N(0) =0$, we have
$$\sup_{B_{1}(0)}\abs{V^i_N(x)}\; \le \kappa_5$$
for some constant $\kappa_5$.
By \cite[Lemma~3.7.2]{ari-bor-ghosh} we have
$$x\mapsto \sup_{v\in\Ussm}\Exp_x\biggl[\int_0^{\tc_{1}} (1+g^0(X_t))\,\D{t}\biggr]
\in \sorder(\calV)\,.$$
Thus from \eqref{E5.23} and \eqref{E5.25} we obtain that $\sup_{x\in B_{R}(0)}
\abs{V^i_N(x)}\; \le \kappa_R$ where $\kappa_R$ is independent of $i, N$.
Therefore using standard elliptic regularity
theory in \eqref{E5.17} we deduce that $\{V^i_N\}$ is
bounded in $\Sobl^{2, p}(\Rd)$.
This completes the proof of part (a).

Next we prove part (b). Recall the definition of $\Hat{r}^N$ from \eqref{hat-r}.
Consider the unique solution $W^N\in\calC^2(\Rd)$ of the
equation
\begin{equation}\label{E5.29}
\min_{u\in\Act}\;\bigl[L^u W^N(x) + \Hat{r}^N_{\bm\mu}(x,u)\bigr]\; =\;
\lambda_N, \quad W^N(0)=0, \quad W^N\in\sorder(\calV).
\end{equation}
For existence and uniqueness of $W^N$
we refer the reader to \cite[Theorem~3.7.12]{ari-bor-ghosh}.
From \eqref{E5.17} and \eqref{E5.29} we have the following characterizations
\begin{align*}
\Tilde{\varrho}^i_N &\;=\; \min_{\uppi\in\eG}\;\int_{\Rd\times\Act}
\Breve{r}^{i,N}_{\bm{\mu}}(x,u)\,\uppi(\D{x},\D{u})\,,\\
\lambda_N &\;=\; \min_{\uppi\in\eG}\;\int_{\Rd\times\Act}
\Hat{r}^N_{\bm\mu}(x,u)\,
\uppi(\D{x},\D{u})\,.
\end{align*}

It is easy to see that $\Hat{r}^N$ satisfies a similar estimate as \eqref{E5.23}
for all $x$ and $u$. Using \eqref{E5.22}--\eqref{E5.23} we see that for any $\eps>0$,
we can find $R>0$
large enough satisfying
\begin{equation}\label{E5.30}
\sup_{\uppi\in\eG}\;\int_{B_R^c(0)\times\Act}
\Breve{r}^{i,N}_{\bm{\mu}}(x,u)\,\uppi(\D{x},\D{u})
+ \sup_{\uppi\in\eG}\;\int_{B_R^c(0)\times\Act}
\Hat{r}^{N}_{\bm\mu}(x,u)\,\uppi(\D{x},\D{u})
\; \le \; \varepsilon\,.
\end{equation}
Recall the projection map $\mathfrak{P}=\mathfrak{P}_R$ and
$\Tilde\mu^i_N = \mu^i_N\circ \mathfrak{P}^{-1}$.
By Lemma~\ref{L5.1}, for each $\eps>0$,
there exists $R_{1}>0$  such that
\begin{equation}\label{E5.31}
\begin{aligned}
\sup_{(x,u)\in B_R\times\Act}\;\bigl|\Hat{r}^{N}_{\bm\mu}(x,u)
- \Hat{r}^{N}_{\Tilde{\bm\mu}}(x,u)\bigr| &\; \le \; \eps\,,\\
\sup_{(x,u)\in B_R\times\Act}\;\bigl|\Breve{r}^{i,N}_{\bm{\mu}}(x,u)
- \Breve{r}^{i,N}_{\Tilde{\bm\mu}}(x,u)\bigr| &\; \le \; \eps\,.
\end{aligned}
\end{equation}

It is also easy to see that for any
$\{y^j\}_{j\ge 1}\subset \Bar{B}_{R_{1}}(0)$ we have
\begin{align*}
\eD_p\biggl(\frac{1}{N-1}\sum_{j\neq i}\del_{y^j},
\frac{1}{N}\sum_{j=1}^N\del_{y^j}\biggr)
&\;\le\; \biggl[\eD_{1}\biggl(\frac{1}{N-1}\sum_{j\neq i}\del_{y^j},
\frac{1}{N}\sum_{j=1}^N\del_{y^j}\biggr)\biggr]^{\frac{1}{\bar{q}}}\,
(2R)^{\nicefrac{\bar{q}-1}{\bar{q}}}\\
&\;\le\; \frac{4R}{N^{\nicefrac{1}{\bar{q}}}},,
\end{align*}
which gives, by \eqref{E5.16}, that
\begin{align}\label{E5.32}
\sup_{x\in B_R, u\in\Act}\;&\bigl| \Breve{r}^{i,N}_{\Tilde{\bm\mu}}(x,u)
-\Hat{r}^{N}_{\Tilde{\bm\mu}}(x,u)\bigr|\nonumber
\\[2mm]
&\qquad=\sup_{x\in B_R, u\in\Act}\;\biggl| \int_{\R^{Nd}}
\biggl[r\biggl(x,u, \frac{1}{N-1}\sum_{j\neq i}\del_{y^j}\biggr)
- r\biggl(x,u, \frac{1}{N}\sum_{j=1}^N\del_{y^j}\biggr)\biggr]\,
\prod_{j=1}^N \Tilde{\mu}^j(\D{y^j})\biggr|\nonumber
\\[5pt]
&\qquad\le \frac{\kappa_{1}}{N^{\nicefrac{1}{\bar{q}}}}
\end{align}
for some constant $\kappa_{1}$, which depends on $R_{1}$ but not on $N$.
Thus combining \eqref{E5.32} with \eqref{E5.30} and \eqref{E5.31}
we obtain 
$\sup_{i, j}\abs{\Tilde{\varrho}^i_N-\lambda_N}\to 0$ as $N\to\infty$. 
An argument similar to \eqref{E5.25} and \eqref{E5.28} also gives
\begin{equation}\label{E5.33}
\begin{gathered}
W_N(x) = \inf_{v\in\Ussm}\Exp_x^v\biggl[\int_0^{\tc_r} 
\big(\Hat{r}^{N}_{\bm\mu}(X_t, v(X_t))-\lambda_N\big)\,\D{t}
+ W_N(X_{\tc_r})\biggr], 
\quad r>0\,.
\\[2mm]
\norm{W_N}_{\Sob^{2, p}(B_R(0))}  \le \varpi_{1},\quad p\in[1, \infty),
\end{gathered}
\end{equation}
for some constant $\varpi_{1}$ independent of $N$.
Therefore, by \eqref{E5.25}, \eqref{E5.28} and \eqref{E5.33},
for every $\eps>0$ we can find $r>0$ small enough such that
\begin{align}\label{E5.34}
|V^i_n(x) - W_n(x)|\;\le \sup_{v\in\Ussm} \Exp_x^v\biggl|\int_0^{\tc_r} 
\Bigl(\Breve{r}^{i,N}_{\bm\mu}(X_t, v(X_t))
-\Hat{r}^N_{\bm\mu}(X_t, v(X_t))+\lambda_N-\Tilde{\varrho}^i_N\Bigr)\, 
\D{t}\biggr| + \eps\,.
\end{align}
Equations \eqref{E5.31}--\eqref{E5.32} imply that 
$|\Breve{r}^{i,N}_{\bm\mu}(x,u)-\Hat{r}^N_{\bm\mu}(x,u)|\to 0$ 
as $N\to\infty$ uniformly on compact subsets of $\Rd\times\Act$. 
Since $g^0\in\sorder(\abs{x}^q)$, using \eqref{E5.24} 
(together with Fatou's lemma), we obtain
$$\limsup_{R\to\infty}\;\sup_{v\in\Ussm}\Exp^v_x\biggl[\int_0^{\tc_r}
\Ind_{B^c_R(0)}(X_t) g^0(X_t)\, \D{t}\biggr]\; =\; 0$$
uniformly in $x$ belonging to compact subsets of $\Rd$. 
It follows by \eqref{E5.34} that
$$\sup_{i}\;\norm{V^i_N-W_N}_{L^\infty(B_R)}\;\xrightarrow[N\to\infty]{}\; 0\,.$$
As earlier, using \eqref{E5.33} we can show that $\{W_N\}$ is bounded in 
$\Sobl^{2, p}(\Rd)$, $p\ge 1$. 
Thus $\{V^i_N-W_N\}$ is bounded in $\Sobl^{2, p}(\Rd)$, $p\ge 1$, which implies
the
local compactness of $\{V^i_N-V^j_N\}$ in $\Sobl^{2, p}(\Rd)$, $p\ge 1$, and
clearly $V^i_N-V^j_N\to0$ in $\Sobl^{1, \infty}(\Rd)$ as $N\to\infty$.
Therefore from 
\eqref{E5.17} and \eqref{E5.29} we obtain that
$$\frac{1}{2}\trace \bigr(a(x) \grad^2(V^i_N-W_N)(x)\bigr)\;=\; f^i_N(x)\,,$$
where $f^i_N\to 0$ in $L^\infty_{\textrm{loc}}(\Rd)$
uniformly in $1\leq i\leq N$. Thus using
standard results of elliptic pde we get that $\{V^i_N-W_N\}$ converges to $0$
in $\Sobl^{2, p}(\Rd)$, $p\ge 1$,
uniformly in $1\leq i\leq N$.
Hence $\{V^i_N-V^j_N\}$ converges to $0$ in $\Sobl^{2,p}(\Rd)$, $p\ge 1$,
uniformly in $1\leq i, j\leq N$.

Next we show that $\sup_{i, j}\eD_p(\mu^i_N, \mu^j_N)\to 0$ as $N\to\infty$. 
Due to Proposition~\ref{P-villani} and \eqref{E5.22} it is enough to show that
$\sup_{i, j}\;\dd(\mu^i_N, \mu^j_N)\to 0$ as $N\to\infty$. 
Let $w_N$ be the continuous selector from the minimizer of \eqref{E5.29},
and $\mu_{w_N}$ be the corresponding invariant
measure. We show that
\begin{equation}\label{E5.36}
\sup_{i\in\{1,\dotsc, N\}} \dd(\mu^i_N, \mu_{w_N})\;\xrightarrow[N\to\infty]{}\; 0\,. 
\end{equation}
By Assumption~\ref{A5.3}
it follows that
$u\mapsto L^u V^i_N(x) + \Breve{r}^{i,N}_{\bm\mu}(x,u)$ 
is a strictly convex function. Therefore there exists a unique continuous measurable
selector $v^i_N:\Rd\to \Act$ from the minimizer in \eqref{E5.17}. 
Also, $\mu^i_N$ is the unique invariant probability measure 
corresponding to $v^i_N$ by \eqref{E5.18}. We claim that 
\begin{equation}\label{E5.37}
(x,u)\;\mapsto\; \Breve{r}^{i,N}_{\bm\mu}(x,u) \quad 
\text{is equicontinuous on compact subsets of }\; \Rd\times\Act\,.
\end{equation}
To show \eqref{E5.37} we use continuity of $r$ on
$\Rd\times\Act\times\calP_{\bar{q}}(\Rd)$. 
We consider the set $C\times\Act$ where $C$ is a compact subset of $\Rd$. 
Let $\eps>0$ be given. Then using Lemma~\ref{L5.1} we can find $R>0$ and 
the projected measures $\Tilde{\mu}^i_N$ such that
\begin{align}\label{E5.38}
\sup_{(x,u)\in C\times\Act}\;
\bigl|\Hat{r}^{N}_{{\bm\mu}}(x,u)
- \Hat{r}^{N}_{\Tilde{\bm\mu}}(x,u)\bigr|\; \le\; \frac{\eps}{4}
\qquad \forall\,  N\ge 1\,.
\end{align}
Since $\calP(\Bar{B}_R(0))$ is a compact set, using the continuity of $r$,
we can find $\delta>0$ such that
\begin{equation}\label{E5.39}
|r(x,u, \mu)-r(\Bar{x}, \Bar{u}, \mu)|\; \le\; \frac{\varepsilon}{4}\,, \quad 
\text{whenever} \; \abs{x-\Bar{x}} + d_\Act(u, \Bar{u})\le \delta, \; 
\text{and}\; x, \Bar{x}\in C\,.
\end{equation}
Thus using \eqref{E5.39} we obtain 
\begin{align*}
\bigl| \Hat{r}^{N}_{\Tilde{\bm\mu}}(x,u)-
\Hat{r}^{N}_{\Tilde{\bm\mu}}(\Bar{x},\Bar{u})\bigr|
&\;=\; \biggl| \int_{\Bar{B}_{R}^{Nd}}
r\biggl(x,u,\frac{1}{N}\sum_{j=1}^N\delta_{y^j}\biggr)\,
\prod_{j=1}^n \Tilde{\mu}^j_N(\D{y^j})\\[3pt]
&\mspace{200mu} -\int_{\Bar{B}_{R}^{Nd}}
r\biggl(\Bar{x}, \Bar{u}, \frac{1}{N}\sum_{j=1}^N\delta_{y^j}\biggr)\,
\prod_{j=1}^n \Tilde{\mu}^j_N(\D{y^j})\biggr|\\
&\;\le\; \eps/4\,,
\end{align*}
whenever $\abs{x-\Bar{x}} + d_\Act(u, \Bar{u})\le \delta$  and  
$x, \Bar{x}\in C$. Combining this with \eqref{E5.38} we establish \eqref{E5.37}. 
This also
shows that 
$$(x,u)\;\mapsto\; \Hat{r}^{N}_{\bm\mu}(x,u) \quad 
\text{is equicontinuous on compact subsets of }\; \Rd\times\Act\,.$$
Suppose that \eqref{E5.36} is not true. 
Then for $\eps>0$ we can find a subsequence $N_k$ and $i_k\in{1,\dotsc, N_k}$ 
such that
\begin{equation}\label{E5.40}
\dd(\mu^{i_k}_{N_k}, \mu_{w_{N_k}})\;\ge\; \eps \; >\; 0 \quad \text{for all}\; N_k\,.
\end{equation}
We can further chose a subsequence of $\{N_k, i_k\}$, relabel it with the 
same indices, such that the following hold:
\begin{equation}\label{E5.41}
\begin{gathered}
V^{i_k}_{N_k}\to V, \; \text{in}\; \Sobl^{2, p}(\Rd), \
\quad W_{N_k}\to V, \; \text{in}\; \Sobl^{2, p}(\Rd), \; \text{as}\; N_k\to\infty,
\\
\Breve{r}^{i_kN_k}_{\bm\mu}\to \vartheta, \; \text{in}\; \calCl(\Rd), 
\quad \Hat{r}^{N_k}_{\bm\mu}\to \vartheta, \; 
\text{in}\; \calCl(\Rd), \; \text{as}\; N_k\to\infty\,,
\\
\Tilde{\varrho}^{i_k}_{N_k}\to \varrho, \quad \lambda_{N_k}\to \varrho, \; 
\text{as}\; N_k\to\infty\,,
\\
v^{i_k}_{N_k}\to v, \; \text{in}\; \Ussm, \quad w_{N_k}\to w, \; 
\text{in}\; \Ussm, \; \text{as}\; N_k\to\infty\,.
\end{gathered}
\end{equation}
The convergence in the first and third lines are justified by the compactness 
property and the uniqueness of the limit which we established earlier. 
For the
second line we use the Arzel\`{a}--Ascoli theorem and \eqref{E5.31}--\eqref{E5.32}, 
while the fourth line is a consequence of the compactness property of $\Usm$
\cite[Section~2.4]{ari-bor-ghosh}. 
We first show that $v=w$. From \eqref{E5.17}, \eqref{E5.29} and \eqref{E5.41} we
obtain
\begin{equation}\label{E5.43}
\min_{u\in\Act}\;\bigl[L^u V(x) + \vartheta(x,u)\bigr]
\; =\; \varrho, \quad V(0)\;=\;0, \quad V\;\in\;\sorder(\calV)\,.
\end{equation}
Using Assumption~\ref{A5.3}\,(iii) it is also easy to see that 
$v^{i_k}_{N_k}\to v$ and $w_{N_k}\to w$ pointwise, as $N_k\to\infty$ and 
$$\min_{u\in\Act}\;\bigl[L^u V(x) + \vartheta(x,u)\bigr]
\; =\; L^v V(x) + \vartheta(x,v(x)) \;=\; L^w V(x) + \vartheta(x, w(x))\,.$$
Thus using the strict convexity of the Hamiltonian we obtain
$v(x) = w(x)$ for all $x$. 
By \cite[Lemma~3.2.6]{ari-bor-ghosh}, there exists $\mu\in\eH$,
corresponding to $v$, such that
$$\dd(\mu^{i_k}_{N_k}, \mu) + \dd(\mu_{w_{N_k}}, \mu)\;
\xrightarrow[N_k\to\infty]{}\; 0\,.$$
But this contradicts \eqref{E5.40} and thus \eqref{E5.36} holds.

Next we prove part (c). 
In view of \eqref{E5.43} we only need to show that $\vartheta(x,u)= r(x,u, \mu)$ where 
$\mu$ is the invariant probability measure corresponding to the minimizing selector
$v$.
Without loss of generality, we assume that 
$\Hat{r}^N_{\bm\mu}(x,u)\to \varpi(x,u)$ as $N\to\infty$. 
Fix $(x,u)\in\Rd\times\Act$. Then 
$\nu\mapsto r(x,u, \nu)$
is a continuous map. From part (b) we also have
$\sup_{1\le j\le N}\,\dd(\mu^j_N, \mu)\to 0$ 
as $N\to\infty$. 
Let $\Tilde{\nu}_R\df\nu\circ\mathfrak{P}^{-1}_{R}$. 
Then it easy to see
that 
$$\Tilde{\nu}_R \;=\; \nu|_{\Bar{B}_R(0)} + \nu(B_R^c(0))\, \del_{0}\,.$$
Since $\mu^i_N\to \mu$ in total variation 
(by \cite[Lemma~3.2.5]{ari-bor-ghosh}),  we deduce that 
$(\Tilde{\mu}_R)^i_N\to \Tilde\mu_R$ in total variation as well. 
Therefore using \eqref{E5.16.5} and mimicking the arguments in 
\cite[pp.~530]{feleqi} we can show that 
$$\biggl| \int_{\R^{Nd}} r\biggl(x,u, \frac{1}{N}\sum_{i=1}^N \delta_{y^j}\biggr)\,
 \prod_{j=1}^N (\Tilde\mu_R)^j_N(\D{y^j})
 - \int_{\R^{Nd}} r\biggl(x,u, \frac{1}{N}\sum_{i=1}^N \delta_{y^j}\biggr)\,
 \prod_{j=1}^N \Tilde\mu_R(\D{y^j})\biggr|\;\xrightarrow[N\to\infty]{}\; 0\,.$$
On the other hand (see \cite{cardaliaguet, hewitt-savage})
$$\int_{\R^{Nd}} r\biggl(x,u, \frac{1}{N}\sum_{i=1}^N \delta_{y^j}\biggr)\,
\prod_{j=1}^N \Tilde\mu_R(\D{y^j})\;\xrightarrow[N\to\infty]{}\;
r(x,u, \Tilde\mu_R)\,.$$
To complete the proof we use Lemma~\ref{L5.1} and the fact that 
$r(x,u, \Tilde\mu_R)\to r(x,u, \mu)$ as $R\to\infty$.
\end{proof}

\section*{Acknowledgements}
The work of Ari Arapostathis was supported in part by the Office of Naval
Research through grant N00014-14-1-0196, and in part by a grant from
the POSTECH Academy-Industry Foundation.
This research of Anup Biswas was supported in part by a INSPIRE faculty fellowship.

\def\cprime{$'$}

\end{document}